\documentclass[11pt]{amsart}

\usepackage{amsmath, amssymb, amsfonts, amsthm, xcolor,tikz}

\usepackage{scalerel}
\usepackage{comment}
\definecolor{my-link}{rgb}{0.5,0.0,0.0}
\definecolor{my-blue}{rgb}{0.0,0.0,0.6}
\definecolor{my-red}{rgb}{0.5,0.0,0.0}
\definecolor{my-green}{rgb}{0.2,0.5,0.2}
\definecolor{darkgreen}{rgb}{0.0,0.5,0.0}
\definecolor{darkblue}{rgb}{0.0,0.0,0.3}
\definecolor{light-gray}{gray}{0.7}
\RequirePackage[colorlinks,urlcolor=my-blue,linkcolor=my-red,citecolor=my-green]{hyperref}
\usepackage[numbers,comma,square,sort&compress]{natbib}

\definecolor{Lp}{rgb}{0.2,0.5,0.2}
\definecolor{Rp}{rgb}{0.0,0.0,0.6}
\definecolor{Lm}{rgb}{1,0,1}
\definecolor{Rm}{rgb}{1,0.5,0}
\definecolor{Inst}{rgb}{0.65,0.1,0.1}
\definecolor{pt}{gray}{0.7}

\usepackage{enumitem}

\numberwithin{figure}{section}
\definecolor{nicosred}{rgb}{0.65,0.1,0.1}

\usepackage{datetime}

\usepackage{constants}

\newconstantfamily{Om}{symbol=\Omega}  
\newcommand{\Omnew}[1]{\ensuremath{\Cl[Om]{#1}}}  
\newcommand{\Omref}[1]{\ensuremath{\Cr{#1}}}  

\newconstantfamily{Omp}{symbol=\Omega'\!}  
\newcommand{\Omnewp}[1]{\ensuremath{\Cl[Omp]{#1}}}  

\newcommand{\dual}{^\star}
\newcommand{\buse}[5]{\mathcal{B}^{\theta#1}(#2,#3,#4, #5)} 
\newcommand{\bus}[2]{\mathcal{B}^{#1#2}} 

\newcommand{\geo}{\Gamma}
\newcommand{\from}[1]{_{#1}}
\newcommand{\dir}[3]{^{#1, #2#3}}
\newcommand{\baddir}{\Theta^\w}
\newcommand{\tht}{\theta}
\newcommand{\sig}{{\scaleobj{0.8}{\square}}} 
\newcommand{\sigg}{{\scaleobj{0.9}{\square}}} 
\newcommand{\IG}[1]{{\mathcal S}^{#1}} 

\newcommand{\half}{\frac{1}{2}}
\newcommand{\R}{\mathbb{R}}
\newcommand{\Z}{\mathbb{Z}}
\newcommand{\N}{\mathbb{N}}
\newcommand{\Q}{\mathbb{Q}}
\renewcommand{\P}{\mathbb{P}}
\newcommand{\E}{\mathbb{E}}

\newcommand{\ep}{\varepsilon}
\newcommand{\e}{\ep}
\newcommand{\bfx}{{\mathbf x}}
\newcommand{\bfy}{{\mathbf y}}
\newcommand{\bfz}{{\mathbf z}}

\newcommand{\kS}{\mathfrak{S}}

\newcommand{\CI}{{\mathrm{CI}}}
\newcommand{\NU}{{\mathrm{NU}}}

\newcommand{\w}{\omega}

\font \mymathbb = bbold10 at 11pt
\newcommand{\one}{\mbox{\mymathbb{1}}}

\newcommand\abullet{\hspace{0.6pt}{\raisebox{1pt}{\scaleobj{0.6}{\bullet}}}\hspace{0.8pt}}  

\def\wt{\widetilde}

\newtheorem{thm}{Theorem}[section]
\newtheorem{lem}[thm]{Lemma}
\newtheorem{prop}[thm]{Proposition}
\newtheorem{cor}[thm]{Corollary}

\theoremstyle{definition}
\newtheorem{defn}[thm]{Definition}

\theoremstyle{remark}
\newtheorem{rmk}[thm]{Remark}

\numberwithin{figure}{section}
\numberwithin{equation}{section}

\usepackage[margin=1in]{geometry}

\title[Shocks and instability in BLPP]{Shocks and instability\\ in Brownian last-passage percolation}

\author[F.~Rassoul-Agha]{Firas Rassoul-Agha}
\address{Firas Rassoul-Agha\\ University of Utah\\  Mathematics Department\\ 155S 1400E\\   Salt Lake City, UT 84112\\ USA.}
\email{firas@math.utah.edu}
\urladdr{http://www.math.utah.edu/~firas}
\thanks{F.\ Rassoul-Agha was partially supported by National Science Foundation grants DMS-1811090 and DMS-2054630 and MPS-Simons Fellowship grant 823136}

\author[M.~Sweeney]{Mikhail Sweeney}
\address{Mikhail Sweeney\\ University of Utah\\  Mathematics Department\\ 155S 1400E\\   Salt Lake City, UT 84112\\ USA.}
\email{sweeney@math.utah.edu}
\thanks{M.\ Sweeney was partially supported by National Science Foundation grant DMS-2054630}

\thanks{Data sharing is not applicable to this article as no datasets were generated or analyzed during the current study.}

\thanks{The authors have no conflicts of interest to declare that are relevant to the content of this article.}


\date{\usdate\today}

\begin{document}

\begin{abstract}
For stochastic Hamilton-Jacobi (SHJ) equations, instability points are the space-time locations where two eternal solutions with the same asymptotic velocity differ. Another crucial structure in such equations is shocks, which are the space-time locations where the velocity field is discontinuous. In this work, we provide a detailed analysis of the structure and relationships between shocks, instability, and competition interfaces in the Brownian last-passage percolation model, which serves as a prototype of a semi-discrete inviscid stochastic HJ equation in one space dimension. Among our findings, we show that the shock trees of the two unstable eternal solutions differ within the instability region and align outside of it. Furthermore, we demonstrate that one can reconstruct a skeleton of the instability region from these two shock trees.
\end{abstract}

\maketitle

\setcounter{tocdepth}{1}
\tableofcontents

\section{Introduction}\label{sec:intro}
In this paper we examine instability, shocks, and competition interfaces in 
Brownian last-passage percolation (BLPP), a semi-discrete one-dimensional growth model that shares essential structural features with inviscid stochastic Hamilton-Jacobi equations.  The BLPP model is introduced briefly in Section \ref{sec:BLPP} and then in greater detail in Section~\ref{sec:setting}.  We begin the introduction, however, by 
placing BLPP within the broader context of stochastic Hamilton-Jacobi equations and the Kardar-Parisi-Zhang (KPZ) universality class, as the definitions, methods, and many of the structural conclusions developed here should extend naturally to these more general settings.

\subsection{Stochastic Hamilton-Jacobi equations}

Stochastic Hamilton-Jacobi (SHJ) equations are a broad class of randomly forced Hamilton-Jacobi equations of the form
\[\partial_t \Phi+ H(\nabla \Phi) = \nu \Delta\Phi - F,\]
where $\Phi = \Phi(x,t):\R^d\times \R\to\R $ is a scalar function, $\nabla\Phi$ is its spatial gradient representing momentum, and the Hamiltonian $H:\R^d\to\R$ is assumed to be a convex function. The term $F=F_\omega(x,t)$ represents the random external force potential, with the random forcing given by its spatial gradient $f = -\nabla F$. 
The viscosity parameter $\nu\ge0$ dictates the level of diffusion, with $\nu>0$ representing the viscous case and $\nu=0$ corresponding to the inviscid case.

Given an appropriate initial condition $\varphi$ at the initial time $s$, the solution to the viscous SHJ equation is unique and smooth. In contrast, the inviscid SHJ equation admits multiple solutions with discontinuities in the spatial gradient, known as \emph{shocks}. Shocks play a crucial role in Hamilton-Jacobi equations and are fundamental to understanding phenomena such as sonic booms, breaking waves, traffic jams, and discontinuities in spacetime curvature, all of which are manifestations of nonlinear shock wave phenomena.

The solution of interest to us is the so-called \emph{viscosity solution}, which is obtained by solving the equation with $\nu>0$ and then taking the limit as $\nu\to0$. This solution is expressed by the Hopf-Lax-Oleinik variational representation. Specifically, define the random Hamiltonian $H_\omega(p,x,t) =H(p)+F_\omega(x,t)$ and its Legendre transform, the random Lagrangian $L_\omega(v,x,t)= \sup_p \{ p\cdot v - H_\omega(p,x,t) \}$. Then for an initial condition $\varphi:\R^d\to\R$ at time $s$, the viscosity solution at time $t>s$ is given by 
\begin{align}\label{LO}
\Phi(x,t) = \inf_{\gamma:\gamma(t)=x} \Bigl\{ \varphi(\gamma(s))+\int_{s}^{t} L_\omega(\gamma'(r), \gamma(r), r)\,dr \Bigr\},
\end{align}
where the infimum is taken over all absolutely continuous paths $\gamma:[s,t]\to\R$ with $\gamma(t)=x$. Paths that achieve the infimum are known as \emph{characteristic lines} or \emph{geodesics}, borrowing terminology from first-passage percolation models. Thus, the equation is solved by picking up an initial value $\varphi(\gamma(s))$ and then collecting the Lagrangian action along the characteristic line. Shock points $(t,x)$ are those from which emanate multiple characteristic lines. A similar variational formula exists in the viscous case; see (2) in \cite{Bak-Kha-18}. In this case, the paths $\gamma$ are referred to as \emph{random polymers}. 

The representation \eqref{LO} shows that the SHJ equation is a Markov process, making it suitable for analysis with tools from random dynamical systems (RDS). For an overview of RDS, refer to \cite{Arn-98}. A fundamental concept in RDS is the stability or reliability of the system. In a seminal paper, Sinai \cite{Sin-91} studied the viscous SHJ equation with a forcing term that is either regular and periodic in both space and time or periodic and regular in space and white in time. He demonstrated that solutions to the equation, starting from different initial conditions, can be coupled to a process defined for all times. This process acts as a \emph{pullback attractor}, as defined in Definition 9.3.1 of \cite{Arn-98}.

The existence of a unique, globally defined stochastic process (called an \emph{eternal solution}) that is measurable with respect to the history of the noise is commonly referred to as the one force-one solution (1F1S) principle (see, for example, the introduction of \cite{E-etal-00}). The authors of \cite{Bak-Kha-10} incorporate the pullback attractor property into the definition of the 1F1S principle. When the 1F1S principle holds, the system is considered stable or reliable, and \emph{stochastic synchronization} occurs: solutions starting from different initial conditions (within the basin of attraction of the pullback attractor) converge over time.

In this paper, we are interested in the one-dimensional case $d=1$. In this case, the asymptotic velocities $\lim_{x\to\infty}x^{-1}\Phi(x,t)$ and $\lim_{x\to-\infty}x^{-1}\Phi(x,t)$ are conserved by the dynamics of the Hamilton-Jacobi equation and the 1F1S principle is discussed for given values of the conserved quantity. 
The ergodic theory of the SHJ calls for the 1F1S principle to hold for each fixed velocity value, on a full-probability event dependent on this value.
However, there may be exceptional values of the conserved quantity for which the 1F1S principle fails to hold.  In such instances, multiple eternal solutions with the same asymptotic velocity exist, indicating that the random dynamical system is unstable. The points of instability are those points $(x,t)$ at which the eternal solutions differ. 

The instability phenomenon described above occurs in both the viscous and inviscid cases, whereas shocks are unique to inviscid Hamilton-Jacobi equations. Our interest lies in understanding the connection between the structure of shocks and the structure of instability. Consequently, this paper focuses on inviscid SHJ equations.

Before describing the BLPP model considered in this paper, we provide a list of works where the aforementioned phenomena have been rigorously demonstrated.

The 1F1S principle for deterministic values of the asymptotic velocity has been rigorously established for several models. Initially, it was demonstrated in settings where space is compact (or essentially compact) and the random forcing is fairly regular \cite{Bak-13,Bak-Kha-10,Dir-Sou-05,E-etal-00,Gom-etal-05,Hoa-Kha-03,Kif-97,Sin-93}. Recently, the principle was proven for the one-dimensional KPZ equation (see Section \ref{sec:KPZ}) with space-time white noise on the torus \cite{Ros-22}. In \cite{Bak-16,Bak-Cat-Kha-14,Bak-Li-19,Dri-etal-23}, the necessity of spatial compactness was eliminated, but the random forcing was essentially discrete (e.g., a Poisson point process) or semi-discrete (a kick-type forcing activated at integer times and smooth in space). These works did not consider instability for exceptional values of the conserved quantity.

The 1F1S principle was shown to hold for deterministic velocities and to fail for a random countable dense set of exceptional velocities in two discretizations of SHJ equations \cite{Sep-Sor-23-pmp,Jan-Ras-Sep-23} (see Sections \ref{sec:LPP} and \ref{sec:BLPP}). In the fully continuous space-time setting, this was proved for the KPZ equation \cite{Jan-Ras-Sep-23-1F1S-} and its inviscid counterpart, the KPZ fixed point \cite{Bus-Sep-Sor-24}. For more details, see Section \ref{sec:KPZ}.

\begin{rmk}
    Although the 1F1S principle results are not explicitly stated in \cite{Sep-Sor-23-pmp,Jan-Ras-Sep-23,Bus-Sep-Sor-24}, they can be derived from the directional semi-infinite geodesics coalescence results presented in these papers. This connection is perhaps not surprising since these geodesics are characteristics from time $-\infty$ that can be used to define eternal solutions. This path to proving the 1F1S principle is demonstrated in \cite{Jan-Ras-20-aop} for a viscous model and in \cite{Bak-16} for an inviscid model.
\end{rmk}


\subsection{The KPZ equation and KPZ universality}\label{sec:KPZ}
A particular SHJ of note is the famed KPZ equation
\begin{align}\label{KPZ}
\partial_t h -\half (\partial_x h)^2 = \nu\partial_{xx}h + \beta W,
\end{align}
obtained by setting $d=1$, $\nu>0$, $H(p) = \frac{p^2}2$, $F=-\beta W$, where $W$ is a space-time white noise and $\beta$ a real parameter modulating the strength of the noise, and then defining $h=-\Phi$. 
This viscous SHJ equation was introduced by Kardar, Parisi, and Zhang \cite{Kar-Par-Zha-86} to describe the evolving height interface $h(x,t)$ of a growing random surface. 
 It has been argued in the physics literature that the KPZ equation represents a broad \emph{universality class} of physical systems and mathematical models that exhibit similar long-term behavior. This behavior is characterized by a fluctuations scaling exponent of $1/3$ a correlation scaling exponent of $2/3$, which coincides with the polymer paths' fluctuation scaling exponent. Additionally, these systems exhibit Tracy-Widom asymptotic long-time scaling limit probability distributions. The KPZ universality class is believed to contain viscous and inviscid SHJ equations, interacting particle systems, percolation/growth models, paths/polymers in random environments, driven diffusive systems, and random matrices. 

Taking a formal spatial derivative of the KPZ equation says that the velocity $u=-\partial_x h$ satisfies the stochatic viscous Burger's equation:
\begin{align}\label{Burgers}
\partial_t u + u\partial_x u = \nu\partial_{xx} u  - \beta \partial_x W.
\end{align}
The well-posedness of an appropriately renormalized version of the KPZ equation \eqref{KPZ} was shown only recently, first on the torus \cite{Gub-Imk-Per-15,Gub-Per-17,Hai-13,Hai-14} and then on
the real line \cite{Per-Ros-19}. For the case of the stochastic Burgers equation \eqref{Burgers} see \cite{Ber-Can-Jon-94,Gon-Jar-10-,Gub-Per-17,Gub-Per-18,Gub-Per-18-review,Gub-Per-20}. 

As part of a broader study of randomly forced models in fluid dynamics, 
\cite{For-Nel-Ste-77} performed a dynamical renormalization group analysis of the stochastic Burgers equation \eqref{Burgers} and predicted the existence of the scaling limit of $\e^{1/2}h(\e^{-1} x,\e^{-3/2}t)-C_\e t$, as $\e\to0$, where $h$ solves \eqref{Burgers} with $\nu=\frac12$ and $\beta=1$ and $C_\e$ is an appropriate constant. 
This prediction was recently proved in \cite{Qua-Sar-23} for the KPZ equation with appropriate initial conditions, including both continuous initial data and the fundamental solution, corresponding to a Dirac $\delta$ measure initial condition. The limit, originally constructed in \cite{Mat-Qua-Rem-21} and termed the \emph{KPZ fixed point}, is believed to be the universal attractor for the entire KPZ universality class under the same scaling. This was proved for a number of models in \cite{Qua-Sar-23,Vir-20-,Dau-Vir-21-}. By rescaling the KPZ equation (see Remark 1.1 in \cite{Jan-Ras-Sep-23-1F1S-}), this statement implies that the solution of equation \eqref{KPZ} with $\beta=\sqrt{2\nu}$ 
converges to the KPZ fixed point, as $\nu\to0$. 
The KPZ fixed point can thus be viewed as an inviscid version of the KPZ equation.


We believe that the methods developed in this paper can be transferred to other models in the KPZ universality class.  In particular, 
as mentioned in the previous section, stability and instability for the KPZ equation and the KPZ fixed point have been studied in 
\cite{Jan-Ras-Sep-23-1F1S-} and \cite{Bus-Sep-Sor-24}, respectively.  
Applying the methods in this paper, in the forthcoming work \cite{Ras-Swe-24-b-}, we further describe the instability graph for the KPZ fixed point and investigate its connections to shocks.

\subsection{Last-passage percolation}\label{sec:LPP} 
Analogous to the way Brownian motion can be discretized by a random walk, \eqref{LO} can be discretized by taking $s, t, x$ to be integers with $s<t$ and replacing the continuous paths $\gamma:[s,t] \to \R$ with nearest-neighbor paths $(\gamma_k)_{k=s}^t$. After rotating the space-time lattice, this results in the directed last-passage percolation (LPP) model on the cubic lattice $\Z^{d+1}$.

In the directed LPP model on $\Z^{d+1}$, random weights $\{\w_x:x\in\Z^{d+1}\}$ are placed at the lattice sites according to a shift-invariant ergodic (usually product) probability measure $\P$. 
The admissible paths between vertices are those with increments in $\{e_1,\dotsc,e_{d+1}\}$. The weight of an admissible path $\gamma$ is the sum of the weights along the path, $\sum_{z\in\gamma}\w_z$. For $x,y\in\Z^{d+1}$ with $x\le y$ coordinate-wise, the point-to-point last-passage time $G_{x,y}$ from $x$ to $y$ is the maximum path weight
$G_{x,y} = \max_{\gamma} \sum_{z\in\gamma}\w_z$
over admissible paths from $x$ to $y$. For $x\in\Z^{d+1}$ and $n\in\Z$ with $n>x\cdot(e_1+\cdots+e_{d+1})=\ell$ and a function $\varphi:\Z^{d+1}\to\R$, the point-to-level last-passage time is 
\begin{align}\label{Gp2l}
G_{x,(n)}^\varphi = \max_{\gamma}\Bigl\{\varphi(\gamma_n)+\sum_{z\in\gamma}\w_z\Bigr\},
\end{align}
where the maximum is taken over all admissible paths $(\gamma_i)_{i=\ell}^n$ with $\gamma_\ell=x$ and maximizing paths are called geodesics.

The planar directed last-passage percolation (LPP) model on $\Z^2$  is a pivotal model in probability theory, situated at the intersection of multiple disciplines, including queueing theory, interacting particle systems, integrable systems, and representation theory.

Comparing the formula \eqref{Gp2l} to \eqref{LO}, one can see how the directed LPP model discretizes a stochastic inviscid Hamilton-Jacobi equation. In this context, the time coordinate of a vertex $y\in\Z^{d+1}$ is represented by $t=-y\cdot(e_1+\cdots+e_{d+1})$ and, therefore, eternal solutions correspond to taking $n\to\infty$ in \eqref{Gp2l}. This, however, will result in $G_{x,(n)}^\varphi$ blowing up and one instead considers the limit of the differences $G_{x,(n)}^\varphi-G_{y,(n)}^\varphi$.

By writing
\[G_{x,(n)}^\varphi=\max\bigl\{\varphi(y)+G_{x,y}:y\in\Z^{d+1},y\cdot(e_1+\cdots+e_{d+1})=n\bigr\},\]
one sees that eternal solutions can be studied by analyzing the existence or absence of the almost sure limits
\begin{align}\label{Blim}
B^\xi(x,y)=\lim_{n\to\infty}(G_{x,\lfloor n\xi\rfloor}-G_{y,\lfloor n\xi\rfloor})
\end{align}
for $x,y\in\Z^{d+1}$ and $\xi\in\R^{d+1}$ with  $\xi\cdot(e_1+\cdots+e_{d+1})=1$. These limits are related to the existence, uniqueness, and coalescence of semi-infinite directed geodesics. Therefore, $B^\xi$ is called the \emph{Busemann function}, borrowing terminology from metric geometry. 

Busemann functions were first introduced to first-passage percolation by Chuck Newman \cite{New-95} and were initially utilized to prove results about semi-infinite geodesics in the same context by Hoffman \cite{Hof-08}. Since then, they have become a fundamental tool for studying semi-infinite geodesics and polymer measures \cite{Ahl-Hof-16-,Cat-Pim-12,Cat-Pim-13,Dam-Han-14,Dam-Han-17,Fan-Sep-20,Fer-Mar-Pim-09,Fer-Pim-05,Geo-Ras-Sep-17-ptrf-1,Geo-Ras-Sep-17-ptrf-2,How-New-97,How-New-01,Lic-New-96,Jan-Ras-20-aop,Jan-Ras-20-jsp,Gro-Jan-Ras-25-,Bat-Fan-Sep-25}.

In the context of the directed LPP model, \cite{Jan-Ras-Sep-23} demonstrated that the limits in \eqref{Blim} exist for deterministic directions $\xi$, but fail to exist for an exceptional random countable dense set of these directions. For each such exceptional direction, \cite{Jan-Ras-Sep-23} also introduced and analyzed the \emph{instability graph}. This graph has vertices $u \in \Z^2$ and nearest-neighbor edges $(u, u + e_i)$, $i \in {1,2}$, characterized by the failure of the limit in \eqref{Blim} to hold for $x = u$ and $y = u + e_i$.

What the lattice LPP model lacks, however, is the presence of shocks. Due to the discrete nature of the space, there is no meaningful notion of discontinuity in the spatial gradient. To study the interaction between shocks and instability, a model with a continuous space variable is required. One such model is the Brownian last-passage percolation model, which we describe next.

\subsection{Brownian LPP and our contribution}\label{sec:BLPP} 
Brownian last-passage percolation (BLPP) is similar to the planar LPP  described earlier, but with a discrete time coordinate and a continuous space coordinate. Here, we provide a brief description of the model, with a more rigorous introduction to the model presented in Section \ref{sec:setting}.

In the Brownian last-passage percolation model (BLPP), paths are restricted to up-right staircases on $\Z\times\R$. These paths are indexed by their exit times from each integer level $m$, denoted by $s_m$. The random weights are given by i.i.d.\ standard two-sided Brownian motions $\{B_k:k\in\Z\}$.
For integers $m\le n$ and $s<t$, the weight of a path from $(m,s)$ to $(n,t)$ with jump times $s=s_{m-1}\le s_m\le\cdots\le s_n=t$ is the cumulative amount of Brownian motions it collects: 
$\sum_{k=m}^n B_k(s_{k-1}, s_k)$. The Brownian LPP time $L_{(m,s),(n,t)}$ from $(m,s)$ to $(n,t)$ is defined as the supremum of the weights of all admissible paths from $(m,s)$ to $(n,t)$.
Paths that attain this supremum are called point-to-point geodesics, as in the lattice LPP model, and semi-infinite geodesics are rooted infinite paths which are geodesics between any two of their points. 

The Busemann functions
\[\mathcal{B}^\tht(m,s,n,t)=\lim_{k\to\infty}(L_{(m,s),(k,\theta k)}-L_{(n,t),(k,\theta k)}),\]
$\tht>0$, have been defined and studied in \cite{Alb-Ras-Sim-20,Sep-Sor-23-aihp,Sep-Sor-23-pmp,Bus-Sep-Sor-24}. In particular, it is shown in \cite{Sep-Sor-23-pmp} that there is an exceptional random countable dense set of directions $\tht$ for which these limits do not exist. 

In this paper, we introduce and study the properties of the instability graph of points $(n,t)$ where these Busemann limits do not exist. These points are also where (at least) two $\tht$-directed semi-infinite geodesics originate, which separate immediately and never touch again. 

On the other hand, \cite{Sep-Sor-23-aihp} showed that, due to the continuous nature of the space variable, for each direction $\tht>0$, there exists a countable set of points $(n,t)$ from which two $\tht$-directed semi-infinite geodesics emanate. These geodesics separate immediately and then rejoin and coalesce at a later point. These are the shock points, as they are locations where two distinct characteristic lines originate, carrying the same information from the remote past. This is in contrast with the aforementioned two distinct semi-infinite geodesics emanating from instability points, which do not touch after separating, thus bringing different information from the remote past.

We demonstrate that the shock points exhibit the familiar tree structure and investigate the relationship between their locations and the instability graph. In particular, we show that a skeleton of the instability graph can be reconstructed by knowing the locations of the shock points.\smallskip

This work represents the first investigation into the relationship between shocks and instability in inviscid stochastic Hamilton-Jacobi equations. It lays the foundation for a broader analysis of the interplay between these two critical structures, providing a new framework for future research in this area.
\medskip

{\bf Organization.}
In Section \ref{sec:setting}, we introduce the Brownian last-passage percolation model and highlight its parallels with the inviscid stochastic Hamilton-Jacobi equation. Our main results are presented in Section \ref{sec:results}. Section \ref{sec:previous} compiles results from previous papers that we will need for this work.
The instability graph is defined in Section \ref{def:Bus:analytic} through the Busemann process, with an equivalent definition via semi-infinite geodesics given in Section \ref{def:Bus:geometric}. Section \ref{sec:inst} is devoted to the properties of the instability graph. Section \ref{sec:shocks} introduces shock points and examines their relationship with the instability graph. Section \ref{sec:cif} explores the relationship between the instability graph, shocks, and the origins of competition interfaces.
Section \ref{sec:aux} contains the proofs of several auxiliary results stated in Section \ref{sec:results}. Additionally, two more results concerning semi-infinite geodesics and shocks are left to the appendix, as they are not used in the rest of the paper but are still of significant interest.
Notably, the graph defined in Sections \ref{def:Bus:analytic} and \ref{def:Bus:geometric}, and studied in Sections \ref{sec:inst}-\ref{sec:cif}, is in fact more general than the one mentioned in Section \ref{sec:results}.

\medskip

{\bf Notation.} 
$\R$ denotes the set of real numbers, $\Q$ the rationals, $\Z$ the integers, $\R_+$ the nonnegative reals, and $\Z_+$ the nonnegative integers. 
We abbreviate $a\vee b=\max(a,b)$. 
$(m,s)\le(n,t)$ means $m\le n$ and $s\le t$. 
$(m,s)\preceq(n,t)$ means $m\ge n$ and $s\le t$, i.e.\ $(n,t)$ is south-east of or down-right from $(m,s)$. 

\section{The setting}\label{sec:setting}

\subsection{Brownian LPP} Let $(B_i(t), t \in \R)_{i \in \Z}$ be a sequence of independent two-sided standard Brownian motions, defined on a Polish probability space $(\Omega,\kS,\P)$. $\E$ denotes expectation under $\P$. 

An \emph{up-right path} between two points $(m,s)\le(n,t)\in\Z\times\R$ with $s<t$ is described by its jump times  $s = s_{m-1} \le s_m \le \dotsc \le s_n = t$. Alternatively, the path can be thought of as the linear interpolation of the sequence of points $\{(k,s_{k-1}),(k,s_k):m\le k\le n\}$.

Given a realization $\w\in\Omega$, space-time points $(m, s) \le (n, t)$ in $\Z\times\R$ with $s<t$, and a sequence of jump times $s = s_{m-1} \le s_m \le \dotsc \le s_n = t$, we let $\sum_{k = m}^n \bigl(B_k(s_k)-B_k(s_{k -1})\bigr)$ be the \emph{length} of the path from $(m,s)$ to $(n,t)$ defined by the jump times. 
The \emph{last-passage time} from $(m,s)$ to $(n,t)$ is defined as 
\begin{align}\label{LPP}
L_{(m, s), (n, t)}(\w) = \sup \Bigl\{ \sum_{k = m}^n\bigl(B_k(s_k)-B_k(s_{k-1})\bigr) : s = s_{m-1} \le s_m \le \dots \le s_n = t \Bigr\}.
\end{align}

An up-right path is a geodesic from $(m,s)$ to $(n,t)$ if its jump points maximize the supremum in \eqref{LPP}. 
Theorem B.1 in \cite{Ham-19} demonstrates that for any $m\le n$ in $\Z$ and $s<t$ in $\R$, there is $\P$-almost surely a unique geodesic path from $(m,s)$ to $(n,t)$. However, there may be exceptional start and end points for which there are multiple geodesic paths. Indeed, Lemma 3.5 in \cite{Dau-Ort-Vir-22} shows that for any $m\le n$ in $\Z$ and $s<t$ in $\R$ and for any realization of $\w\in\Omega$ for which the Brownian motions $B_k$, $m\le k\le n$, are continuous, there exist jump times $s = s^S_{m-1} < s^S_m < \dotsc < s^S_n = t$, $S\in\{L,R\}$, that maximize the supremum in \eqref{LPP} and such that these are, respectively, the leftmost and rightmost geodesic paths: any maximizing sequence of jump times satisfies $s^L_k\le s_k\le s^R_k$, for all integers $m\le k\le n$. 

\begin{figure}[ht!]
\begin{center}

\begin{tikzpicture}[>=latex, scale=0.6]
\begin{scope}
\draw(0,0)--(6,0);
\draw(0,1)--(6,1);
\draw(0,2)--(6,2);
\draw(0,3)--(6,3);
\draw[line width=2pt,color=Lp](1,0)--(2,0)--(2,1)--(2.5,1)--(2.5,2)--(4,2)--(4,3)--(5,3);
\end{scope}

\begin{scope}[shift={(10,0)}]
\draw(0,0)--(6,0);
\draw(0,1)--(6,1);
\draw(0,2)--(6,2);
\draw(0,3)--(6,3);
\draw[line width=2pt,color=Lp](1,0)--(2,0)--(2,1)--(2.5,1)--(2.5,3)--(4,3)--(5,3);
\end{scope}

\end{tikzpicture}
\end{center}

\caption{\small Left: An illustration of an up-right path. According to Proposition \ref{nodouble}, up-right paths do not jump twice in a row like the path in the right panel.}
\label{fig:path}
\end{figure}

A semi-infinite up-right path starting at $(m,s)\in\Z\times\R$ is determined by its jump times $(s_k)_{k\ge m-1}$ with $s_{m-1}=s$ and  $s_k\le s_{k+1}$ for all $k\ge m-1$. Such a path is called a \emph{semi-infinite geodesic} if for each integer $n>m$ $(s_{m-1},s_m,\dotsc,s_n)$ is a geodesic between $(m,s)$ and $(n,s_n)$. Similarly, a bi-infinite up-right path is determined by its jump times $(s_k)_{k\in\Z}$ and it is said to be a \emph{bi-infinite geodesic} if for any integers $m<n$, 
$(s_{m-1},s_m,\dotsc,s_n)$ is a geodesic between $(m,s_{m-1})$ and $(n,s_n)$.

We are interested in the large-scale properties of the Brownian last-passage percolation model and, in particular, the structure of the semi-infinite geodesics. 
The following limit is one of the main tools we use in our analysis.  

\begin{thm}[Theorem 4.2 in \cite{Alb-Ras-Sim-20}]\label{thm: LPPresult}
Fix $m,n\in\Z$, $s,t\in\R$, and $\tht>0$. Then with $\P$-probability one, the limit
\begin{align}\label{B:lim}
\buse{}{m}{s}{n}{t}=\lim_{k \to \infty} (L_{(m,s), (k, t_k)} - L_{(n,t), (k, t_k)})
\end{align}
exists almost surely and is independent of the choice of the sequence $\{ t_k \}$, so long as $\lim_{k \to \infty} \frac{t_k}{k} = \tht$. We call $\bus{\theta}{}$ the Busemann function with velocity $\tht$.
\end{thm}

The above limits clearly satisfy the cocycle property:
\[\buse{}{\ell}{r}{m}{s}+\buse{}{m}{s}{n}{t}=\buse{}{\ell}{r}{n}{t}\]
for all $\ell,m,n\in\Z$ and $r,s,t\in\R$.
Lemma 6.3(iii) in \cite{Sep-Sor-23-aihp} says that one also has the following monotonicity: for all $\tht'>\tht>0$ fixed, 
    \[\bus{\tht'}{}(m,s,m,t)\le\buse{}{m}{s}{m}{t}\quad\text{and}\quad
    \buse{}{m}{s}{m+1}{s}\le\bus{\tht'}{}(m,s,m+1,s)\]
hold $\P$-almost surely, simultaneously for all $m\in\Z$, and $s<t$ in $\R$. These two properties imply that on a single full $\P$-probability event, the limits
\begin{align}\label{bus.def}
\buse{-}{m}{s}{n}{t}=\lim_{\Q\ni\tht'\nearrow\tht}\bus{\tht'}{}(m,s,n,t)
\quad\text{and}\quad
\buse{+}{m}{s}{n}{t}=\lim_{\Q\ni\tht'\searrow\tht}\bus{\tht'}{}(m,s,n,t)
\end{align}
exist for all $m,n\in\Z$ and $s,t\in\R$.
This produces the \emph{Busemann process}
    \[\bigl\{\buse{\sig}{m}{s}{n}{t}:\tht>0,\sigg\in\{-,+\},m,n\in\Z,s,t\in\R\bigr\}.\]
\cite{Sep-Sor-23-aihp,Sep-Sor-23-pmp} established many properties of this process. Theorem \ref{thm:B} below summarizes the ones we need for this paper. 

The sign distinction in \eqref{bus.def} would be unnecessary if the Busemann functions $\bus{\tht}{}$ defined in \eqref{B:lim} were continuous in $\tht\in\Q$.  
However, Theorem 2.5 of \cite{Sep-Sor-23-pmp} shows that the set
\begin{align}\label{def:baddir}
\baddir=\bigl\{\tht>0:\bus{\tht}{-}\ne\bus{\tht}{+}\bigr\}
\end{align}
is almost surely a ``genuinely random'' countable dense subset of $(0,\infty)$ 
(see also Theorem \ref{baddir dense} below).  
Consequently, the distinction between the left and right limits is not only necessary but in fact essential for the structures analyzed in this paper. See Section \ref{sec:inst.intro} for more.

The above Busemann process produces a process of semi-infinite geodesics: 
    \begin{align}\label{geo-proc}
    \bigl\{\geo\from{(m,s)}\dir{S}{\tht}{\sig}:\tht>0,\sigg\in\{-,+\},S\in\{L,R\},m\in\Z,s\in\R\bigr\}.
    \end{align}
For each $\tht>0$, $\sigg\in\{-,+\}$, and $(m,s)\in\Z\times\R$, $\geo\from{(m,s)}\dir{L}{\tht}{\sig}$ and $\geo\from{(m,s)}\dir{R}{\tht}{\sig}$ are, respectively, the leftmost and the rightmost up-right paths starting at $(m,s)$ and satisfying
    \[B_n(s_n)-\buse{\sig}{n+1}{0}{n+1}{s_n}=\sup\{B_n(s)-\buse{\sig}{n+1}{0}{n+1}{s}:s\ge s_{n-1}\},\quad\text{for all }n\ge m.\]
\cite{Sep-Sor-23-aihp,Sep-Sor-23-pmp} introduced and studied the above geodesics and proved a large number of properties that these paths satisfy.
Theorem \ref{thm:geo} summarizes the properties we need in this work. 
Theorem 4.8(ii–iii) of \cite{Sep-Sor-23-pmp} states that for every $\tht>0$ and $\sigg\in\{-,+\}$ there exists a ``genuinely random'' countable set of points $(m,s)$ 
at which the geodesics 
$\geo\from{(m,s)}\dir{L}{\tht}{\sig}$ and 
$\geo\from{(m,s)}\dir{R}{\tht}{\sig}$ differ.  
Theorem \ref{thm:shocks} below summarizes the properties of these points.  
As with the sign distinction in the Busemann limits discussed above, the 
difference between leftmost and rightmost semi-infinite geodesics is both necessary and central to the structures analyzed in this paper.  
See Section \ref{sec:shocks.intro} for further discussion.

The formula \eqref{LPP} is analogous to the Hopf-Lax-Oleinik variational formula for the fundamental solutions of inviscid Hamilton-Jacobi equations, with time running downward. In this analogy, point-to-point geodesics correspond to the characteristic lines of these equations.

Similarly, the function $\buse{\sig}{0}{0}{n}{t}$ can be viewed as an \emph{eternal} solution, meaning it is defined for all times $n\in\Z$. Given that time progresses downward in this analogy, we can interpret the semi-infinite geodesics  $\geo\from{(m,s)}\dir{L}{\tht}{\sig}$ and $\geo\from{(m,s)}\dir{R}{\tht}{\sig}$ as characteristic curves of $\bus{\tht}{\sig}$, tracing back from 
$(m,s)$ to the distant past.

\subsection{Stability and instability}\label{sec:inst.intro} 
Eternal solutions and semi-infinite geodesics are related to the notions of pullback attractors and stochastic synchronization, from random dynamical systems. See Sections 3.3 and 3.5 in \cite{Jan-Ras-Sep-23-1F1S-} for these connections in the context of the Kardar-Parisi-Zhang equation. 

Recall the random set $\baddir$, defined in \eqref{def:baddir}.
When $\tht>0$ is such that $\tht\not\in\baddir$, the $\pm$ distinction disappears and $\geo\from{(m,s)}\dir{S}{\tht}{-}=\geo\from{(m,s)}\dir{S}{\tht}{+}$ for all $(m,s)\in\Z\times\R$ and $S\in\{L,R\}$.
By Theorem 4.21 in \cite{Sep-Sor-23-pmp},
we have that $\P$-almost surely, for all $\tht\not\in\baddir$, all $\tht$-directed geodesics, out of the various points $(m,s)\in\Z\times\R$, coalesce. This translates to the uniqueness of the pullback attractors (Definition 9.3.1 in \cite{Arn-98}), indicating stability, where solutions started from initial conditions in an appropriate space of functions (basin of attraction) synchronize. Conversely, according to Remark 4.22 in \cite{Sep-Sor-23-pmp}, $\P$-almost surely, for any $\tht\in\baddir$, there exist at least two non-coalescing $\tht$-directed geodesics, $\geo\from{(m,s)}\dir{L}{\tht}{-}$ and $\geo\from{(m,s)}\dir{R}{\tht}{+}$, out of each $(m,s)\in\Z\times\R$. In this case, the associated random dynamical system exhibits multiple pullback attractors, suggesting instability, where stochastic synchronization fails.
For more details and precise definitions in the related cases of the discrete last-passage percolation model and the Kardar-Parisi-Zhang (KPZ) equation, refer to Section 4.1 in \cite{Jan-Ras-Sep-23} and Section 3.5 in \cite{Jan-Ras-Sep-23-1F1S-}, respectively. This instability, when $\tht\in\baddir$, can also be interpreted as a phase transition in the familiar setting of Gibbs measures. See Section 2.4 in \cite{Jan-Ras-18-arxiv} for details in the case of directed random polymers.

With the above in mind, we are interested in analyzing the properties of the set
\[\bigl\{(n,t)\in\Z\times\R:\bus{\tht}{-}\ne\bus{\tht}{+}\text{ ``near $(n,t)$''}\bigr\}.\]
The precise definition of the actual object we study, called the \emph{instability graph}, is given in Section \ref{sec:results}, then a more general version is developed in Sections \ref{def:Bus:analytic} and \ref{def:Bus:geometric}. Properties of this graph are given in Section \ref{sec:inst}. 

\subsection{Shocks}\label{sec:shocks.intro} Viewing semi-infinite geodesics as characteristic lines suggests this definition of shock points (of $\bus{\tht}{\sig}$). 
For $\tht>0$ and $\sigg\in\{-,+\}$, let 
\[\NU_1^{\tht\sig}=\bigl\{(m,s)\in\Z\times\R:(m+1,s)\in\geo\from{(m,s)}\dir{L}{\tht}{\sig},\exists\delta>0\text{ such that }(m,s+\delta)\in\geo\from{(m,s)}\dir{R}{\tht}{\sig}\bigr\}.\]
In words, points in $\NU_1^{\tht\sig}$ are ones at which the two geodesics $\geo\from{(m,s)}\dir{L}{\tht}{\sig}$ and $\geo\from{(m,s)}\dir{R}{\tht}{\sig}$ 
 split immediately.  

\begin{defn}
    Given $\tht>0$ and $\sigg\in\{-,+\}$ a point $(m,s)\in\Z\times\R$ is called a $\tht\sigg$-\emph{shock point} if $(m,s)\in\NU_1^{\tht\sig}$. When $\tht$ and $\sigg$ are clear from the context, we just say that $(m,s)$ is a shock point.
\end{defn}

\begin{rmk}
    The above definition only considers semi-infinite geodesics from the process \eqref{geo-proc}. {\it A priori}, there may be other semi-infinite geodesics in the model and hence there may be more shock points. However, as explained in Remark \ref{rk:all geo} below, it is expected that the process \eqref{geo-proc} contains all the semi-infinite geodesics of the model and hence that there are no other shocks that may have been overlooked.
\end{rmk}

The set $\NU_1^{\tht\sig}$ was introduced in \cite{Sep-Sor-23-pmp}. Theorem \ref{thm:shocks} below summarizes the results we need from that paper.

Shock points mark places where we have multiple semi-infinite geodesics, corresponding to the same Busemann function (i.e.\ eternal solution). As per Theorem \ref{thm:shocks}\eqref{thm:shocks:a}, these points are present for all $\tht>0$ and $\sigg\in\{-,+\}$. However, according to Theorem \ref{thm:geo}\eqref{thm:geo:dir}, the geodesics $\geo\from{(m,s)}\dir{S}{\tht}{\sig}$, $S\in\{L,R\}$ and $\sigg\in\{-,+\}$, are all $\tht$-directed. Thus, when $\tht\in\baddir$ and $\bus{\tht}{-}$ and $\bus{\tht}{+}$ differ we get two eternal solutions, each producing semi-infinite geodesics traveling at the same speed $\tht$. Consequently, there might be points that are not shock points but still emit multiple $\tht$-directed geodesics. 

As highlighted in the introduction, the BLPP model has aspects of a Burgers' equation, where the system dissipates energy at shock points. It is then natural to ponder the connection between the locations of shock points and instability points. This forms the primary focus of our investigation in this paper and is presented in Section \ref{sec:shocks} and summarized in Section \ref{sec:results}. 
A pertinent question left open is whether instability points are linked to any form of energy dissipation.

\subsection{Competition interfaces}\label{sec:intro.ci} Another question of interest to us is the relationship between instability and competition interfaces.  More precisely, given $S\in\{\text{left},\text{right}\}$ and a starting point $(m,s)\in\Z\times\R$, we can split $(m,s)+\Z_+\times\R_+$ into two regions, based on whether or not the $S$-most geodesic from $(m,s)$ to the point $(n,t)\in(m,s)+\Z_+\times\R_+$ (with $(n,t)\ne(m,s)$) goes through $(m+1,s)$, i.e.\ whether the geodesic proceeds from $(m,s)$ upward or to the right. The two regions are then separated by an interface, an up-right path on $(\Z+\half)\times\R$, starting at $(m+\half,s)$, called the \emph{$S$-competition interface}. The detailed construction of the paths is done in Remark 4.24 in \cite{Sep-Sor-23-pmp}. Theorem \ref{thm:cif-SS} below summarizes the properties we need. In particular, part \eqref{thm:cif:a} of that theorem says that, $\P$-almost surely, for any $(m,s)\in\Z\times\R$, the left and the right competition interfaces out of $(m,s)$ have asymptotic directions $\tht_{(m,s)}^L$ and $\tht_{(m,s)}^R$, respectively. Section \ref{sec:cif} explores the relationship between these roots of competition interfaces, the instability graph, and shock points.

\section{An overview of the main results}\label{sec:results}

In this section, we give a quick run through our main results. Section \ref{sec:previous} collects results from earlier works that we will use in this paper, while Section \ref{sec:more} presents a number of new  preliminary results. The sum of the results in Sections \ref{sec:previous} and \ref{sec:more} holds on a single event $\Omref{Omega8}\in\kS$ with $\P(\Omref{Omega8})=1$. See Remark \ref{Omega}. This is the full $\P$-probability event on which all of our results in the present section and in Sections \ref{def:Bus:analytic}-\ref{sec:cif} hold.

\subsection{Instability}
We begin by defining instability points and edges and the instability graph.  Recall the exceptional set $\baddir$, defined in \eqref{def:baddir}.
 Recall that for a non-decreasing function $f:\R\to\R$,  $t\in\R$ is a \emph{point of increase} of $f$ if, for all $r'<t$ and $r''>t$, $f(r')<f(r'')$.

\begin{defn}\label{def:inst}
	Take $\tht\in\baddir$.
\begin{enumerate}  [label={\rm(\alph*)}, ref={\rm\alph*}]   \itemsep=3pt  
    \item\label{def:inst:a} For $m\in\Z$ and $t\in\R$, we call the vertical closed interval $[(m-\half,t),(m+\half,t)]$ a $\tht$-\emph{instability edge} if
    $s\mapsto\bus{\tht-}{}(m,0,m,s)-\bus{\tht+}{}(m,0,m,s)$ 	has a point of increase at $t$.
    \item\label{def:inst:b} For $m\in\Z$ and $t\in\R$, we say that $(m+\half, t)$ is a \emph{proper $\tht$-instability point} if $\bus{\tht-}{}(m,t,m+1,t)<\bus{\tht+}{}(m,t,m+1,t)$.
    \item\label{def:inst:c} The \emph{instability graph} $\IG{\tht}=\IG{\tht}(\w)$ is the union of all the $\tht$-instability edges and the closure of the set of proper $\tht$-instability points.
    \item\label{def:inst:d} For $m\in\Z$ and $t\in\R$, the point $(m+\half,t)$ is a $\tht$-\emph{instability point} if it belongs to $\IG{\tht}$. A $\tht$-instability point that is not proper is called \emph{improper}.
    \item\label{def:inst:e} For $m\in\Z$, and an interval $I\subset\R$, $\{(m+\half,r):r\in I\}$ is a  $\tht$-\emph{instability interval} if it is a subset of $\IG{\tht}$. It is a proper instability interval if $(m+\half,r)$ is a \emph{proper} $\tht$-instability point for all $r\in I$. 
    \item\label{def:inst:f}  For $m\in\Z$ and $t\in\R$, $(m+\half,t)$ is called a \emph{double-edge $\tht$-instability point} if both $[(m-\half,t),(m+\half,t)]$ and $[(m+\half,t),(m+\frac32,t)]$ are instability edges.
\end{enumerate}
\end{defn}

In the above definition, proper instability points are defined explicitly through the Busemann process but improper instability points are defined via a closure procedure. 
Theorem \ref{thm:instpt} gives a characterization of all instability points in terms of the Busemann process. In words, it says that $(m+\half,s)\in(\Z+\half)\times\R$ is a $\tht$-instability point if and only if for any $r'<t$ and $r''>t$ we have $\buse{-}{m}{r'}{m+1}{r''}<\buse{+}{m}{r'}{m+1}{r''}$. Then Lemma \ref{geodefinstabedge1} and Theorem \ref{geo:instpt} provide equivalent characterizations in terms of the geodesics process \eqref{geo-proc}. In words, they say that $[(m-\half,t),(m+\half,t)]$ is a $\tht$-instability interval if and only if $\geo\from{(m,t)}\dir{R}{\tht}{+}$ and $\geo\from{(m,t)}\dir{L}{\tht}{-}$ separate immediately and never touch again and $(m+\half,t)$ is a $\tht$-instability point if and only if $\geo\from{(m,t)}\dir{R}{\tht}{+}$ and $\geo\from{(m+1,t)}\dir{L}{\tht}{-}$ do not intersect.

Section \ref{sec:inst} is devoted to the properties of the instability graph.   We summarize the main results of that section in the following two theorems.

The first theorem says that this graph is a connected closed bi-infinite set that goes north-east in direction $\tht$. It consists of the union of closed horizontal intervals of the form $[(m+\half,s),(m+\half,t)]$, $m\in\Z$ and $t>s$ in $\R$, with downward closed vertical instability edges $[(m-\half,s),(m+\half,s)]$ from the instability intervals' left endpoints and upward vertical instability edges $[(m+\half,t),(m+\frac32,t)]$ from the instability intervals' right endpoints, and with an additional set of vertical instability edges out of a Hausdorff dimension $1/2$ set of points $(m+\half,r)$, $s<r<t$, in the interiors of instability intervals, going either up or down, but not both. See Figure \ref{fig:inst-summary} and the left panel in Figure \ref{fig:inst}. 

\begin{thm}\label{inst:summary}
    The following holds for all $\w\in\Omref{Omega8}$ and all $\tht\in\baddir$.
    \begin{enumerate}  [label={\rm(\roman*)}, ref={\rm\roman*}]   \itemsep=3pt 
    \item\label{inst:summary.biinf} The instability graph is bi-infinite: For any instability point $\bfx\dual\in \IG{\tht}$ there exist an infinite up-right path and an infinite down-left path, both starting at $\bfx\dual$ and moving along the instability graph $\IG{\tht}$.
    \item\label{inst:summary.directed} Any up-right path $\bfx_{0:\infty}\dual$ on the instability graph $\IG{\tht}$ has direction $\tht$.
    \item\label{inst:summary.nodouble} There are no double-edge instability points in $\IG{\tht}$. 
    \item\label{inst:summary.improper} The only improper $\tht$-instability points are the boundaries of maximal $\tht$-instability intervals.
     \item\label{inst:summary.unbounded} The instability graph extends infinitely far to the left and to the right: For any $m\in\Z$
     \[\sup\bigl\{s:\bigl[(m-\tfrac12,s),(m+\tfrac12,s)\bigr]\in\IG{\tht}\bigr\}=\infty\quad\text{and}\quad
    \inf\bigl\{s:\bigl[(m-\tfrac12,s),(m+\tfrac12,s)\bigr]\in\IG{\tht}\bigr\}=-\infty.\]
    \item\label{inst:summary.I} Take $m\in\Z$ and $s<t$. The interval $\{(m+\half,r):s\le r\le t\}$ is a maximal $\tht$-instability interval if and only if $(m+\half,r)$ is a proper $\tht$-instability point for all $r\in(s,t)$ and $(m+\half,s)$ and $(m+\half,t)$ are improper $\tht$-instability points.
    \item\label{inst:summary.Hausdorff} For any $m\in\Z$, the set of instability points $(m+\half,s)\in\IG{\tht}$ from which descends an instability edge has a Hausdorff dimension of $\half$. 
    \end{enumerate}
\end{thm}

The next theorem demonstrates that the instability graph exhibits a web-like structure. Specifically, one can find up-right paths on the graph connecting any two instability points to a shared instability point (we call a \emph{common ancestor}), as well as down-left paths on the graph also leading to a common instability point (a \emph{common descendant}). See the center panel in Figure \ref{fig:inst}.

\begin{thm}\label{instabweb:summary}
 The following holds for all $\w\in\Omref{Omega8}$ and all $\tht\in\baddir$. 
    For every pair of instability points $ \bfx\dual, \bfy\dual \in\IG{\tht}$ there exist instability points $ \bfz_1\dual,\bfz_2\dual\in \IG{\tht}$ and up-right paths on the graph $\IG{\tht}$ going from $\bfz_1\dual$ to $\bfx\dual$, from $\bfz_1\dual$ to $\bfy\dual$, from $\bfx\dual$ to $\bfz_2\dual$, and from $\bfy\dual$ to $\bfz_2\dual$.  
\end{thm}

\begin{figure}[ht!]
\begin{center}

\begin{tikzpicture}[>=latex, scale=0.6]
\begin{scope}
\draw[line width=1pt,color=Inst](0,1)--(3,1);
\draw[line width=1pt,color=Inst](0,0)--(0,1);
\draw[line width=1pt,color=Inst](0.15,0)--(0.15,1);
\draw[line width=1pt,color=Inst](0.3,0)--(0.3,1);
\draw[line width=1pt,color=Inst](0.8,0)--(0.8,1);
\draw[line width=1pt,color=Inst](1,0)--(1,1);
\draw[line width=1pt,color=Inst](1.15,0)--(1.15,1);
\draw[line width=1pt,color=Inst](1.4,0)--(1.4,1);
\draw[line width=1pt,color=Inst](1.9,0)--(1.9,1);
\draw[line width=1pt,color=Inst](2.05,0)--(2.05,1);
\draw[line width=1pt,color=Inst](2.2,0)--(2.2,1);
\draw[line width=1pt,color=Inst](2.35,0)--(2.35,1);
\draw[line width=1pt,color=Inst](2.5,0)--(2.5,1);
\draw[line width=1pt,color=Inst](2.66,0)--(2.66,1);

\draw[line width=1pt,color=Inst](0.4,1)--(0.4,2);
\draw[line width=1pt,color=Inst](0.55,1)--(0.55,2);
\draw[line width=1pt,color=Inst](0.7,1)--(0.7,2);
\draw[line width=1pt,color=Inst](0.9,1)--(0.9,2);
\draw[line width=1pt,color=Inst](1.3,1)--(1.3,2);
\draw[line width=1pt,color=Inst](1.5,1)--(1.5,2);
\draw[line width=1pt,color=Inst](1.65,1)--(1.65,2);
\draw[line width=1pt,color=Inst](1.8,1)--(1.8,2);
\draw[line width=1pt,color=Inst](1.95,1)--(1.95,2);
\draw[line width=1pt,color=Inst](2.45,1)--(2.45,2);
\draw[line width=1pt,color=Inst](2.6,1)--(2.6,2);
\draw[line width=1pt,color=Inst](2.75,1)--(2.75,2);
\draw[line width=1pt,color=Inst](2.9,1)--(2.9,2);
\draw[line width=1pt,color=Inst](3,1)--(3,2);
\end{scope}

\begin{scope}[shift={(7,0)}]
\draw[line width=1pt,color=Inst](-1,-1)--(-1,0)--(0,0)--(0,1)--(2,1)--(2,2)--(4,2);
\draw[line width=1pt,color=Inst](-1,-1)--(0.5,-1)--(0.5,0)--(3,0)--(3,1)--(4,1)--(4,2);

\draw(0.5,1)node[above]{$\bfx$};
\draw(1.5,0)node[below]{$\bfy$};
\draw(4,2)node[above]{$\bfz$};
\draw(-1,-1)node[below]{$\bfz'$};

\shade[ball color=Inst](0.5,1)circle(2mm);
\shade[ball color=Inst](1.5,0)circle(2mm);
\shade[ball color=Inst](4,2)circle(2mm);
\shade[ball color=Inst](-1,-1)circle(2mm);

\end{scope}

\begin{scope}[shift={(14,-1.5)}]
\draw[line width=1.5pt,color=Inst](0,0)--(1,1)--(3,2)--(3.75,3);
\draw[line width=1.5pt,color=Inst](1,1)--(5,2)--(5.7,3)--(7,4);
\draw[line width=1pt,color=my-blue](0,0)--(0.8,1);
\draw[line width=1pt,color=my-blue](0,0)--(1.2,.95);
\draw[line width=1.5pt,color=Inst](0,0)--(1.5,1);
\draw[line width=1pt,color=my-blue](1,1)--(2.5,2);
\draw[line width=1pt,color=my-blue](1,1)--(2.7,2);
\draw[line width=1pt,color=my-blue](1,1)--(3.5,2);
\draw[line width=1pt,color=my-blue](1,1)--(5.6,2);
\draw[line width=1pt,color=my-blue](1,1)--(4.5,2);
\draw[line width=1pt,color=my-blue](3,2)--(3.4,3);
\draw[line width=1pt,color=my-blue](3,2)--(3.55,3);
\draw[line width=1pt,color=my-blue](3,2)--(3.95,3);
\draw[line width=1pt,color=my-blue](3,2)--(4.1,3);
\draw[line width=1pt,color=my-blue](5,2)--(5.5,3);
\draw[line width=1pt,color=my-blue](5,2)--(5.9,3);
\draw[line width=1.5pt,color=Inst](5,2)--(6.1,3);
\draw[line width=1pt,color=my-blue](5.7,3)--(7.2,4);
\draw[line width=1pt,color=my-blue](5.7,3)--(7.4,4);
\draw[line width=1pt,color=my-blue](3.55,3)--(3.7,4);
\draw[line width=1pt,color=my-blue](3.55,3)--(3.85,4);
\draw[line width=1pt,color=my-blue](3.55,3)--(4,4);

\end{scope}

\end{tikzpicture}
\end{center}

\caption{\small Left: An illustration of a maximal $\tht$-instability interval.
In reality, the set of endpoints of the vertical edges has a Hausdorff dimension $\half$. Thus, each vertical edge has uncountably many vertical edges near it, similar to the way the zeros of standard Brownian motion behave. There are no $\tht$-instability points from which emanate both a vertical edge going upward and another going downward. Center: An illustration of the web-like structure. Points $\bfx$ and $\bfy$ have a common NE ancestor $\bfz$ and a common SW descendant $\bfz'$. Right: An illustration of the shocks tree structure. When $\tht\not\in\baddir$, the coloring is immaterial. When $\tht\in\baddir$ and $\protect\sigg=+$, red depicts the subtree of points in $\NU_1^{\tht+}\setminus\NU_1^{\tht-}$ and blue depicts the ``bushes'' of points in $\NU_1^{\tht+}\cap\NU_1^{\tht-}$. Similarly, when $\tht\in\baddir$ and $\protect\sigg=-$, red depicts the subtree of points in $\NU_1^{\tht-}\setminus\NU_1^{\tht+}$ and blue depicts the ``bushes'' of points in $\NU_1^{\tht+}\cap\NU_1^{\tht-}$.}
\label{fig:inst}
\end{figure}

\subsection{Shocks}
In Section \ref{sec:shocks} we study the structure of shock points and their relation to the instability graph. 
As one might expect from the connection to the Burgers' equation, shock points form coalescing trees.

\begin{defn}\label{def:shock-tree}
   Take $\tht>0$ and $\sigg\in\{-,+\}$. 
   We say that $\bfy\in\NU_1^{\tht\sig}$ is a NE \emph{ancestor} of $\bfx\in\NU_1^{\tht\sig}$ or, equivalently, that $\bfx$ is a SW \emph{descendant} of $\bfy$ if $\bfy$ is weakly between $\geo\from{\bfx}\dir{L}{\tht}{\sig}$ and $\geo\from{\bfx}\dir{R}{\tht}{\sig}$.
\end{defn}

Note that for any $\bfx\in\NU_1^{\tht\sig}$, $\bfx$ is both an ancestor and a descendant of $\bfx$. The above ancestry relation is a partial order on $\NU_1^{\tht\sig}$.

We show that for each $\tht>0$ and $\sigg\in\{-,+\}$, the above ancestry relation induces a tree structure on $\NU_1^{\tht\sig}$. The tree's branches coalesce in the south-west direction and die out in the north-east direction. That is, the tree does not have a bi-infinite backbone. Each shock point $(m,s)$ has a unique immediate child on level $m-1$. Interestingly, even though the tree branches die out in the north-east direction, each shock point $(m,s)$ that has a NE ancestor on level $m+1$ has infinitely many of them on that level. 
We also show that the shock points in $\NU_1^{\tht+}\setminus\NU_1^{\tht-}$  form a subtree of the $\NU_1^{\tht+}$ tree and the shock points in $\NU_1^{\tht-}\setminus\NU_1^{\tht+}$  form a subtree of the $\NU_1^{\tht-}$ tree. In both cases, the remaining points in $\NU_1^{\tht-}\cap\NU_1^{\tht+}$ form smaller trees (bushes) with each such bush containing infinitely many branches but only finitely many generations. See the right panel in Figure \ref{fig:inst}.

\begin{rmk}
    In the standard tree convention, children branch out from their parents. However, our terminology reverses this direction because the Hamilton-Jacobi equation progresses downward in time. Thus, we refer to the points further down as children and the points further up as ancestors.  
\end{rmk}

\begin{thm}\label{shocks:summary}
    The following holds for all $\w\in\Omref{Omega8}$, all $\tht>0$, and all $\sigg\in\{-,+\}$.
    \begin{enumerate}  [label={\rm(\roman*)}, ref={\rm\roman*}]   \itemsep=3pt 
    \item Take $(m,s)\in\NU_1^{\tht\sig}$. Then for any $r<s$ with $(m,r)\in\NU_1^{\tht\sig}$, $(m,r)$ is not a descendant of $(m,s)$.
    \item Take $\bfx, \bfy,\bfz \in\NU_1^{\tht\sig}$. Assume that $\bfx$ and $\bfy$ are both SW descendants of $\bfz$. Then either $\bfx$ is a SW descendant of $\bfy$ or $\bfy$ is a SW descendant of $\bfx$. 
    \item For each $(m,s) \in \NU_1^{\tht\sig}$ there exists a unique $t<s$ such that $(m-1,t)\in \NU_1^{\tht\sig}$ and $(m-1,t)$ is a descendant of $(m,s)$. Furthermore, if $s_1<s_2$ are such that $(m,s_1),(m,s_2)\in\NU_1^{\tht\sig}$ and $(m-1,t_i)$ is the unique SW descendent of $(m,s_i)$ for each $i\in\{1,2\}$, then $t_1\le t_2$.
    \item Suppose $(m,s)\in\NU_1^{\tht\sig}$ has a NE ancestor. Then it has infinitely many NE ancestors on level $m+1$.
    \item Take $\bfx, \bfy \in\NU_1^{\tht\sig}$. Then $\bfx$ and $\bfy$ have a common SW descendant $\bfz \in \NU_1^{\tht\sig}$.
    \item Take $(m,s) \in\NU_1^{\tht+}\setminus\NU_1^{\tht-} $. Then all SW descendants of $ (m,s) $ in the $\NU_1^{\tht+}$ tree  are in $\NU_1^{\tht+}\setminus \NU_1^{\tht-} $.
    \item Take $ (m,s) \in\NU_1^{\tht-}\setminus\NU_1^{\tht+} $. Then all SW descendants of $ (m,s) $ in the $\NU_1^{\tht-}$ tree are in $\NU_1^{\tht-}\setminus \NU_1^{\tht+} $. 
    \item  Take $ (m,s) \in \NU_1^{\tht+}\cap \NU_1^{\tht-} $. Then $ (m,s) $ has a SW descendant in $ \NU_1^{\tht+}\setminus \NU_1^{\tht-} $ and a SW descendant in $ \NU_1^{\tht-}\setminus \NU_1^{\tht+} $.
    \end{enumerate}
\end{thm}

To analyze the relationship between shock points and the instability graph one needs to consider $\tht\in\baddir$. Then, in this case, the $\pm$ distinction implies that there are two types of shock points: $\NU_1^{\tht-}$ and $\NU_1^{\tht+}$. Note that shock points live on the primal semi-discrete lattice $\Z\times\R$ while instability points live on the dual lattice $(\Z+\half)\times\R$. 
We prove that shock points that have no instability points above them are both $\tht-$ and $\tht+$ shocks. On the other hand, if a shock point has an instability point above it, then it is either only a $\tht+$ shock point if the instability point is a left endpoint of an instability interval, or only a $\tht-$ shock point if the instability point above it is an interior point of an instability interval. Each shock point of the latter type has an instability edge going through it. Lastly, a feature that identifies the right endpoints of instability intervals is that they have a $\tht-$ shock point above them.  Thus, knowing the locations of the $\tht\pm$ shock points allows us to reconstruct all the horizontal instability  intervals, and a countable number of vertical instability edges. 

\begin{thm}\label{thm:shocks-IG}
The following holds for all $\w\in\Omref{Omega8}$ and all $\tht\in\baddir$. 
Suppose that for some $m\in\Z$ and $s<t$, $\{(m+\half,r):s\le r\le t\}$ is a $\tht$-instability interval and that $(m+\half,s)$ and $(m+\half,t)$ are improper $\tht$-instability points.
\begin{enumerate}  [label={\rm(\roman*)}, ref={\rm\roman*}]   \itemsep=3pt 
\item\label{thm:shocks-IG.a} $(m,s)\in\NU_1^{\tht+}\setminus\NU_1^{\tht-}$ and $(m+1,t)\in\NU_1^{\tht-}\setminus\NU_1^{\tht+}$. 
\item\label{thm:shocks-IG.b} For any $r\in(s,t)$, $(m,r)\not\in\NU_1^{\tht+}$. Furthermore, $(m,r)\in\NU_1^{\tht-}$ if, and only if, $[(m-\half,r),(m+\half,r)]$ is a $\tht$-instability edge that is right-isolated among $\tht$-instability edges. 
\item\label{thm:shocks-IG.c} For any $r\in(s,t)$ with $(m,r)\in\NU_1^{\tht-}$, $[(m-\half,r),(m+\half,r)]$ is a $\tht$-instability edge.
\item\label{thm:shocks-IG.d} There exists an $\e>0$ such that for any $r\in[t,t+\e]$, $(m,r)\not\in\NU_1^{\tht-}\cup\NU_1^{\tht+}$.
\item\label{thm:shocks-IG.e} For any $r\in\R$ such that $(m+\half,r)\not\in\IG{\tht}$, $(m,r)$ is either in $\NU_1^{\tht-}\cap\NU_1^{\tht+}$ or not in $\NU_1^{\tht-}\cup\NU_1^{\tht+}$.
\item\label{thm:shocks-IG.f} There exist sequences $r_n'\nearrow s$ and $r_n''\searrow s$ such that $(m,r_n')\in\NU_1^{\tht-}\cap\NU_1^{\tht+}$ and $(m,r_n'')\in\NU_1^{\tht-}\setminus\NU_1^{\tht+}$ for all $n$.
\end{enumerate}
\end{thm}

\begin{rmk}\label{more shocks}
Parts \eqref{thm:shocks-IG.a}, \eqref{thm:shocks-IG.b},  \eqref{thm:shocks-IG.e}, and \eqref{thm:shocks-IG.f} in the above theorem give a natural injection that says that there are infinitely many more $\tht-$ shocks than $\tht+$ shocks. We will see, in Section \ref{sec:shocks}, a more general result that says that the number of shock points strictly increases as the velocity $\tht$ decreases. See Remark \ref{more shocks 2}.
\end{rmk}

\begin{rmk}
    Parts \eqref{thm:shocks-IG.a}, \eqref{thm:shocks-IG.b}, \eqref{thm:shocks-IG.d}, and \eqref{thm:shocks-IG.e} show that $\NU_1^{\tht+}\setminus\NU_1^{\tht-}$ consists of isolated points. This is in contrast to $\NU_1^{\tht-}\setminus\NU_1^{\tht+}$, which by parts \eqref{thm:shocks-IG.a}, \eqref{thm:shocks-IG.b}, \eqref{thm:shocks-IG.d}, \eqref{thm:shocks-IG.e}, and Theorem \ref{thm:shocks}\eqref{thm:shocks:c} is a left-dense set, and to  $\NU_1^{\tht-}\cap\NU_1^{\tht+}$, which by parts (\ref{thm:shocks-IG.a}-\ref{thm:shocks-IG.c}), \eqref{thm:shocks-IG.e},  and Theorem \ref{thm:shocks}\eqref{thm:shocks:c} is also a left-dense set. 
\end{rmk}

\subsection{Competition interfaces}
Since the set of vertical instability edges has a Hausdorff dimension $\half$ while shock points are only countably infinite, one cannot reconstruct all the instability edges from only knowing the locations of shock points. However, one can completely reconstruct the instability graph from knowing the starting points of left and right competition interfaces with asymptotic direction $\tht$ (see Section \ref{sec:intro.ci}). Specifically, we show that an instability edge goes through $(m,s)\in\Z\times\R$ if, and only if, $(m+\half,s)$ is the starting point of either a left or a right competition interface with asymptotic direction $\tht$, i.e.\ $\tht\in\{\tht_{(m,s)}^L,\tht_{(m,s)}^R\}$. 
Furthermore, these points identify the shock points that are below instability points: if $(m+\half,s)$ is an instability point, then $(m,s)\in\NU_1^{\tht+}\setminus\NU_1^{\tht-}$ if and only if $(m+\half,s)$ is the starting point of a right, but not left, competition interface with an asymptotic direction $\tht$ (i.e., $\tht=\tht_{(m,s)}^R < \tht_{(m,s)}^L$). Similarly, $(m,s)\in\NU_1^{\tht-}\setminus\NU_1^{\tht+}$ if and only if $(m+\half,s)$ is the starting point of a left, but not right, competition interface with an asymptotic direction $\tht$ (i.e., $\tht=\tht_{(m,s)}^L > \tht_{(m,s)}^R$), and $(m,s)\not\in\NU_1^{\tht-}\cup\NU_1^{\tht+}$ if and only if $(m+\half,s)$ is the starting point of both left and right competition interfaces with the same asymptotic direction $\tht$ (i.e., $\tht=\tht_{(m,s)}^L=\tht_{(m,s)}^R$).

\begin{thm}\label{thm:cif}
 The following holds for all $\w\in\Omref{Omega8}$, $\tht\in\baddir$, and $(m,s)\in\Z\times\R$.
    \begin{enumerate} [label={\rm(\alph*)}, ref={\rm\alph*}]   \itemsep=3pt 
    \item\label{thm:cif.a} $[(m-\half,s),(m+\half,s)]$ is a $\tht$-instability edge if, and only if, for some $S\in\{L,R\}$, $\geo\from{(m,s)}\dir{S}{\tht+}{}$ and $\geo\from{(m,s)}\dir{S}{\tht-}{}$ split at $(m,s)$ {\rm(}i.e.\ $\tht\in\{\tht^L_{(m,s)},\tht^R_{(m,s)}\}${\rm)}.
    \item\label{thm:cif.b} If $(m+\half,s)$ is the left endpoint of a maximal $\tht$-instability interval,
    then $\geo\from{(m,s)}\dir{R}{\tht+}{}$ and $\geo\from{(m,s)}\dir{R}{\tht-}{}$ split at $(m,s)$ but $\geo\from{(m,s)}\dir{L}{\tht+}{}$ and $\geo\from{(m,s)}\dir{L}{\tht-}{}$ do not. That is, $\tht=\tht^R_{(m,s)}<\tht^L_{(m,s)}$.
    \item\label{thm:cif.c} If $(m+\half,s)$ is a proper $\tht$-instability point and $(m,s)\in\NU_1^{\tht-}$, then $\geo\from{(m,s)}\dir{L}{\tht+}{}$ and $\geo\from{(m,s)}\dir{L}{\tht-}{}$ split at $(m,s)$ but $\geo\from{(m,s)}\dir{R}{\tht+}{}$ and $\geo\from{(m,s)}\dir{R}{\tht-}{}$ do not. That is, $\tht=\tht^L_{(m,s)}>\tht^R_{(m,s)}$.
    \item\label{thm:cif.d} If $[(m-\half,s),(m+\half,s)]$ is a $\tht$-instability edge but $(m,s)\not\in\NU_1^{\tht-}\cup\NU_1^{\tht+}$, then $\geo\from{(m,s)}\dir{L}{\tht-}{}=\geo\from{(m,s)}\dir{L}{\tht+}{}$ and $\geo\from{(m,s)}\dir{R}{\tht-}{}=\geo\from{(m,s)}\dir{R}{\tht+}{}$ split at $(m,s)$. That is, $\tht=\tht^R_{(m,s)}=\tht^L_{(m,s)}$.
    \end{enumerate}
\end{thm}

In Sections \ref{def:Bus:analytic} through \ref{sec:cif}, we will introduce and explore a broader concept of instability. This expanded notion enables us to examine the relationship between the locations of shock points as the parameter 
$\tht$ varies. The definitions and results presented in this section will subsequently be seen as a special case of the more general findings discussed later.

\section{Inputs from previous works}\label{sec:previous}

In this section, we compile several previous results that will be utilized in this work.

\subsection{The Busemann process}
The following theorem summarizes the properties of the Busemann process (defined in \eqref{bus.def}) that we need in this work. 
It comes from Theorem 3.5 in \cite{Sep-Sor-23-aihp} and Theorems 3.15 and 7.19(v) in \cite{Sep-Sor-23-pmp}.

\begin{thm}\label{thm:B}
    There exists a real-valued process 
    \begin{align}\label{Bus-proc}
        \bigl\{\buse{\sig}{m}{s}{n}{t}:\tht>0,\sigg\in\{-,+\},m,n\in\Z,s,t\in\R\bigr\}
    \end{align}
    and an event $\Omnew{Omega0}\in\kS$ with $\P(\Omref{Omega0})=1$ and such that the following hold for all $\w\in\Omref{Omega0}$, $\tht>0$, and $\sigg\in\{-,+\}$.
    \begin{enumerate} [label={\rm(\roman*)}, ref={\rm\roman*}]   \itemsep=3pt 
    \item\label{B:cocycle}{\rm(}Cocycle{\rm)} For any $\ell,m,n\in\Z$, and $r,s,t\in\R$,
    \begin{align}\label{cocycle}
    \buse{\sig}{\ell}{r}{m}{s}+\buse{\sig}{m}{s}{n}{t}=\buse{\sig}{\ell}{r}{n}{t}.
    \end{align}
    \item\label{B:mono}{\rm(}Monotonicity{\rm)} For any 
    $\tht'>\tht$, 
    $m\in\Z$, and $s<t$ in $\R$, 
    \begin{align}\label{B:mono-eqn}
    \begin{split}
    &\bus{\tht'}{+}(m,s,m,t)\le\bus{\tht'}{-}(m,s,m,t)\le\buse{+}{m}{s}{m}{t}\le \buse{-}{m}{s}{m}{t}
    \quad\text{and}\quad\\
    &\buse{-}{m}{s}{m+1}{s}\le \buse{+}{m}{s}{m+1}{s}\le\bus{\tht'}{-}(m,s,m+1,s)\le\bus{\tht'}{+}(m,s,m+1,s). 
    \end{split}
    \end{align}
    \item\label{B:cont}{\rm(}Continuity{\rm)} For any $m,n\in\Z$, $(s,t)\mapsto\buse{\sig}{m}{s}{n}{t}$ is a continuous function.
    \item\label{B:cont2}{\rm(}Limits{\rm)} $\bus{\gamma}{\sig}$ converges uniformly on compacts to $\bus{\tht}{-}$ as $\gamma$ increases to $\tht$ and to $\bus{\tht}{+}$ as $\gamma$ decreases to $\tht$.
    \item\label{B:jump}{\rm(}Locally constant{\rm)} For any compact set $K\subset\Z\times\R$ there exists an $\e=\e(\w,K,\tht)>0$ such that for any $(m,s),(n,t)\in K$, $\gamma\in(\tht-\e,\tht)$, and $\delta\in(\tht,\tht+\e)$, 
        \[\bus{\gamma}{\sig}(m,s,n,t)=\bus{\tht}{-}(m,s,n,t)\quad\text{and}\quad 
        \bus{\delta}{\sig}(m,s,n,t)=\bus{\tht}{+}(m,s,n,t).\]
    \end{enumerate}
\end{thm}

%

\subsection{Semi-infinite geodesics} Next, we need results concerning  semi-infinite geodesics.

\begin{defn} The following are a collection of definitions pertaining to semi-infinite paths.
\begin{enumerate} [label={\rm(\alph*)}, ref={\rm\alph*}]   \itemsep=3pt 
    \item Given two semi-infinite up-right paths $\Gamma^1$ and $\Gamma^2$ starting at the same level $m\in\Z$ and defined by their jump times $\{s_k^1:k\ge m-1\}$ and $\{s_k^2:k\ge m-1\}$, respectively, we say that $\Gamma^1$ is to the left of, or equivalently above, $\Gamma^2$ if $s_k^1\le s_k^2$ for all $k\ge m-1$. We write $\Gamma^1\preceq\Gamma^2$ or, equivalently, $\Gamma^2\succeq\Gamma^1$. 
    \item Fix a starting level $\ell\in\Z$. A sequence of up-right, semi-infinite paths, defined by the sequences of jump times $(s_i^{(k)})_{i\ge \ell-1}$, $k\in\N$, converges to the up-right path defined by the sequence of jump times $(s_i)_{i\ge \ell-1}$ if for each integer $i\ge\ell-1$, $s_i^{(k)}\to s_i$ as $k\to\infty$.
    \item A semi-infinite or bi-infinite up-right path defined by the sequence of jump times $(s_n)_{n\ge\ell-1}$ or, respectively, $(s_n)_{n\in\Z}$ is said to have direction $\tht>0$ if $s_n/n\to\tht$ as $n\to\infty$. 
    \item Two up-right, semi-infinite paths, defined by the sequences of jump times $(s_i)_{i\ge \ell-1}$ and $(s'_j)_{j\ge k-1}$, coalesce if $n=\min\bigl\{j\ge(\ell-1)\vee(k-1):s_j=s'_j\not\in\{s_{\ell-1}\}\cap\{s'_{k-1}\}\bigr\}<\infty$ and $s_i=s'_i$ for all $i\ge n$. In words, the two paths coalesce if they either start at the same point, immediately split, then intersect later and continue together, or if they start at different points and intersect later to proceed together.
\end{enumerate}
\end{defn}

The following theorem summarizes the properties of the process \eqref{geo-proc} that we need in this work. Its claims follow from Theorem 4.3(ii,v) in \cite{Sep-Sor-23-aihp}, Theorems 4.3(iii), 4.5(ii), and 4.11 in \cite{Sep-Sor-23-pmp}, and our Proposition \ref{nodouble}.

\begin{thm}
\label{thm:geo}
There exists an event $\Omnew{Omega1}\in\kS$ with $\Omref{Omega1}\subset\Omref{Omega0}$, $\P(\Omref{Omega1})=1$, and such that the geodesics process \eqref{geo-proc} satisfies all of the following properties, for all $\w\in\Omref{Omega1}$, $\tht>0$, and $\sigg\in\{-,+\}$.

\begin{enumerate} [label={\rm(\roman*)}, ref={\rm\roman*}]   \itemsep=3pt 
\item\label{thm:geo:mono}{\rm(}Monotonicity{\rm)}  
For all $m\in\Z$, $s<t$ in $\R$, and $S \in \{L,R\}$
\begin{align}
    &\geo\from{(m,s)}\dir{R}{\tht}{\sig} \preceq \geo\from{(m,t)}\dir{L}{\tht}{\sig} \, , \quad 
    \geo\from{(m,t)}\dir{L}{\tht}{\sig}  \preceq \geo\from{(m,t)}\dir{R}{\tht}{\sig}\, , \quad
    \geo\from{(m,t)}\dir{S}{\tht}{-} \preceq \geo\from{(m,t)}\dir{S}{\tht}{+}\,.\label{geomono}
\end{align}
\item\label{thm:geo:cont}{\rm(}Continuity{\rm)}
For all $m\in\Z$, $s\in\R$, and $S \in \{L,R\}$
\begin{align}
    &\lim\limits_{r\nearrow s}\geo\from{(m,r)}\dir{S}{\tht}{\sig} = \geo\from{(m,s)}\dir{L}{\tht}{\sig} \, \quad\text{and}\quad
     \lim\limits_{r\searrow s}\geo\from{(m,r)}\dir{S}{\tht}{\sig}=\geo\from{(m,s)}\dir{R}{\tht}{\sig}.\label{geocont}
\end{align}
\item\label{thm:geo:dir}{\rm(}Directedness{\rm)} For all $m\in\Z$, $s\in\R$, and $S \in \{L,R\}$, $\geo\from{(m,s)}\dir{S}{\tht}{\sig}$ has direction $\tht$.
\item\label{thm:geo:coal}{\rm(}Coalescence{\rm)} For all $m,n\in\Z$, $s,t\in\R$, and $S,S' \in \{L,R\}$, $\geo\from{(m,s)}\dir{S}{\tht}{\sig}$ and $\geo\from{(n,t)}\dir{S'}{\tht}{\sig}$ coalesce. 
\item\label{thm:geo:LR}{\rm(}Leftmost and rightmost{\rm)} For all $(m,s)\in\Z\times\R$, if $\bfx$ and $\bfy$ are on $\geo\from{(m,s)}\dir{L}{\tht}{\sig}$, then the segment on this path that is between $\bfx$ and $\bfy$ coincides with the leftmost geodesic path between $\bfx$ and $\bfy$. Similarly, if $\bfx$ and $\bfy$ are on $\geo\from{(m,s)}\dir{R}{\tht}{\sig}$, then the segment on this path that is between $\bfx$ and $\bfy$ coincides with the rightmost geodesic path between $\bfx$ and $\bfy$.
\end{enumerate}
\end{thm}

\begin{rmk}
    Theorem 4.11 in \cite{Sep-Sor-23-pmp} leaves the possibility of the geodesic $\geo\from{(m,s)}\dir{L}{\tht}{\sig}$ moving vertically, going through $(n,s)$ and $(n+1,s)$ for some integer $n>m$, and the geodesic $\geo\from{(n,s)}\dir{R}{\tht}{\sig}$ first going to the right from $(n,s)$, and then later merging with $\geo\from{(m,s)}\dir{L}{\tht}{\sig}$. See Remark 4.14 in that paper. Our Proposition \ref{nodouble} says that this cannot happen. 
\end{rmk}

\begin{rmk}\label{rk:all geo}
    For each $(m,s)\in\Z\times\R$, the paths $\{(m,t):t\ge s\}$ and $\{(r,s):r\ge m\}$ are degenerate geodesics. 
    Theorem 3.1(iv-v) in \cite{Sep-Sor-23-aihp} implies that with $\P$-probability one, for any starting point $(m,s)\in\Z\times\R$, every non-degenerate semi-infinite geodesic out of $(m,s)$ is $\tht$-directed for some $\tht>0$. Then Theorem 4.5(ii) in the same paper implies that this geodesic stays between $\geo\from{(m,s)}\dir{L}{\tht}{-}$ and $\geo\from{(m,s)}\dir{R}{\tht}{+}$.
    That said, it is expected that for every non-degenerate semi-infinite geodesic out of $(m,s)$ there exists a sign $\sigg\in\{-,+\}$ such that the geodesic is trapped between $\geo\from{(m,s)}\dir{L}{\tht}{\sig}$ and $\geo\from{(m,s)}\dir{R}{\tht}{\sig}$. This has already been shown in the case of the fully discrete last-passage percolation model on $\Z^2$ with exponentially distributed weights \cite[Theorem 3.11(a)]{Jan-Ras-Sep-23} and in the fully continuous model called the \emph{Directed Landscape} \cite[Theorem 1.8(iv)]{Bus-25-}. Proving this conjecture for the Brownian last-passage percolation model is work in progress. If it does hold, then together with Lemma \ref{no-geo}, this would imply that the geodesics process \eqref{geo-proc} contains all the non-degenerate semi-infinite geodesics of the model (see also Remark 4.23 in \cite{Sep-Sor-23-pmp}). 
    This last fact is known to hold in the case of the last-passage percolation model with exponential weights \cite[Theorem 3.11(a)]{Jan-Ras-Sep-23}. 
\end{rmk}

\subsection{Shock points} 

Although Theorem B.1 in \cite{Ham-19} says that for any given $(m,s)\le(n,t)$ in $\Z\times\R$, there exists, $\P$-almost surely, a unique geodesic path between the two points, there are random points between which there exist multiple geodesics. Thus, \cite{Sep-Sor-23-pmp} introduced the sets
\[\NU_0^{\tht\sig}=\bigl\{(m,s)\in\Z\times\R:\geo\from{(m,s)}\dir{L}{\tht}{\sig}\ne\geo\from{(m,s)}\dir{R}{\tht}{\sig}\bigr\}\]
and
\begin{align}\label{def:NU1}
\NU_1^{\tht\sig}=\bigl\{(m,s)\in\Z\times\R:(m+1,s)\in\geo\from{(m,s)}\dir{L}{\tht}{\sig},\exists\delta>0\text{ such that }(m,s+\delta)\in\geo\from{(m,s)}\dir{R}{\tht}{\sig}\bigr\},
\end{align}
for $\tht>0$ and $\sigg\in\{-,+\}$.
In words, points in $\NU_0^{\tht\sig}$ are ones for which the geodesics $\geo\from{(m,s)}\dir{L}{\tht}{\sig}$ and $\geo\from{(m,s)}\dir{R}{\tht}{\sig}$ eventually split and points in $\NU_1^{\tht\sig}$ are ones at which the two geodesics split immediately.  Our Lemma \ref{geo-split} shows that the two geodesics can only split at their starting point. This implies the following.

\begin{figure}[ht!]
\begin{center}

\begin{tikzpicture}[>=latex, scale=0.6]
\begin{scope}
\draw(0,0)--(6,0);
\draw(0,1)--(6,1);
\draw(0,2)--(6,2);
\draw[line width=2pt,color=Rp,->](1,0)--(2,0)--(2,1)--(2.5,1)--(2.5,1.95)--(4,1.95)--(4,2.46)--(5.5,3.8);
\draw[line width=2pt,color=Lp,->](1,0)--(1,1)--(1.5,1)--(1.5,2.05)--(3.9,2.05)--(3.9,2.54)--(5.1,3.55);
\shade[ball color=pt](1,0)circle(2mm);
\end{scope}

\begin{scope}[shift={(10,0)}]
\draw(0,0)--(6,0);
\draw(0,1)--(6,1);
\draw(0,2)--(6,2);
\draw(0,3)--(6,3);
\draw[line width=2pt,color=Rp,->](1,0)--(1,1)--(2,1)--(2,2)--(2.5,2)--(2.5,2.95)--(4,2.95)--(4,3.45)--(5.5,4.95);
\draw[line width=2pt,color=Lp,->](0.9,0)--(0.9,2)--(1.5,2)--(1.5,3.05)--(3.9,3.05)--(3.9,3.5)--(4.9,4.5);
\shade[ball color=pt](.95,0)circle(2mm);
\end{scope}

\end{tikzpicture}
\end{center}

\caption{\small Left: An illustration of a point $(m,s)\in\NU_1^{\tht\protect\sig}$. Right: An illustration of a point  $(m,s)\in\NU_0^{\tht\protect\sig}\setminus\NU_1^{\tht\protect\sig}$. Proposition \ref{nodouble} prevents the existence of points of the latter type. In both figures, the top up-right path is $\geo\from{(m,s)}\dir{L}{\tht}{\protect\sig}$ and the bottom up-right path is $\geo\from{(m,s)}\dir{R}{\tht}{\protect\sig}$.}
\label{fig:shocks}
\end{figure}

\begin{lem}\label{NU0=NU1}
There exists an event $\Omnew{Omega2}\in\kS$ such that $\Omref{Omega2}\subset\Omref{Omega1}$, $\P(\Omref{Omega2})=1$, and for any $\w\in\Omref{Omega2}$, $\tht>0$, and $\sigg\in\{-,+\}$, $\NU_0^{\tht\sig}=\NU_1^{\tht\sig}$.
\end{lem}

Note that even though for $(m,s)\in\NU_1^{\tht\sig}$, $\geo\from{(m,s)}\dir{L}{\tht}{\sig}$ and $\geo\from{(m,s)}\dir{R}{\tht}{\sig}$ split, the coalescence in Theorem \ref{thm:geo}\eqref{thm:geo:coal} says that they have to come together at some later point and then proceed together from there onwards.

The following is a summary of the subset of results from \cite{Sep-Sor-23-pmp} that we will need in this work. A few more results relating shock points to competition interfaces appear further down in Theorem \ref{thm:cif-SS}.

\begin{thm}[Theorems 4.8(i,iii) and 4.32(iv) in \cite{Sep-Sor-23-pmp}]
\label{thm:shocks}
There exists an event $\Omnew{Omega3}\in\kS$ with $\Omref{Omega3}\subset\Omref{Omega2}$, $\P(\Omref{Omega3})=1$, and such that the following hold for all $\w\in\Omref{Omega3}$, $\tht>0$, and $\sigg\in\{-,+\}$.

\begin{enumerate} [label={\rm(\roman*)}, ref={\rm\roman*}]   \itemsep=3pt 
\item\label{thm:shocks:0} For any $(m,s)\in\Z\times\Q$, $(m,s)\not\in\NU_1^{\tht\sig}$.
\item\label{thm:shocks:a} $\NU_1^{\tht\sig}$ is countably infinite. 
\item\label{thm:shocks:c} For each $(m,s)\in\NU_1^{\tht\sig}$, there exists a sequence $(m,s_k)\in\NU_1^{\tht\sig}$ such that $s_k$ strictly increases to $s$ as $k\to\infty$.
\item\label{thm:shocks:d} For each $(m,s)\in\NU_1^{\tht+}\setminus\NU_1^{\tht-}$, there is a sequence $s_k$, strictly decreasing to $s$, and such that $(m,s_k)\in\NU_1^{\tht-}$ for all $k$.
\end{enumerate}
\end{thm}

\begin{rmk}
    Even though we will not use this fact, it is noteworthy that Theorem 4.8(ii) in \cite{Sep-Sor-23-pmp} states the following stronger property than the one in Theorem \ref{thm:shocks}\eqref{thm:shocks:a}: with $\P$-probability one, $\bigcup_{\tht>0}(\NU_1^{\tht-}\cup\NU_1^{\tht+})$ is countable. 
\end{rmk}

\subsection{Instability points}
Recall the set $\baddir$, defined in \eqref{def:baddir}. The next theorem says that $\baddir$ is almost surely a ``genuinely random'' countable dense set. 

\begin{thm}[Theorem 2.5 in \cite{Sep-Sor-23-pmp}]\label{baddir dense}
For any $\tht>0$, $\P(\tht\in\baddir)=0$.
There exists an event $\Omnew{Omega3a}\in\kS$ with $\Omref{Omega3a}\subset\Omref{Omega3}$, $\P(\Omref{Omega3a})=1$, and such that for any $\w\in\Omref{Omega3a}$, $\baddir$ is countable and dense in $\R_+$.
\end{thm}

We will also need the following result, which follows from Theorem 5.4 in \cite{Sep-Sor-23-pmp} and from Theorems 5.3(iii) and 5.5(ii) together with Corollaries 8.10 and 8.11 in \cite{Bus-Sep-Sor-24}.

\begin{thm}\label{B:increase}
There exists an event $\Omnew{Omega3b}\in\kS$ with $\Omref{Omega3b}\subset\Omref{Omega3a}$, $\P(\Omref{Omega3b})=1$, and such that for any $\w\in\Omref{Omega3b}$, $\tht\in\baddir$, and any $m\in\Z$, 
    \[\lim_{t\to-\infty}(\bus{\tht}{-}(m,0,m,t)-\bus{\tht}{+}(m,0,m,t))=-\infty,\quad
    \lim_{t\to\infty}(\bus{\tht}{-}(m,0,m,t)-\bus{\tht}{+}(m,0,m,t))=\infty,\]
    and the set of points of increase of the function 
    \[t\mapsto\bus{\tht-}{}(m,0,m,t)-\buse{+}{m}{0}{m}{t}\] 
    has a Hausdorff dimension of $\half$. 
\end{thm}

We will need for our analysis the following connection between the discontinuities of the Busemann process (as a function of $\tht$) and the intersection of leftmost and rightmost geodesics.

\begin{thm}[Theorems 4.20 and 8.8 in \cite{Sep-Sor-23-pmp}]
\label{thm:Bus-geo}
There exists an event $\Omnew{Omega4}\in\kS$ with $\Omref{Omega4}\subset\Omref{Omega3b}$, $\P(\Omref{Omega4})=1$, and such that for any $\w\in\Omref{Omega4}$, $\tht\ge\gamma>0$, and $\bfx\succeq\bfy$ in $\Z\times\R$, the following statements are equivalent.

\begin{enumerate} [label={\rm(\roman*)}, ref={\rm\roman*}]   \itemsep=3pt 
\item\label{thm:Bus-geo.a} $\bus{\gamma}{-}{(\bfx,\bfy)}=\bus{\theta}{+}{(\bfx,\bfy)}$.
\item\label{thm:Bus-geo.b} $\geo\from{\bfy}\dir{R}{\gamma}{-}\cap\geo\from{\bfx}\dir{L}{\tht}{+}\ne\varnothing$.
\end{enumerate}
\end{thm}

\subsection{Competition interfaces}\label{sec:def.cif}
Recall from Section \ref{sec:intro.ci} that the competition interface starting at $(m+\half,s)$ is an up-right path on 
the dual semi-discrete lattice $(\Z+\tfrac12)\times\R$
which separates points $(n,t)\in (m,s)+\Z_+\times\R_+$ whose 
geodesic from $(m,s)$ begins with a vertical step from those whose geodesic begins with a horizontal step. (The detailed construction is done in Remark 4.24 in \cite{Sep-Sor-23-pmp}.)
Theorem 4.29 in \cite{Sep-Sor-23-pmp} states that the left and the right competition interfaces out of $(m,s)$ have asymptotic directions $\tht_{(m,s)}^L$ and $\tht_{(m,s)}^R$, respectively. Then, as a result of the directedness in Theorem \ref{thm:geo}\eqref{thm:geo:dir}, we see that for any $S\in\{L,R\}$, if $\tht<\tht_{(m,s)}^S$, then $\geo\from{(m,s)}\dir{S}{\tht}{\sig}$, $\sigg\in\{-,+\}$, both move up from $(m,s)$, and if $\tht>\tht_{(m,s)}^S$, then $\geo\from{(m,s)}\dir{S}{\tht}{\sig}$, $\sigg\in\{-,+\}$, both proceed to the right from $(m,s)$. In particular, $\tht_{(m,s)}^R\le\tht_{(m,s)}^L$, and if $\tht\not\in(\tht_{(m,s)}^R,\tht_{(m,s)}^L)$, then $\geo\from{(m,s)}\dir{L}{\tht}{\sig}=\geo\from{(m,s)}\dir{R}{\tht}{\sig}$, and $(m,s)\not\in\NU_1^{\tht\sig}$, for any $\sigg\in\{-,+\}$. Similarly, if $\tht\in(\tht_{(m,s)}^R,\tht_{(m,s)}^L)$, then $\geo\from{(m,s)}\dir{L}{\tht}{\sig}$ and $\geo\from{(m,s)}\dir{R}{\tht}{\sig}$ split at $(m,s)$ and $(m,s)\in\NU_1^{\tht-}\cap\NU_1^{\tht+}$. These results are summarized in the following theorem, relating competition interfaces and shocks. 

Let $\{\sigma^L_{(m,s),n}:n\ge m\}$, respectively $\{\sigma^R_{(m,s),n}:n\ge m\}$, denote the jump times of the left, respectively right, competition interface out of $(m+\half,s)$. Thus, $\sigma^L_{(m,s),m}=\sigma^R_{(m,s),m}=s$.

\begin{thm}[Corollary 2.10 and Theorems 4.29 and 4.32(iii) in \cite{Sep-Sor-23-pmp}]\label{thm:cif-SS}
There exists an event $\Omnew{Omega5}\in\kS$ with $\Omref{Omega5}\subset\Omref{Omega4}$, $\P(\Omref{Omega5})=1$, and such that the following hold for any $\w\in\Omref{Omega5}$ and $(m,s)\in\Z\times\R$.

\begin{enumerate} [label={\rm(\roman*)}, ref={\rm\roman*}]   \itemsep=3pt 
\item\label{thm:cif:a} For all $S\in\{L,R\}$, the limit
\[\tht_{(m,s)}^S=\lim_{n\to\infty}\frac{\sigma^S_{(m,s),n}}n\]
exists and we have $\tht_{(m,s)}^R\le\tht_{(m,s)}^L$. \item\label{thm:cif:b} $(m,s)\in\NU_1^{\tht-}$ if and only if $\tht_{(m,s)}^R<\tht\le\tht_{(m,s)}^L$.
\item\label{thm:cif:c} $(m,s)\in\NU_1^{\tht+}$ if and only if $\tht_{(m,s)}^R\le\tht<\tht_{(m,s)}^L$.
\end{enumerate}
\end{thm}

Examining the above we see that if for $S\in\{L,R\}$, $\tht=\tht_{(m,s)}^S$, then $\geo\from{(m,s)}\dir{S}{\tht}{-}$ and $\geo\from{(m,s)}\dir{S}{\tht}{+}$ end up on different sides of the $S$-competition interface out of $(m+\half,s)$ and, consequently, they split at $(m,s)$.
Thus, the set $\baddir$ of exceptional directions is exactly the union of all the asymptotic directions of left competition interfaces from the various points in $(\Z+\half)\times\R$ and also the union of all the asymptotic directions of right competition interfaces. See Theorem 4.36(i) in \cite{Sep-Sor-23-pmp}.
A more detailed description of this relation between competition interfaces and instability points is one of our results, given as Theorem \ref{thm:cif}. A more general version is given in Theorem \ref{thm:cif.int}.

The next theorem will allow us to describe the Hausdorff dimension of certain objects we study.  For $m\in\Z$ define the set
\begin{align*}
\CI_m
&=\bigl\{s:\exists(n,t)\in\Z\times\R,n> m,t>s,\text{ and a geodesic from $(m,s)$ to $(n,t)$}\\
&\qquad\qquad\qquad\text{that goes through $(m+1,s)$}\bigr\}.
\end{align*}
The acronym $\CI$ stands for ``competition interface.''  
For each $m\in\mathbb{Z}$, the set $\CI_m$ consists of those points $s$ for which $(m,s)$ is the starting point of a competition interface.

\begin{thm}[Theorem 2.10 in \cite{Sep-Sor-23-pmp}]\label{CI-Hausdorff} There exists an event $\Omnew{Omega6}\in\kS$ with $\Omref{Omega6}\subset\Omref{Omega5}$, $\P(\Omref{Omega6})=1$, and such that for all $\w\in\Omref{Omega6}$ and $m\in\Z$, the set $\CI_m$ does not contain any rational point, has Hausdorff dimension $\half$, and is dense in $\R$. \end{thm}

\section{Further geodesics properties}\label{sec:more} 
In this section we present three preliminary results that will be used in the sequel and may also be of independent interest. 
Their proofs are deferred to Section \ref{sec:aux}.
The first result states that, almost surely, there are no geodesics with two consecutive vertical edges. 

\begin{prop}\label{nodouble}
There exists an event $\Omnewp{Omega'1}\in\kS$ such that $\P(\Omref{Omega'1})=1$ and for any $\w\in\Omref{Omega'1}$ neither of the following occurs: 
\begin{enumerate} [label={\rm(\alph*)}, ref={\rm\alph*}]   \itemsep=3pt 
\item\label{nodouble-a} There exist $(m,n)\in\Z^2$ and $(s,t)\in\R^2$ with $n\ge m+2$, $t>s$,
and $L_{(m,s),(n,t)}=L_{(m,s),(n-2,t)}$.
\item\label{nodouble-b} There exist $(m,n)\in\Z^2$ and $(s,t)\in\R^2$ with $n\ge m+2$, $t>s$,
and $L_{(m,s),(n,t)}=L_{(m+2,s),(n,t)}$.
\end{enumerate}
\end{prop}

The following is a consequence of the above proposition. 

\begin{lem}\label{geo-split}
	There exists an event $\Omnewp{Omega'2}\in\kS$ such that $\Omref{Omega'2}\subset\Omref{Omega1}\cap\Omref{Omega'1}$, $\P(\Omref{Omega'2})=1$, and for any $\w\in\Omref{Omega'2}$, $\tht>0$, $\sigg\in\{-,+\}$, and any $(m,t)\in\Z\times\R$, $ \geo\from{(m,t)}\dir{R}{\tht}{\sig} $ and $ \geo\from{(m,t)}\dir{L}{\tht}{\sig} $ either separate immediately or are equal. 
\end{lem}

The last result of this section concerns bi-infinite geodesics. For each $m\in\Z$ and $z\in\R$, the paths $\{(m,t):t\in\R\}$ and $\{(n,s):n\in\Z\}$ are degenerate bi-infinite geodesics. 

\begin{defn}\label{lrm-llm}
 An infinite geodesic is said to be \emph{locally rightmost} (resp.\ \emph{locally leftmost})  if it is the rightmost (resp.\ leftmost) geodesic between any pair of points on the geodesic.    
\end{defn}

\begin{thm}\label{nobiinf}
	There exists an event $\Omnew{Omega8}\in\kS$ such that $\Omref{Omega8}\subset\Omref{Omega6}\cap\Omref{Omega'2}$, $\P(\Omref{Omega8})=1$, and for any $\w\in\Omref{Omega8}$, there are no non-degenerate bi-infinite locally rightmost or locally leftmost  geodesics.
\end{thm}

This is a stronger version of Theorem 3.1(iv) in \cite{Sep-00-aihp},  which states that for any given $\tht>0$ there exists a full $\P$-probability event $\Omega_\tht$ (possibly depending on $\tht$) on which there are no bi-infinite geodesics going in direction $\tht$. Our method of proof relies on the uniqueness of finite geodesics, which necessitates the assumption of local extremality. For more on this,  see Remark \ref{rk:lrmllm}.

\begin{rmk}\label{Omega}
At this point, we have a full $\P$-probability event $\Omref{Omega8}$ on which all the results of Sections \ref{sec:previous} and \ref{sec:more} hold. 
All of our subsequent results will hold on this event. Thus, any instance of $\w$ mentioned in the remaining text should be understood as belonging to $\Omref{Omega8}$.
\end{rmk}

\section{Analytical characterization of the instability graph}\label{def:Bus:analytic}

In this section, we define instability points and the instability graph though the Busemann process \eqref{Bus-proc}.  We begin by defining the instability edges. (The notion of a point of increase was recalled before Definition \ref{def:inst}.)

\begin{defn}\label{def:edge.int}
	For $\tht\ge\gamma>0$, $m\in\Z$, and $t\in\R$, we call the vertical closed interval $[(m-\half,t),(m+\half,t)]$ a $[\gamma,\tht]$-\emph{instability edge} if
	\begin{equation}
			s\mapsto\bus{\gamma-}{}(m,0,m,s)-\buse{+}{m}{0}{m}{s} \label{inst-edge}
	\end{equation}
	has a point of increase at $t$. 
 \end{defn}

By the cocycle, monotonicity, and continuity properties in Theorem \ref{thm:B},
the function in \eqref{inst-edge} is continuous and nondecreasing. 

Next, we define instability points, also using Busemann functions.

 \begin{defn}\label{def:instpt.int}
	For $\tht\ge\gamma>0$, $m\in\Z$, and $t\in\R$, we say that $(m+\half, t)$ is a \emph{proper $[\gamma,\tht]$-instability point} if $\bus{\gamma-}{}(m,t,m+1,t)<\buse{+}{m}{t}{m+1}{t}$.
\end{defn}

By the monotonicity and continuity properties in Theorem \ref{thm:B},
the function 
\[t\mapsto\buse{+}{m}{t}{m+1}{t}-\bus{\gamma-}{}(m,t,m+1,t)\] 
is continuous and nonnegative. Consequently, the set of proper $[\gamma,\tht]$-instability points on a horizontal level $m$ is an open set and hence a countable union of disjoint open intervals.

\begin{defn}\label{def:IG.int}
    For $\tht\ge\gamma>0$, the $[\gamma,\tht]$-\emph{instability graph} $\IG{[\gamma,\tht]}=\IG{[\gamma,\tht]}(\w)$ is the union of all the $[\gamma,\tht]$-instability edges and the closure of the set of proper $[\gamma,\tht]$-instability points. 
\end{defn}

\begin{defn}\label{def:pt.int}   
For $\tht\ge\gamma>0$, $m\in\Z$, and $t\in\R$, the point $(m+\half,t)$ is a $[\gamma,\tht]$-\emph{instability point} if it belongs to $\IG{[\gamma,\tht]}$. A $[\gamma,\tht]$-instability point that is not proper is called \emph{improper}.
\end{defn}

In the rest, unless we are dealing with multiple instability graphs simultaneously, we will often simplify our language by omitting the explicit mention of $[\gamma,\tht]$ and instead refer to ``instability edges'', ``instability points'', and other instability objects, when addressing $[\gamma,\tht]$-instability edges, $[\gamma,\tht]$-instability points, and other $[\gamma,\tht]$-instability objects, respectively.  When $\gamma=\tht$, we write $\IG{\tht}$ instead of $\IG{[\tht,\tht]}$ and will refer to $\tht$-instability objects.

\begin{rmk}
    When $\gamma=\tht\not\in\baddir$, $\IG{[\gamma,\tht]}=\varnothing$ and all the definitions and results on the instability graph are vacuous. When $\gamma=\tht\in\baddir$, the definitions in this section match the ones in Definition \ref{def:inst}. Consequently, the results presented in Section \ref{sec:results} follow 
    from the more general results we present throughout the rest of the paper by specializing to the case $\gamma=\tht\in\baddir$. See Lemmas \ref{IGint=union} and \ref{IG-approx} below for further connections between $\IG{\tht}$ and $\IG{[\gamma,\delta]}$.
\end{rmk}

The next theorem gives a characterization of instability points.

\begin{thm}\label{thm:instpt}
For all $\w\in\Omref{Omega8}$, $\tht\ge\gamma>0$, $m\in\Z$, and $t\in\R$, the following are equivalent.
\begin{enumerate}[label={\rm(\roman*)}, ref={\rm\roman*}]   \itemsep=3pt 
\item\label{instpt.a} $(m+\half,t)$ is a $[\gamma,\tht]$-instability point.
\item\label{instpt.b} For any $r'<t$ and $r''>t$ we have
\begin{align}\label{bus:cond}
\bus{\gamma}{-}(m+1,r',m,r'')>\buse{+}{m+1}{r'}{m}{r''}.
\end{align}
\item\label{instpt.c} For any $r'<t$ and $r''>t$ we have
\begin{align}\label{bus:cond2}
\begin{split}
&\bus{\gamma}{-}(m+1,r',m+1,r'')>\buse{+}{m+1}{r'}{m+1}{r''}\\
&\text{or}\quad\bus{\gamma}{-}(m+1,r'',m,r'')>\buse{+}{m+1}{r''}{m}{r''}.
\end{split}
\end{align}
\item\label{instpt.d} For any $r'<t$ and $r''>t$ we have
\begin{align}\label{bus:cond3}
\begin{split}
&\bus{\gamma}{-}(m,r',m,r'')>\buse{+}{m}{r'}{m}{r''}\\
&\text{or}\quad\bus{\gamma}{-}(m+1,r',m,r')>\buse{+}{m+1}{r'}{m}{r'}.
\end{split}
\end{align}
\end{enumerate}
\end{thm}

\begin{proof}
The equivalence of (\ref{instpt.b}-\ref{instpt.d}) is a direct consequence of the cocycle and monotonicity properties in Theorem \ref{thm:B}. 

	Suppose next that $(m+\half,t) $ is an instability point and take arbitrary $r'<t<r''$. Consider the cases that can happen at the instability point.
    If there is an instability edge between $(m+\frac12,t)$ and $(m+\frac32,t)$, then $\bus{\gamma}{-}(m+1,r',m+1,r'')>\buse{+}{m+1}{r'}{m+1}{r''}$. 
    Monotonicity implies $\bus{\gamma}{-}(m+1,r'',m,r'')\ge\buse{+}{m+1}{r''}{m}{r''}$. Then the cocycle property implies
        \begin{align*}
            \bus{\gamma}{-}(m+1,r',m,r'')
            &=\bus{\gamma}{-}(m+1,r',m+1,r'')+\bus{\gamma}{-}(m+1,r'',m,r'')\\
            &>\buse{+}{m+1}{r'}{m+1}{r''}+\buse{+}{m+1}{r''}{m}{r''}\\
            &=\buse{+}{m+1}{r'}{m}{r''}.
            \end{align*}
    The case where there is an instability edge between $(m+\frac12,t)$ and $(m-\frac12,t)$ is similar, using $\bus{\gamma}{-}(m+1,r',m,r'')=\bus{\gamma}{-}(m+1,r',m,r')+\bus{\gamma}{-}(m,r',m,r'')$. 
    
    Suppose now that there is a sequence of proper instability points $(m+\frac12,r_n)$ with $r_n\to t$. (This includes the case where $r_n=t$ for all $n$.) Then we have for each $n$,
    $\bus{\gamma}{-}(m+1,r_n,m,r_n)>\buse{+}{m+1}{r_n}{m}{r_n}$. But then for $n$ large enough, $r'<r_n<r''$ and hence
         \begin{align*}
             &\bus{\gamma}{-}(m+1,r',m,r'')\\
             &\qquad=\bus{\gamma}{-}(m+1,r',m+1,r_n)+\buse{\gamma}{-}(m+1,r_n,m+1,r')+\bus{\gamma}{-}(m+1,r'',m,r')\\
             &\qquad>\buse{+}{m+1}{r'}{m+1}{r_n}+\buse{+}{m+1}{r_n}{m+1}{r'}+\buse{+}{m+1}{r''}{m}{r'}\\
             &\qquad=\buse{+}{m+1}{r'}{m}{r''}.
             \end{align*}
    We thus see that \eqref{bus:cond} is satisfied in all cases.
    
    Suppose now that $(m+\frac12,t)$ is not an instability point. Then there is no instability edge between $(m+\frac12,t)$ and $(m+\frac32,t)$ and, therefore, $\bus{\gamma}{-}(m+1,r',m+1,r'')=\buse{+}{m+1}{r'}{m+1}{r''}$ for all $r'<t$ and $r''>t$ close enough to $t$. Furthermore, there are no proper instability points converging to $(m+\frac12,t)$. This implies that  $\bus{\gamma}{-}(m+1,r'',m,r'')=\buse{+}{m+1}{r''}{m}{r''}$ for all $r''>t$ close enough to $t$. Then \eqref{bus:cond2} fails to hold. Thus, we have shown that if \eqref{instpt.a} does not hold, then \eqref{instpt.c}, which is equivalent to \eqref{instpt.b}, does not hold. The proof of the lemma is complete.
\end{proof}

\begin{defn}\label{def:II}
     For $\tht\ge\gamma>0$, $m\in\Z$, and an interval $I\subset\R$, $\{(m+\half,r):r\in I\}$ is a  $[\gamma,\tht]$-\emph{instability interval} if it is a subset of $\IG{[\gamma,\tht]}$. It is a proper $[\gamma,\tht]$-instability interval if $(m+\half,r)$ is a \emph{proper} $[\gamma,\tht]$-instability point for all $r\in I$. 
\end{defn}

\begin{defn}
    For $\tht\ge\gamma>0$, $m\in\Z$, and $t\in\R$, $(m+\half,t)$ is called a \emph{double-edge $[\gamma,\tht]$-instability point} if both $[(m-\half,t),(m+\half,t)]$ and $[(m+\half,t),(m+\frac32,t)]$ are $[\gamma,\tht]$-instability edges.
\end{defn}

\begin{lem}\label{lm:PIP}
    The following holds for all $\w\in\Omref{Omega8}$,
    $\tht\ge\gamma>0$, $m\in\Z$, and $s<t$. Suppose $\{(m+\half,r):s<r<t\}\subset\IG{[\gamma,\tht]}$. Then for each $r\in(s,t)$,  $(m+\half,r)$ is either a proper $[\gamma,\tht]$-instability point or a double-edge $[\gamma,\tht]$-instability point.
\end{lem}

\begin{proof}
    Suppose $(m+\half,r)$ is not a proper $[\gamma,\tht]$-instability point. Then it must be that $\bus{\gamma}{-}(m,r,m+1,r)=\buse{+}{m}{r}{m+1}{r}$. 
    Take $s<r'<r<t$. Since $(m+\half,(r'+r)/2)$ is an instability point, Theorem \ref{thm:instpt}\eqref{instpt.c} (with $r''=r$) implies that $\bus{\gamma}{-}(m+1,r',m+1,r)>\buse{+}{m+1}{r'}{m+1}{r}$. This shows that $r$ is a point of increase for $\bus{\gamma}{-}(m+1,0,m+1,\abullet)-\buse{+}{m+1}{0}{m+1}{\abullet}$ and, therefore, $[(m+\half,r),(m+\frac32,r)]$ is an instability edge.
    Similarly, taking $s<r<r''<t$ and using Theorem \ref{thm:instpt}\eqref{instpt.d} with $r'=r$ shows that 
    $r$ is a point of increase for $\bus{\gamma}{-}(m,0,m,\abullet)-\buse{+}{m}{0}{m}{\abullet}$ and, therefore, $[(m-\half,r),(m+\half,r)]$ is an instability edge.
\end{proof}

\begin{rmk}\label{rk:(gamma,tht)}
    When $\tht>\gamma>0$, replacing $\gamma-$ everywhere with $\gamma+$ and changing $\tht+$ everywhere to $\tht-$ gives another instability graph, denoted by $\IG{(\gamma,\tht)}$. Then all of our results about $\IG{[\gamma,\tht]}$ hold for $\IG{(\gamma,\tht)}$.
\end{rmk}

\section{Geometric characterization of the instability graph}\label{def:Bus:geometric}
Now we develop equivalent definitions of instability through semi-infinite geodesics. 
We begin by looking at instability edges.

\begin{lem}
\label{lm:edge}
    For all $\w\in\Omref{Omega8}$, $\tht\ge\gamma>0$,
	 and $(m,t)\in\Z\times\R$, $\bus{\gamma}{-}(m,0,m,\abullet)-\buse{+}{m}{0}{m}{\abullet}$ has a point of increase at $t$ if and only if for any $r'<t$ and $r''>t$,
	\begin{align}\label{R-&L+}
	\geo\from{(m,r')}\dir{R}{\gamma}{-} \cap \geo\from{(m,r'')}\dir{L}{\tht}{+} = \varnothing.
	\end{align}
\end{lem}

This comes straight from Theorem \ref{thm:Bus-geo}. 
The above result requires considering a whole interval around $t$. The following result gives a criterion only in terms of geodesics out of $(m,t)$.

\begin{lem}\label{geodefinstabedge1}
    For all $\w\in\Omref{Omega8}$, $\tht\ge\gamma>0$, and $(m,t)\in\Z\times\R$, $\bus{\gamma}{-}(m,0,m,\abullet)-\buse{+}{m}{0}{m}{\abullet}$ has a point of increase at $t$ if and only if $\geo\from{(m,t)}\dir{R}{\tht}{+}$ and $\geo\from{(m,t)}\dir{L}{\gamma}{-}$ separate immediately and never touch again.
\end{lem}

\begin{figure}[ht!]
\begin{center}

\begin{tikzpicture}[>=latex, scale=0.6]
\begin{scope}
\draw(0,0)--(6,0);
\draw(0,2)--(6,2);
\draw[line width=2pt,color=Rp](1,0)--(1.5,0)--(1.5,2)--(2,2)--(2,3);
\draw[line width=2pt,color=Lm](2.25,0)--(2.25,2)--(3,2)--(3,3);
\draw[line width=2pt,color=Rp](2.25,0)--(3.5,0)--(3.5,2.05)--(5,2.05);
\draw[line width=2pt,color=Lm](3.8,0)--(4.25,0)--(4.25,1.95)--(5,1.95);

\shade[ball color=pt](1,0)circle(1.5mm);
\draw(1,0)node[below]{$r'$};
\shade[ball color=pt](2.25,0)circle(1.5mm);
\draw(2.25,0)node[below]{$t$};
\shade[ball color=pt](3.8,0)circle(1.5mm);
\draw(3.8,0)node[below]{$r''$};
\draw(6,0)node[right]{$m$};
\draw(6,2)node[right]{$m+1$};
\end{scope}

\begin{scope}[shift={(9,0)}]
\draw(0,0)--(6,0);
\draw(0,2)--(6,2);
\draw[line width=2pt,color=Rp](1,0)--(1.5,0)--(1.5,2.15)--(4,2.15);
\draw[line width=2pt,color=Lm](1.9,0)--(1.9,2.05)--(4,2.05);
\draw[line width=2pt,color=Rp](2.02,0)--(2.02,1.95)--(4,1.95);
\draw[line width=2pt,color=Lm](2.4,0)--(2.75,0)--(2.75,1.85)--(4,1.85);

\shade[ball color=pt](1,0)circle(1.5mm);
\draw(1,0)node[below]{$r'$};
\shade[ball color=pt](1.97,0)circle(1.5mm);
\draw(1.97,-0.1)node[below]{$t$};
\shade[ball color=pt](2.4,0)circle(1.5mm);
\draw(2.5,0)node[below]{$r''$};
\shade[ball color=gray](4,2)circle(2mm);
\draw(4,0)node[below]{$T$};
\draw(4,2.1)node[above]{$z$};
\draw[line width=1pt](4,-0.15)--(4,0.15);
\draw(6,0)node[right]{$m$};
\draw(6,2)node[right]{$m+1$};
\end{scope}

\begin{scope}[shift={(18,0)}]
\draw(0,0)--(6,0);
\draw(0,1)--(6,1);
\draw(0,3)--(6,3);
\draw[line width=2pt,color=Rp](1,0)--(1,1.05)--(2,1.05)--(2,1.55);
\draw[line width=2pt,color=Rp,dashed](2,1.5)--(3,2.5);
\draw[line width=2pt,color=Rp](3,2.5)--(3,3.05)--(4.5,3.05);
\draw[line width=2pt,color=Lm](1.5,0)--(1.5,0.95)--(2.1,0.95)--(2.1,1.45);
\draw[line width=2pt,color=Lm,dashed](2.07,1.42)--(3.1,2.5);
\draw[line width=2pt,color=Lm](3.1,2.5)--(3.1,2.95)--(4.5,2.95);
\draw[line width=2pt,color=Rp](1.5,0.05)--(3.5,0.05)--(3.5,1.05)--(4,1.05)--(4,1.55);
\draw[line width=2pt,color=Rp,dashed](4,1.55)--(4.5,2.55);
\draw[line width=2pt,color=Rp](4.5,2.55)--(4.5,3.05);
\draw[line width=2pt,color=Lm](2,-0.05)--(3.6,-0.05)--(3.6,0.95)--(4.1,0.95)--(4.1,1.5);
\draw[line width=2pt,color=Lm,dashed](4.1,1.47)--(4.57,2.47);
\draw[line width=2pt,color=Lm](4.62,2.47)--(4.62,3);

\shade[ball color=pt](1,0)circle(1.5mm);
\draw(1,0)node[below]{$r'$};
\shade[ball color=pt](1.5,0)circle(1.5mm);
\draw(1.5,-0.1)node[below]{$t$};
\shade[ball color=pt](2,0)circle(1.5mm);
\draw(2.1,0)node[below]{$r''$};
\shade[ball color=gray](4.6,3)circle(2mm);
\draw(4.6,3.1)node[above]{$z$};
\draw(6,0)node[right]{$m$};
\draw(6,1)node[right]{$m+1$};
\end{scope}

\end{tikzpicture}
\end{center}

\caption{\small An illustration of the proof of Lemma \ref{geodefinstabedge1}. Left: The proof of the first direction. The two geodesics in the middle are $\geo\from{(m,t)}\dir{L}{\gamma}{-}$ and $\geo\from{(m,t)}\dir{R}{\tht}{+}$. They separate immediately and do not meet again, thus separating the two outer geodesics $\geo\from{(m,r')}\dir{R}{\tht}{+}$ and $\geo\from{(m,r'')}\dir{L}{\gamma}{-}$. Middle: The first case in the second direction. Right: The second and third cases in the second direction.}
\label{fig:geodefinstabedge1}
\end{figure}
  
\begin{proof}
	Suppose that the lemma's geodesic condition is true. Take $r'<t$ and $r''>t$. 
	Then, by the first inequality in \eqref{geomono}, the geodesic $\geo\from{(m,r'')}\dir{L}{\tht}{+}$ either instantly coalesces with or stays to the right of $\geo\from{(m,t)}\dir{R}{\tht}{+}$. Likewise, $\geo\from{(m,r')}\dir{R}{\gamma}{-}$ must stay to the left of $\geo\from{(m,t)}\dir{L}{\gamma}{-}$. Thus, by montonicity, $\geo\from{(m,r')}\dir{R}{\gamma}{-} \cap\geo\from{(m,r'')}\dir{L}{\tht}{+} = \varnothing$ and  by Lemma \ref{lm:edge}, $t$ is a point of increase of $\bus{\gamma}{-}(m,0,m,\abullet)-\buse{+}{m}{0}{m}{\abullet}$. See the left panel in Figure \ref{fig:geodefinstabedge1}.
 	
	For the other direction, suppose that the lemma's geodesic condition is false. 
    Then $\geo\from{(m,t)}\dir{L}{\gamma}{-}$ and $\geo\from{(m,t)}\dir{R}{\tht}{+}$ 
  must do one of the following: (1) both go up one level and then both proceed rightwards, (2) split instantly and reconvene again at some later point $\bfz \in \Z\times\R$, or (3) go rightwards together immediately; they cannot proceed vertically together for more than one level, nor can they split at $(m+1,t)$, as both scenarios would make $\geo\from{(m,t)}\dir{L}{\gamma}{-}$ take two consecutive vertical steps, which does not happen by Proposition \ref{nodouble}. See the middle panel in Figure \ref{fig:geodefinstabedge1}.

  For the first case, suppose that $\geo\from{( m,t )}\dir{R}{\tht}{+}$ and $\geo\from{(m,t)}\dir{L}{\gamma}{-}$ both proceed upwards to $(m+1,t)$ and then proceed together rightward on level $m+1$ to $(m+1,T)$ for $T>t$. In that case, by the montonicity \eqref{geomono} and continuity \eqref{geocont}, for $r''>t$ and $r'<t$ sufficiently close to $t$, $\geo\from{(m,r'')}\dir{L}{\tht}{+}$ and $\geo\from{(m,r')}\dir{R}{\gamma}{-}$ must both pass through $(m+1,T)$, 
which by Lemma \ref{lm:edge} implies that $t$ is not a point of increase for $\bus{\gamma}{-}(m,0,m,\abullet)-\buse{+}{m}{0}{m}{\abullet}$.
	
    In the second and third cases,  $\geo\from{(m,t)}\dir{R}{\tht}{+}$ must immediately proceed laterally on level $m$. Then for $r''>t$ sufficiently close to $t$, $\geo\from{(m,r'')}\dir{L}{\tht}{+} \subset \geo\from{(m,t)}\dir{R}{\tht}{+}$. By Proposition \ref{nodouble}, $\geo\from{(m,t)}\dir{L}{\gamma}{-}$ must either proceed laterally on level $m$ or $m+1$. Regardless, on that level $M \in \{m,m+1\}$, $\geo\from{(M,t)}\dir{L}{\gamma}{-}$ is the limit of $\geo\from{(M,r')}\dir{R}{\gamma}{-}$ as $r'\nearrow t$ and thus for $r'<t$ sufficiently near $t$, $\geo\from{(M,r')}\dir{R}{\gamma}{-}$ must proceed laterally to $(M,t)$, coalescing there with $\geo\from{(M,t)}\dir{L}{\gamma}{-}$. By monotonicity, taking such an $r'$ yields that $ \geo\from{(m,r')}\dir{R}{\gamma}{-}$ must also coalesce with $ \geo\from{(m,t)}\dir{L}{\gamma}{-}$ at $(M,t)$. Thus, for $r' <t$ and $r''>t$ sufficiently close to $t$, $\bfz \in \geo\from{(m,r')}\dir{R}{\gamma}{-} \cap  \geo\from{(m,r'')}\dir{L}{\tht}{+}  $ and, according to Lemma \ref{lm:edge}, $t$ is not a point of increase for $\bus{\gamma}{-}(m,0,m,\abullet)-\buse{+}{m}{0}{m}{\abullet}$
    See the right panel in Figure \ref{fig:geodefinstabedge1}.
\end{proof}

\begin{lem}\label{edge-Hausdorff}
    The following holds for all $\w\in\Omref{Omega8}$. Take $\tht\ge\gamma>0$. If $\tht=\gamma$, then assume $\tht\in\baddir$.  For all $m\in\Z$, the set of points of increase of the function in \eqref{inst-edge} has a Hausdorff dimension $\half$. 
\end{lem}

\begin{proof}
    By Theorem \ref{baddir dense}, there exists a $\tht'\in[\gamma,\tht]\cap\baddir$. 
    By the monotonicity property in Theorem \ref{thm:B}\eqref{B:mono} we have, for any $\e>0$, 
    \[\bus{\gamma}{-}(m,s-\e,m,s+\e)\ge\bus{\tht'}{-}(m,s-\e,m,s+\e)\ge\bus{\tht'}{+}(m,s-\e,m,s+\e)\ge\bus{\tht}{+}(m,s-\e,m,s+\e).\]
    Therefore, the set in question contains the set of points of increase of the function 
    
    \[s\mapsto\bus{\tht'}{-}(m,0,m,s)-\bus{\tht'}{+}(m,0,m,s),\] which by Theorem \ref{B:increase} has a Hausdorff dimension of $\half$. On the other hand, if $s$ is a point of increase of \eqref{inst-edge}, then Lemma \ref{geodefinstabedge1} implies that $\geo\from{(m,s)}\dir{L}{\gamma}{-}$ goes up from $(m,s)$ and, therefore, $s\in\CI_m$, which by Theorem \ref{CI-Hausdorff} has a Hausdorff dimension of $\half$.
\end{proof}

We now turn to the instability points.
Theorem \ref{thm:Bus-geo} implies the following characterization of proper instability points through semi-infinite geodesics.

\begin{lem}\label{geodefproperIS}
    For all $\w\in\Omref{Omega8}$, $\tht\ge\gamma>0$,  $m\in\Z$, and $t\in\R$,
	$
	\buse{+}{m}{t}{m+1}{t}<\bus{\gamma}{-}(m,t,m+1,t),
	$
        if and only if
	$
	\geo\from{(m+1,t)}\dir{R}{\gamma}{-}\cap \geo\from{(m,t)}\dir{L}{\tht}{+} = \varnothing.
	$
\end{lem}

The next result characterizes all instability points. 

\begin{thm}\label{geo:instpt}
    For all $\w\in\Omref{Omega8}$, $\tht\ge\gamma>0$,  $m\in\Z$, and $t\in\R$, $(m+\half,t)$ is a $[\gamma,\tht]$-instability point 
        if and only if $\geo\from{(m,t)}\dir{R}{\tht}{+}$ and $\geo\from{(m+1,t)}\dir{L}{\gamma}{-}$ do not intersect.
\end{thm}

\begin{proof}
    Suppose first that $\geo\from{(m,t)}\dir{R}{\tht}{+}\cap\geo\from{(m+1,t)}\dir{L}{\gamma}{-}=\varnothing$. Take $r'<t$ and $r''>t$. Since $\geo\from{(m+1,r')}\dir{R}{\gamma}{-}$ must remain to the left of $\geo\from{(m+1,t)}\dir{L}{\gamma}{-}$ and $\geo\from{(m,r'')}\dir{L}{\tht}{+}$ must remain to the right of $\geo\from{(m,t)}\dir{R}{\tht}{+}$, we have $\geo\from{(m+1,r')}\dir{R}{\gamma}{-}\cap\geo\from{(m,r'')}\dir{L}{\tht}{+}=\varnothing$. 
    This and Theorem \ref{thm:Bus-geo} imply \eqref{instpt.b} in Theorem \ref{thm:instpt}, which by that theorem is equivalent to $(m+\half,t)$ being an instability point. 

For the reverse implication, suppose that 
$\geo\from{(m,t)}\dir{R}{\theta}{+}$ and $\geo\from{(m+1,t)}\dir{L}{\gamma}{-}$ do intersect.  
If $\geo\from{(m,t)}\dir{R}{\theta}{+}$ first goes to $(m+1,t)$ and the two geodesics then follow the same vertical path up to some level $n\ge m+1$ before splitting, then $\geo\from{(m,t)}\dir{L}{\gamma}{-}$ would necessarily pass through $(m+1,t)$ and take at least one additional vertical step with $\geo\from{(m+1,t)}\dir{L}{\gamma}{-}$.  
By Proposition \ref{nodouble}, this is impossible.  
Hence there must exist an integer $n>m$ and an $r>t$ such that
$(n,r)\in\geo\from{(m,t)}\dir{R}{\theta}{+}\cap\geo\from{(m+1,t)}\dir{L}{\gamma}{-}$.
Then by the limits \eqref{geocont}, for $r'<t$ and $r''>t$ close enough to $t$ we have $(n,r)\in\geo\from{(m+1,r')}\dir{R}{\theta}{-}\cap\geo\from{(m,r'')}\dir{L}{\theta}{+}$.
Together with Theorem \ref{thm:Bus-geo}, this implies that  
\eqref{instpt.b} in Theorem \ref{thm:instpt} does not hold.  
By that theorem, this is equivalent to $(m+\tfrac12,t)$ not being an instability point.
\end{proof}

\section{The structure of the instability graph}\label{sec:inst}

Having defined the instability graph, both analytically and geometrically, we turn to its properties. The first result states a certain monotonicity of the graph $\IG{[\gamma,\tht]}$, in $\gamma$ and $\tht$.

\begin{lem}\label{mono-inst-graph}
For all $\w\in\Omref{Omega8}$ and $\tht'\ge\tht\ge\gamma\ge\gamma'>0$,
$\IG{[\gamma,\tht]}\subset\IG{[\gamma',\tht']}$. 
\end{lem}

\begin{proof}
Suppose $[(m-\half,t),(m+\half,t)]$ is a $[\gamma,\tht]$-instability edge. Then for any $\e>0$, $\bus{\gamma}{-}(m,t-\e,m,t+\e)>\bus{\tht}{+}(m,t-\e,m,t+\e)$.
By the monotonicity property in Theorem \ref{thm:B} we have for any $\e>0$
\begin{align*}
    \bus{\gamma'}{-}(m,t-\e,m,t+\e)
    &\ge\bus{\gamma}{-}(m,t-\e,m,t+\e)\\
    &>\bus{\tht}{+}(m,t-\e,m,t+\e)\ge\bus{\tht'}{+}(m,t-\e,m,t+\e).
\end{align*}
Consequently, $t$ is a point of increase for $s\mapsto\bus{\gamma'}{-}(m,0,m,s)-\bus{\tht'}{+}(m,0,m,s)$ and $[(m-\half,t),(m+\half,t)]$ is also a $[\gamma',\tht']$-instability edge. Similarly, if $(m+\half,t)$ is a $[\gamma,\tht]$-instability point, then it is also a $[\gamma',\tht']$-instability point.
\end{proof}

Next, we show that the instability graph $\IG{[\gamma,\tht]}$ is northeast and southwest bi-infinite.

\begin{lem}\label{biinf:geo}
	For all $\w\in\Omref{Omega8}$, $\tht\ge\gamma>0$, and any instability point $\bfx\dual\in \IG{[\gamma,\tht]}$ there exist an up-right path and a down-left path, both starting at $\bfx\dual$ and moving along the instability graph $\IG{[\gamma,\tht]}$.
\end{lem}

\begin{proof}
	Let $\bfx\dual = (m+\half,t)$ be an instability point in $\IG{[\gamma,\tht]}$. Then, by Theorem \ref{thm:instpt}\eqref{instpt.c}, for any $ r'<t $ and $ r''>t $, 
	\begin{align*}
		\begin{split}
			&\bus{\gamma}{-}(m+1,r',m+1,r'')>\buse{+}{m+1}{r'}{m+1}{r''}\\
			&\text{or}\quad\bus{\gamma}{-}(m+1,r'',m,r'')>\buse{+}{m+1}{r''}{m}{r''}.
		\end{split}
	\end{align*}
    If there is no instability edge going up from $\bfx\dual$, then $\bus{\gamma}{-}(m+1,r',m+1,r'')=\buse{+}{m+1}{r'}{m+1}{r''}$ for all $r'<r$ and $r''>r$ close enough to $r$. Then the second inequality in the above display must hold for all such $r''$. This implies  that $(m+\half,r'')$ is an instability point for all $r''>r$ close enough to $r$.

    The argument for moving in the south-west direction is similar.
\end{proof}

\begin{lem}\label{nodoubleinstab}
	For all $\w\in\Omref{Omega8}$ and $\tht\ge\gamma>0$   there are no double-edge instability points in $\IG{[\gamma,\tht]}$. 
\end{lem}

\begin{proof}
    By Lemma \ref{geodefinstabedge1}, if both $[(m-\half,t),(m+\half,t)]$ and $[(m+\half,t),(m+\frac32,t)]$ are instability edges in $\IG{[\gamma,\tht]}$, then $\geo\from{(m,t)}\dir{L}{\gamma}{-}$ has to go through $(m+2,t)$, which cannot happen by Proposition \ref{nodouble}.
\end{proof}

Together, the above lemma and Lemma \ref{lm:PIP} imply that the interior of an instability interval consists of proper instability points.

\begin{lem}\label{onlypropinprop}
	The following holds for all $\w\in\Omref{Omega8}$, $\tht\ge\gamma>0$,  $m\in\Z$, and $s<t$. Suppose $\{(m+\half,r):s<r<t\}\subset\IG{[\gamma,\tht]}$. Then for each $r\in(s,t)$,  $(m+\half,r)$ is a proper $[\gamma,\tht]$-instability point. 
\end{lem}

\begin{lem}\label{improperendpts}
	For all $\w\in\Omref{Omega8}$ and $\tht\ge\gamma>0$, the only improper $[\gamma,\tht]$-instability points are the left and right endpoints of maximal proper $[\gamma,\tht]$-instability intervals.
\end{lem}

\begin{proof}
	Suppose that $(m+\half,s)\in\IG{[\gamma,\tht]}$ but that $(m+\half,s)$ is not a proper $[\gamma,\tht]$-instability point. By Lemma \ref{nodoubleinstab}, $(m+\half,s)$ cannot be a double-edge instability point. Then by Lemma \ref{biinf:geo}, either $[(m+\half,s),(m+\half,t)]$ is an instability interval or $[(m+\half,r),(m+\half,s)]$ is an instability interval, for some $t>s$ or $r<s$. It cannot be both because then Lemma \ref{onlypropinprop} would imply that $(m+\half,s)$ is proper. Thus, we have shown that $(m+\half,s)$ is an endpoint of an instability interval.
\end{proof}

The next lemma says that the instability graph extends infinitely far to the left and to the right.

\begin{lem}\label{inf many edges}
	Take $\w\in\Omref{Omega8}$, $\tht\ge\gamma>0$, and $m\in\Z$. If $\tht=\gamma$, then assume that $\tht\in\baddir$. Then
    \[\sup\bigl\{s:\bigl[(m-\tfrac12,s),(m+\tfrac12,s)\bigr]\in\IG{[\gamma,\tht]}\bigr\}=\infty\quad\text{and}\quad
    \inf\bigl\{s:\bigl[(m-\tfrac12,s),(m+\tfrac12,s)\bigr]\in\IG{[\gamma,\tht]}\bigr\}=-\infty.\]
\end{lem}

\begin{proof}
    When $\tht=\gamma\in\baddir$, take $\kappa=\tht$.
    If $\gamma<\tht$, then Theorem \ref{baddir dense} implies the existence of a $\kappa\in[\gamma,\tht]\cap\baddir$. 
    The monotonicity \eqref{B:mono-eqn} implies that for all $t\ge0$
    \[\bus{\gamma-}{}(m,0,m,t)-\buse{+}{m}{0}{m}{t}\ge\bus{\kappa-}{}(m,0,m,t)-\bus{\kappa}{+}(m,0,m,t),\]
    with the reversed inequality for all $t\le0$.
    The claim now follows from Definition \ref{def:edge.int} and Theorem \ref{B:increase} (applied with the parameter $\kappa$).
\end{proof}

\begin{lem}\label{improp=endpts}
		The following holds for all $\w\in\Omref{Omega8}$, $\tht\ge\gamma>0$, $m\in\Z$, and $s<r_0<t$. Suppose $\{(m+\half,r):s<r<r_0\}$ and $\{(m+\half,r):r_0<r<t\}$ are proper $\tht$-instability intervals. Then $(m+\half,r_0)$ is a proper $\tht$-instability point. Consequently, if two maximal proper instability intervals are not identical, then their closures are disjoint.
\end{lem}

\begin{proof}
    Since instability intervals are closed we get that $\{(m+\half,r):s\le r\le t\}$ is an instability interval. By Lemma \ref{onlypropinprop}, $(m+\half,r_0)$ is a proper instability point.
\end{proof}

The following is a summary of the above results. The case where $\gamma=\tht\in\baddir$ gives Theorem \ref{inst:summary}, except for its part \eqref{inst:summary.directed}, which follows from Lemma \ref{lm:directed} below. 

\begin{thm}\label{inst:summary.int}
    Take $\w\in\Omref{Omega8}$ and $\tht\ge\gamma>0$. If $\tht=\gamma$, then assume $\tht\in\baddir$.
    \begin{enumerate}  [label={\rm(\roman*)}, ref={\rm\roman*}]   \itemsep=3pt 
    \item\label{inst:summary.int.biinf} The instability graph is bi-infinite: For any instability point $\bfx\dual\in \IG{[\gamma,\tht]}$ there exist an infinite up-right path and an infinite down-left path, both starting at $\bfx\dual$ and moving along the instability graph $\IG{[\gamma,\tht]}$.
    \item\label{inst:summary.int.nodouble} There are no double-edge instability points in $\IG{[\gamma,\tht]}$.  \item\label{inst:summary.int.improper} The only improper $[\gamma,\tht]$-instability points are the boundaries of maximal $[\gamma,\tht]$-instability intervals.
     \item\label{inst:summary.int.unbounded} The instability graph extends infinitely far to the left and to the right: For any $m\in\Z$
     \[\sup\bigl\{s:\bigl[(m-\tfrac12,s),(m+\tfrac12,s)\bigr]\in\IG{[\gamma,\tht]}\bigr\}=\infty\quad\text{and}\quad
    \inf\bigl\{s:\bigl[(m-\tfrac12,s),(m+\tfrac12,s)\bigr]\in\IG{[\gamma,\tht]}\bigr\}=-\infty.\]
    \item\label{inst:summary.int.I} Take $m\in\Z$ and $s<t$. The interval $\{(m+\half,r):s\le r\le t\}$ is a maximal $[\gamma,\tht]$-instability interval if and only if $(m+\half,r)$ is a proper $[\gamma,\tht]$-instability point for all $r\in(s,t)$ and $(m+\half,s)$ and $(m+\half,t)$ are improper $[\gamma,\tht]$-instability points.
    \item\label{inst:summary.int.Hausdorff} For any $m\in\Z$, the set of instability points $(m+\half,s)\in\IG{[\gamma,\tht]}$ from which descends an instability edge has a Hausdorff dimension of $\half$. 
    \end{enumerate}
\end{thm}

\begin{figure}[ht!]
\begin{center}

\begin{tikzpicture}[>=latex, scale=0.6]

\draw(0,0)--(12,0);
\draw[dashed](0,1)--(12,1);
\draw(0,2)--(12,2);
\draw[dashed](0,3)--(12,3);
\draw(0,4)--(12,4);

\draw[line width=2pt,color=Inst](1,-1)--(1,1)--(5,1)--(5,3)--(11,3);
\draw[line width=2pt,color=Inst](5,1)--(8,1)--(8,3);
\draw[line width=2pt,color=Inst](1.15,-1)--(1.15,1);
\draw[line width=2pt,color=Inst](1.3,-1)--(1.3,1);
\draw[line width=2pt,color=Inst](1.5,-1)--(1.5,1);
\draw[line width=2pt,color=Inst](2.5,-1)--(2.5,1);
\draw[line width=2pt,color=Inst](2.65,-1)--(2.65,1);
\draw[line width=2pt,color=Inst](2.8,-1)--(2.8,1);
\draw[line width=2pt,color=Inst](5.2,1)--(5.2,3);
\draw[line width=2pt,color=Inst](5.35,1)--(5.35,3);
\draw[line width=2pt,color=Inst](5.5,1)--(5.5,3);
\draw[line width=2pt,color=Inst](6.1,1)--(6.1,3);
\draw[line width=2pt,color=Inst](6.25,1)--(6.25,3);
\draw[line width=2pt,color=Inst](6.45,1)--(6.45,3);
\draw[line width=2pt,color=Inst](7.5,1)--(7.5,3);
\draw[line width=2pt,color=Inst](7.65,1)--(7.65,3);
\draw[line width=2pt,color=Inst](7.85,1)--(7.85,3);
\draw[line width=2pt,color=Inst](8.95,5)--(10.5,5);
\draw[line width=2pt,color=Inst,->](10.5,5)--(11.5,6);
\draw[line width=2pt,color=Inst,->](11,3)--(13,5);
\draw[line width=2pt,color=Inst,->](0,-1)--(-1,-2);
\draw[line width=2pt,color=Inst](9,3)--(9,5);
\draw[line width=2pt,color=Inst](9.15,3)--(9.15,5);
\draw[line width=2pt,color=Inst](9.35,3)--(9.35,5);
\draw[line width=2pt,color=Inst](0,-1)--(7,-1);
\draw[line width=2pt,color=Inst](7,-1)--(7,1);
\draw[line width=2pt,color=Inst](6.85,-1)--(6.85,1);
\draw[line width=2pt,color=Inst](6.65,-1)--(6.65,1);
\draw[line width=2pt,color=Inst](6.5,-1)--(6.5,1);
\draw[line width=2pt,color=Inst](6.35,-1)--(6.35,1);
\draw[line width=2pt,color=Inst](4.15,-1)--(4.15,1);
\draw[line width=2pt,color=Inst](4,-1)--(4,1);
\draw[line width=2pt,color=Inst](3.85,-1)--(3.85,1);
\draw[line width=2pt,color=Inst](3.6,-1)--(3.6,1);

\end{tikzpicture}
\end{center}

\caption{\small An illustration of Theorem \ref{inst:summary.int}.}
\label{fig:inst-summary}
\end{figure}

Now let us investigate how the semi-infinite geodesics interact with the instability graph $\IG{[\gamma,\tht]}$. The contents of the next three lemmas are summarized in Figure \ref{fig:instabrightendpt}.

\begin{lem}\label{underpropinstab}
    The following holds for all $\w\in\Omref{Omega8}$ and $\tht\ge\gamma>0$. Suppose that for some $m\in\Z$ and $s<t$, $\{(m+\half,r):s<r<t\}$ is a proper $[\gamma,\tht]$-instability interval.  Then $\geo\from{(m,s)}\dir{R}{\tht}{+}$ goes through $(m,t)$. For any $r\in (s,t)$, $\geo\from{(m,r)}\dir{L}{\tht}{+}=\geo\from{(m,r)}\dir{R}{\tht}{+}$ and this geodesic also goes through $(m,t)$. 
\end{lem}

\begin{proof}
	If $\geo\from{(m,r)}\dir{L}{\tht}{+}$ does not go through $(m,t)$, then it must go up at $(m,r'')$ for some $r''\in[r,t)$. But this would mean that $\geo\from{(m,r'')}\dir{L}{\tht}{+}$ goes up and Lemma \ref{geodefproperIS} says that this contradicts $(m+\half,r'')$ being a proper $[\gamma,\tht]$-instability point.
	Consequently, $\geo\from{(m,r)}\dir{L}{\tht}{+}$ cannot go up before it reaches $(m,t)$. By monotonicity, $\geo\from{(m,r)}\dir{R}{\tht}{+}$ must do the same. 
	Taking a rational $r'\in(s,r)$ and applying 
	Theorem \ref{thm:shocks}\eqref{thm:shocks:0} tells us that $\geo\from{(m,r')}\dir{L}{\tht}{+}=\geo\from{(m,r')}\dir{R}{\tht}{+}$, which implies that also $\geo\from{(m,r)}\dir{L}{\tht}{+}=\geo\from{(m,r)}\dir{R}{\tht}{+}$. 
	Since $\geo\from{(m,s)}\dir{R}{\tht}{+} = \lim\limits_{r\searrow s}\geo\from{(m,r)}\dir{R}{\tht}{+}$, $\geo\from{(m,s)}\dir{R}{\tht}{+}$ also goes through $(m,t)$.
\end{proof}

\begin{lem}\label{instableftendpt}
    The following holds for all $\w\in\Omref{Omega8}$ and  $\tht\ge\gamma>0$. Suppose that for some $m\in\Z$ and $s<t$, $\{(m+\half,r):s<r<t\}$ is a proper $[\gamma,\tht]$-instability interval. Then $(m+\half,s)$ is an improper $[\gamma,\tht]$-instability point if, and only if, $\geo\from{(m,s)}\dir{L}{\tht}{+}$ goes up to $(m+1,s)$.  
\end{lem}

\begin{proof}
	When $(m+\half,s)$ is proper, $\geo\from{(m,s)}\dir{L}{\tht}{+}$ must go right (all the way to $(m,t)$) by Lemma \ref{underpropinstab}. Conversely, 
	suppose $\geo\from{(m,s)}\dir{L}{\tht}{+}$ goes to $(m,r)$ for some $r\in(s,t)$. Then it  coalesces with $\geo\from{(m,r)}\dir{L}{\tht}{+}$ and thus can never intersect $\geo\from{(m+1,s)}\dir{R}{\gamma}{-}$ since $\geo\from{(m+1,s)}\dir{R}{\gamma}{-} \preceq \geo\from{(m+1,r)}\dir{R}{\gamma}{-}$ and $\geo\from{(m+1,r)}\dir{R}{\gamma}{-}\cap \geo\from{(m,r)}\dir{L}{\tht}{+}=\varnothing$ (because $(m+\half,r)$ is a proper instability point). This implies that $(m+\half,s)$ is a proper instability point.  
\end{proof}

\begin{lem}\label{instabrightendpt}
    The following holds for all $\w\in\Omref{Omega8}$ and $\tht\ge\gamma>0$. Suppose that for some $m\in\Z$ and $s<t$, $\{(m+\half,r):s<r<t\}$ is a proper $[\gamma,\tht]$-instability interval. If $(m+\half,t)$ is an improper $[\gamma,\tht]$-instability point, then $\geo\from{(m,t)}\dir{L}{\gamma}{-}$ proceeds rightwards out of $(m,t)$ but $\geo\from{(m+1,t)}\dir{L}{\gamma}{-}$ proceeds immediately vertically to $(m+2,t)$ while $\geo\from{(m+1,t)}\dir{R}{\gamma}{-}$ proceeds rightwards.
\end{lem}

\begin{proof}
	Suppose that  $(m+\half,t)$ is an improper $[\gamma,\tht]$-instability point. Then, 
 by Lemma \ref{improperendpts}, $(m+\half,t)$ is the right endpoint of $\{(m+\half,r):s<r<t\}$ and, by Lemma \ref{biinf:geo}, there is a vertical $[\gamma,\tht]$-instability edge up from $(m+\half,t)$. Thus, by Lemma \ref{geodefinstabedge1}, $\geo\from{(m+1,t)}\dir{L}{\gamma}{-}$ must proceed vertically to $(m+2,t)$. By Proposition \ref{nodouble}, $\geo\from{(m,t)}\dir{L}{\gamma}{-}$ must proceed laterally on level $m$. Likewise $\geo\from{(m+1,t)}\dir{R}{\gamma}{-}$ must immediately proceed laterally. Indeed, if it were to proceed vertically along with $\geo\from{(m+1,t)}\dir{L}{\gamma}{-}$, then we would have $\geo\from{(m+1,t)}\dir{L}{\gamma}{-}=\geo\from{(m+1,t)}\dir{R}{\gamma}{-}$ by Lemma \ref{geo-split}. 
 On the other hand, Theorem \ref{geo:instpt} implies $\geo\from{(m+1,t)}\dir{L}{\gamma}{-}\cap\geo\from{(m,t)}\dir{R}{\tht}{+}=\varnothing$.
 Furthermore, since $\geo\from{(m,t)}\dir{L}{\gamma}{-}$  goes right on level $m$, it forces both $\geo\from{(m,t)}\dir{L}{\tht}{+}$ and $\geo\from{(m,t)}\dir{R}{\tht}{+}$ to go right together and hence $\geo\from{(m,t)}\dir{L}{\tht}{+}=\geo\from{(m,t)}\dir{R}{\tht}{+}$ (Lemma \ref{geo-split} again). But then we would have $\geo\from{(m+1,t)}\dir{R}{\gamma}{-}\cap\geo\from{(m,t)}\dir{L}{\tht}{+}=\varnothing$. By Lemma \ref{geodefproperIS}, this makes $(m+\half,t)$ a proper instability point, which it is assumed not to be.
\end{proof}

The following is a corollary of Theorem \ref{CI-Hausdorff} and the geodesics characterization of vertical instability edges. Compare with Corollary \ref{cor:no isolated}, which states that there are no isolated instability edges.

\begin{lem}\label{lm:edge nowhere dense}
    The following holds for all $\w\in\Omref{Omega8}$ and $\tht\ge\gamma>0$. For any $m\in\Z$, the set
    $\{s:[(m-\half,s),(m+\half,s)]\text{ is a }[\gamma,\tht]\text{-instability edge}\}$
    is nowhere dense. 
\end{lem}

\begin{proof}
    By Theorem \ref{CI-Hausdorff}, for any rational $r$, there exists a $t>r$ such that $\geo\from{(m,r)}\dir{L}{\gamma}{-}$ goes through $(m,t)$. By the geodesics ordering, $\geo\from{(m,r)}\dir{R}{\tht}{+}$ also goes through $(m,t)$. Then Lemma \ref{geodefinstabedge1} implies that for any $s\in[r,t)$, $[(m-\half,s),(m+\half,s)]$ is not a $[\gamma,\tht]$-instability edge. The claim now follows from the fact that rational numbers are dense.
\end{proof}

\begin{figure}[ht!]
\begin{center}

\begin{tikzpicture}[>=latex, scale=0.6]

\draw(0,0)--(6,0);
\draw[dashed](0,1)--(6,1);
\draw(0,2)--(6,2);
\draw[dashed](0,3)--(6,3);

\draw(6,0)node[right]{$m$};
\draw(6,1)node[right]{$m+\half$};
\draw(6,2)node[right]{$m+1$};
\draw(6,3)node[right]{$m+\frac32$};

\draw[line width=2pt,color=Lm](0.88,0)--(0.88,2);
\draw[line width=2pt,color=Lp](1,0)--(1,2);
\draw[line width=2pt,color=Rp](1,0)--(5,0);
\draw[line width=2pt,color=Lp](1.17,0.1)--(5,0.1);
\draw[line width=2pt,color=Rm](4,0.22)--(5,0.22);
\draw[line width=2pt,color=Lm](4,0.32)--(5,0.32);

\draw[line width=3.5pt,color=Inst](1,1)--(4.02,1)--(4.02,3);

\draw[line width=2pt,color=Rm](3.95,2)--(5,2);
\draw[line width=2pt,color=Lm](3.9,2)--(3.9,4);

\end{tikzpicture}
\end{center}

\caption{\small A summary of the contents of Lemmas \ref{underpropinstab}-\ref{instabrightendpt}. The geodesics are color-coded as follows: \textcolor{Lm}{$\geo\from{}\dir{L}{\gamma}{-}$}, \textcolor{Rm}{$\geo\from{}\dir{R}{\gamma}{-}$}, \textcolor{Lp}{$\geo\from{}\dir{L}{\tht}{+}$}, and \textcolor{Rp}{$\geo\from{}\dir{R}{\tht}{+}$}. The instability graph is the thicker red line.}
\label{fig:instabrightendpt}
\end{figure}

Next, we describe the web-like structure of the instability graph. 

\begin{lem}\label{instabifbtw}
	The following holds for all $\w\in\Omref{Omega8}$ and $\tht\ge\gamma>0$. 
 Take $\bfx\dual, \bfy\dual \in \IG{[\gamma,\tht]}$ such that $\bfy\dual=(n+\half,t)$ and $\bfx\dual = (m+\half,s)$ with $m\le n$ in $\Z$ and $s\leq t$ in $\R$.  Then the following are equivalent:
	\begin{enumerate}[label={\rm(\alph*)}, ref={\rm\alph*}]   \itemsep=3pt 
		\item\label{instabifbtw.a} There is an up-right path in $\IG{[\gamma,\tht]}$ from $\bfx\dual$ to $\bfy\dual$.
		\item\label{instabifbtw.b} $\bfy\dual$ is between $\geo\from{(m+1,s)}\dir{L}{\gamma}{-}$ and $\geo\from{(m,s)}\dir{R}{\tht}{+}$. 
	\end{enumerate}
\end{lem}

\begin{figure}[ht!]
\begin{center}

\begin{tikzpicture}[>=latex, scale=0.6]
\begin{scope}
\draw[dashed](0,1)--(6,1);
\draw(0,2)--(6,2);
\draw[dashed](0,3)--(6,3);
\draw(0,4)--(6,4);

\draw[line width=3.5pt,color=Inst](4,3)--(1.1,3);
\draw[line width=3.5pt,color=Inst,dashed](1.1,3)--(1.1,1);
\draw[line width=2pt,color=Lm](1,2)--(1,4);
\draw[line width=2pt,color=Rp](1,2)--(3,2);
\draw(1.5,3.85)node[above]{\textcolor{Lm}{$\geo\from{}\dir{L}{\gamma}{-}$}};
\draw(2.5,1.75)node[above right]{\textcolor{Rp}{$\geo\from{}\dir{R}{\tht}{+}$}};

\shade[ball color=Inst](4,3)circle(1.5mm);
\draw(4,3)node[above]{$\bfy\dual$};

\draw(6,1)node[right]{$k-\half$};
\draw(6,2)node[right]{$k$};
\draw(6,3)node[right]{$k+\half$};
\draw(6,4)node[right]{$k+1$};
\end{scope}

\begin{scope}[shift={(11,0)}]
\draw[dashed](0,1)--(6,1);
\draw(0,2)--(6,2);
\draw[dashed](0,3)--(6,3);
\draw(0,4)--(6,4);

\draw[line width=3.5pt,color=Inst](3,1)--(3,3);
\draw[line width=3.5pt,color=Inst,dashed](1,3)--(3,3);
\draw[line width=2pt,color=Lm](1,2.05)--(2.9,2.05)--(2.9,4);
\draw[line width=2pt,color=Rp](1,1.95)--(5,1.95);
\draw(3,0.8)node[below]{$r$};
\draw[line width=1.5pt](1,.85)--(1,1.15);
\draw(1,1)node[below]{$r'$};

\draw(3.5,3.85)node[above]{\textcolor{Lm}{$\geo\from{}\dir{L}{\gamma}{-}$}};
\draw(4.5,1.75)node[above right]{\textcolor{Rp}{$\geo\from{}\dir{R}{\tht}{+}$}};

\shade[ball color=Inst](3,3)circle(1.5mm);
\draw(3,3)node[above right]{$\bfy\dual$};

\draw(6,1)node[right]{$k-\half$};
\draw(6,2)node[right]{$k$};
\draw(6,3)node[right]{$k+\half$};
\draw(6,4)node[right]{$k+1$};
\end{scope}

\end{tikzpicture}
\end{center}

\caption{\small An illustration of the two cases in the proof of \eqref{instabifbtw.b}$\Rightarrow$\eqref{instabifbtw.a} in Lemma \ref{instabifbtw}. In both instances, the configuration of the instability graph and the geodesics (depicted by the solid lines) implies that the instability graph follows the path indicated by the thick dashed red line.}
\label{fig:instabifbtw}
\end{figure}

\begin{proof}
\eqref{instabifbtw.a}$\Rightarrow$\eqref{instabifbtw.b}. Suppose there exists an up-right path $\pi$ in $\IG{[\gamma,\tht]}$ from $\bfx\dual$ to $\bfy\dual$.
By Lemma \ref{underpropinstab},  $\geo\from{(m,s)}\dir{R}{\tht}{+}$ cannot cross an instability interval neither at an interior (and therefore proper) point nor at its left endpoint. Moreover, considering that from the right endpoint of an instability interval (denoted as $\bfz\dual$), the instability graph can only go up, and in accordance with Proposition \ref{nodouble}, geodesics do not take consecutive vertical steps, the geodesic $\geo\from{(m,s)}\dir{R}{\tht}{+}$ is unable to transition from a position below 
$\bfz\dual$ to a position above $\bfz\dual+e_2$.  
Consequently, based on these observations, we see that since $\geo\from{(m,s)}\dir{R}{\tht}{+}$ starts below $\pi$, the whole path has to remain below $\pi$. 
Similarly,  by Lemma \ref{geodefinstabedge1}, $\geo\from{(m+1,s)}\dir{L}{\gamma}{-}$ cannot cross an instability edge. Therefore, since it starts above $\pi$, the whole path has to remain above $\pi$. This proves \eqref{instabifbtw.b}.

\eqref{instabifbtw.b}$\Rightarrow$\eqref{instabifbtw.a}.
By Lemma \ref{biinf:geo}, there always exists an infinite down-left path on $\IG{[\gamma,\tht]}$, starting at $\bfy\dual$. We will show, however, that there exists a down-left path that remains confined between $\geo\from{(m,s)}\dir{R}{\tht}{+}$ and $\geo\from{(m+1,s)}\dir{L}{\gamma}{-}$. This then ensures that the path must go through $\bfx\dual$ and proves \eqref{instabifbtw.a}.  To see the existence of such a path, start at $\bfy\dual$ and follow any down-left path in $\IG{[\gamma,\tht]}$.

If the path reaches a point $(k+\half,r)\in\IG{[\gamma,\tht]}$ such that $\{(k,r),(k+1,r)\}\subset\geo\from{(m+1,s)}\dir{L}{\gamma}{-}$, then Theorem \ref{geo:instpt} says that $\geo\from{(k,r)}\dir{R}{\tht}{+}$ goes right and then Lemma \ref{geodefinstabedge1} implies that 
$[(k-\half,r),(k+\half,r)]$ is an instability edge. We can then use this edge to continue the down-left path without crossing $\geo\from{(m+1,s)}\dir{L}{\gamma}{-}$, and hence continuing to remain between $\geo\from{(m,s)}\dir{R}{\tht}{+}$ and $\geo\from{(m+1,s)}\dir{L}{\gamma}{-}$. 

Similarly, suppose the path reaches a point $(k+\half,r)\in\IG{[\gamma,\tht]}$ such that 
$[(k-\half,r),(k+\half,r)]$ is an instability edge in $\IG{[\gamma,\tht]}$, but where $\{(k,r'),(k,r)\}\subset\geo\from{(m,s)}\dir{R}{\tht}{+}$ for some $r'<r$. By Lemma \ref{geodefinstabedge1}, $\geo\from{(m,s)}\dir{R}{\tht}{+}$ must go right, $\geo\from{(k,r)}\dir{L}{\gamma}{-}$ must go up, and then the two never touch again. By the monotonicity of geodesics, $\geo\from{(k+1,r'')}\dir{L}{\gamma}{-}$ is always to the left of $\geo\from{(k,r)}\dir{L}{\gamma}{-}$, for any $r''\le r$. Therefore, we see that $\geo\from{(k+1,r'')}\dir{L}{\gamma}{-}\cap\geo\from{(k,r'')}\dir{R}{\tht}{+}=\varnothing$ for any $r''\in[r',r]$. Then by Theorem \ref{geo:instpt}, 
$\{(k,r''):r'\le r''\le r\}$ is an instability interval.
We can then follow this interval to continue the down-left path without crossing $\geo\from{(m,s)}\dir{R}{\tht}{+}$, and hence continuing to remain between $\geo\from{(m,s)}\dir{R}{\tht}{+}$ and $\geo\from{(m+1,s)}\dir{L}{\gamma}{-}$. 
\end{proof}

Taking $\gamma=\tht\in\baddir$ in the next lemma gives Theorem \ref{instabweb:summary}\eqref{inst:summary.directed}.

\begin{lem}\label{lm:directed}
The following holds for all $\w\in\Omref{Omega8}$ and $\tht\ge\gamma>0$.
Any up-right path $\bfx_{0:\infty}\dual$ on the instability graph $\IG{\tht}$ is directed into $[\gamma,\tht]$: 
\[\gamma\le\varliminf_{n\to\infty}\frac{\bfx_n\dual\cdot e_2}{\bfx_n\dual\cdot e_1}\le\varlimsup_{n\to\infty}\frac{\bfx_n\dual\cdot e_2}{\bfx_n\dual\cdot e_1}\le\tht.\]
\end{lem}

\begin{proof}
Let $\bfx_0\dual=(m+\half,s)$. Then
Lemma \ref{instabifbtw} implies that for all $n\in\N$,  $\bfx_n\dual$ is between $\geo\from{(m+1,s)}\dir{L}{\gamma}{-}$ and $\geo\from{(m,s)}\dir{R}{\tht}{+}$. The lemma thus follows from the directedness of these geodesics, given in Theorem \ref{thm:geo}\eqref{thm:geo:dir}.
\end{proof}

\begin{defn}
    Take $\tht\ge\gamma>0$. We say that $\bfy\dual\in\IG{[\gamma,\tht]}$ is a NE \emph{ancestor} of $\bfx\dual\in\IG{[\gamma,\tht]}$ or, equivalently, that $\bfx\dual$ is a SW \emph{descendant} of $\bfy\dual$ if there is an up-right path on $\IG{[\gamma,\tht]}$ going from $\bfx\dual$ to $\bfy\dual$. 
\end{defn}

Theorem \ref{instabweb:summary} follows from the case $\gamma=\tht\in\baddir$ in next two theorems.

\begin{thm}\label{instabcmnchild}
	The following holds for all $\w\in\Omref{Omega8}$ and $\tht\ge\gamma>0$. Every pair of instability points $ \bfx\dual, \bfy\dual \in\IG{[\gamma,\tht]}$ have a common NE ancestor $ \bfz\dual \in \IG{[\gamma,\tht]}$.  
 \end{thm}

 \begin{proof}
	Write $ \bfx\dual=(m+\half,s) $ and $ \bfy\dual=(n+\half,t) $ and consider the geodesics $ \geo\from{(m+1,s)}\dir{L}{\gamma}{-}$, $\geo\from{(m,s)}\dir{R}{\tht}{+}$, $\geo\from{(n+1,t)}\dir{L}{\gamma}{-}$, and $\geo\from{(n,t)}  \dir{R}{\tht}{+}$. By the coalescence of geodesics (Theorem \ref{thm:geo}\eqref{thm:geo:coal}), there exist coalescence points $ \mathbf{z}^{\gamma-} $ and $  \mathbf{z}^{\tht+} $ such that $\geo\from{(m+1,s)}\dir{L}{\gamma}{-}$ and $\geo\from{(n+1,t)}\dir{L}{\gamma}{-}$ match beyond $ \mathbf{z}^{\gamma-} $  and $ \geo\from{(m,s)}\dir{R}{\tht}{+}$ and $\geo\from{(n,t)}  \dir{R}{\tht}{+} $ match beyond $  \mathbf{z}^{\tht+} $. Choose a $ u\in \geo\from{(n,t)}  \dir{R}{\tht}{+} $ such that $\mathbf{z}^{\gamma-} \vee \mathbf{z}^{\tht+} \le u  $ coordinatewise. Then $ \geo\from{u} \dir{R}{\tht}{+}$ will continue along with $ \geo\from{(n,t)}  \dir{R}{\tht}{+} $, but $ \geo\from{u} \dir{L}{\gamma}{-}$ will eventually have to separate from $ \geo\from{(n,t)}  \dir{R}{\tht}{+} $ at some point $z=(
 \ell,r)$ to coalesce with $ \geo\from{(m+1,s)}\dir{L}{\gamma}{-} $ (and $ \geo\from{(n+1,t)}\dir{L}{\gamma}{-} $). By Lemma \ref{geodefinstabedge1}, it follows that $[(\ell-\half,r),(\ell+\half,r)]$ is an instability edge in $\IG{[\gamma,\tht]}$.  Then $\bfz\dual=(\ell+\half,r)\in\IG{[\gamma,\tht]}$ is between $ \geo\from{(m+1,s)}\dir{L}{\gamma}{-}$ and $\geo\from{(m,s)}\dir{R}{\tht}{+}$, as well as between $ \geo\from{(n+1,t)}\dir{L}{\gamma}{-}$ and $\geo\from{(n,t)}  \dir{R}{\tht}{+}$ and thus by Lemma \ref{instabifbtw} is a NE ancestor of both $\bfx\dual$ and $ \bfy\dual $. 
\end{proof}

\begin{thm}\label{instabcommonancstr}
	The following holds for all $\w\in\Omref{Omega8}$ and $\tht\ge\gamma>0$. Every pair of instability points $ \bfx\dual, \bfy\dual \in\IG{[\gamma,\tht]}$ have a common SW descendant $ \bfz\dual\in \IG{[\gamma,\tht]}$. 
 \end{thm}

\begin{figure}[ht!]
\begin{center}

\begin{tikzpicture}[>=latex, scale=0.6]

\draw(0,0)--(12,0);
\draw[dashed](0,1)--(12,1);
\draw(0,2)--(12,2);
\draw[dashed](0,3)--(12,3);
\draw(0,4)--(12,4);
\draw(0,8)--(12,8);
\draw[dashed](0,9)--(12,9);
\draw(0,10)--(12,10);
\draw(0,14)--(12,14);
\draw[dashed](0,15)--(12,15);
\draw(0,16)--(12,16);
\draw[dashed](0,17)--(12,17);
\draw(0,18)--(12,18);

\draw(12,0)node[right]{$k-1$};
\draw(12,2)node[right]{$k$};
\draw(12,4)node[right]{$k+1$};
\draw(12,8)node[right]{$\ell$};
\draw(12,10)node[right]{$\ell+1$};
\draw(12,14)node[right]{$m-1$};
\draw(12,16)node[right]{$m$};
\draw(12,18)node[right]{$m+1$};

\draw(12,6)node[right]{$\vdots$};
\draw(12,12)node[right]{$\vdots$};

\draw[line width=2pt,color=Lm,->](1,0)--(1,2)--(1.5,2)--(1.5,4)--(2.25,4)--(2.25,6)--(2.55,6)--(2.55,8)--(2.7,8)--(2.7,10)--(3.3,10)--(3.3,12)--(3.7,12)--(3.7,14)--(4.5,14)--(4.5,16)--(4.8,16)--(4.8,18)--(5.2,18)--(6.2,19);
\draw[line width=2pt,color=Lm](2,2)--(2,4);
\draw[line width=2pt,color=Rp](2,2)--(2.75,2)--(2.75,4)--(3.35,4)--(3.35,6)--(4.5,6);
\draw[line width=2pt,color=Rp](1,0)--(3.2,0)--(3.2,2)--(4,2)--(4,4)--(4.5,4)--(4.5,6)--(5.75,6)--(5.75,8)--(6.25,8)--(6.25,10)--(6.65,10)--(6.65,12)--(7.55,12);
\draw[line width=2pt,color=Rp,<->](2.75,-3)--(3.5,-2)--(4,-2)--(4,0)--(4.5,0)--(4.5,2)--(5.2,2)--(5.2,4)--(6.2,4)--(6.2,6)--(6.75,6)--(6.75,8)--(7.1,8)--(7.1,10)--(7.55,10)--(7.55,12)--(8,12)--(8,14)--(8.45,14)--(8.45,16)--(9,16)--(9,18)--(9.55,18)--(10.55,19);

\shade[ball color=Inst](1,1)circle(1.5mm);
\draw(1.7,1)node[below]{$s_{k-1}$};
\shade[ball color=Inst](5.5,1)circle(1.5mm);
\draw(6.2,1)node[below]{$t_{k-1}$};
\shade[ball color=Inst](2,3)circle(1.5mm);
\draw(2.35,3)node[below]{$s_k$};
\shade[ball color=Inst](6.6,3)circle(1.5mm);
\draw(6.95,3)node[below]{$t_k$};
\shade[ball color=Inst](3.5,9)circle(1.5mm);
\draw(3.85,9)node[below]{$s_\ell$};
\shade[ball color=Inst](8,9)circle(1.5mm);
\draw(8.35,9)node[below]{$t_\ell$};
\shade[ball color=pt](6.25,10)circle(1.5mm);
\draw(6.65,10.1)node[below]{$\pi_\ell^k$};
\shade[ball color=Inst](4.9,15)circle(1.5mm);
\draw(5.25,15)node[below]{$s_m$};
\shade[ball color=Inst](9,15)circle(1.5mm);
\draw(9.35,15)node[below]{$t_m$};
\shade[ball color=Inst](5.5,17)circle(1.5mm);
\draw(5.85,17)node[below]{$x\dual$};
\shade[ball color=Inst](9.8,17)circle(1.5mm);
\draw(10.15,17)node[below]{$y\dual$};

\draw(4,-1)node[right]{$\Gamma$};

\end{tikzpicture}
\end{center}

\caption{\small The proof of Theorem \ref{instabcommonancstr}.}
\label{fig:instabcommonancstr}
\end{figure}

\begin{proof}
	Write $\bfx\dual=(m+\half,s)$ and $\bfy\dual=(n+\half,t)$ and suppose  $ \bfx\dual$ and $\bfy\dual $ do not have a common descendant in $\IG{[\gamma,\tht]}$. By Lemma \ref{biinf:geo}, there are infinite down-left paths on $\IG{[\gamma,\tht]}$ out of each of $\bfx\dual$ and $\bfy\dual$. Therefore, without loss of generality, we continue the proof under the assumption that $m=n$ and $t>s$. 

    Again, by Lemma \ref{biinf:geo}, there exist sequences $s_k$ and $t_k$, $k\le m+1$, such that $s_{m+1}=s$, $t_{m+1}=t$, and for all $k\le m$ we have that $(k-\half,s_{k-1})\in\IG{[\gamma,\tht]}$ is a descendant of $(k+\half,s_k)$ and $(k-\half,t_{k-1})\in\IG{[\gamma,\tht]}$ is a descendant of $(k+\half,t_k)$. Since we assumed that $\bfx\dual$ and $\bfy\dual$ do not share a common descendant and that $s<t$, we must also have $s_k<t_k$ for all $k\le m$.

    Take $k<\ell\le m$. By Lemma \ref{instabifbtw}, being a NE ancestor of $(k+\half,s_k)$, $(\ell+\half,s_\ell)$ must be between $ \geo\from{(k+1,s_k)}\dir{L}{\gamma}{-} $ and $ \geo\from{(k,s_k)}\dir{R}{\tht}{+} $. Since $s_\ell<t_\ell$, $(\ell+\half,t_\ell)$ is to the right of $ \geo\from{(k+1,s_k)}\dir{L}{\gamma}{-} $. But since $\bfy\dual$ does not share a descendant with $\bfx\dual$, $(k+\half,s_k)$ cannot be a descendant of $(\ell+\half,t_\ell)$ and hence, by Lemma \ref{instabifbtw}, $(\ell+\half,t_\ell)$ must be to the right of $ \geo\from{(k,s_k)}\dir{R}{\tht}{+} $. Therefore, we have shown that for all $k<\ell\le m$, $ \geo\from{(k,s_k)}\dir{R}{\tht}{+} $ passes between $(\ell+\half,s_\ell)$ and $(\ell+\half,t_\ell)$ when it crosses level $\ell+\half$.

    Furthermore, since $(k-\half,s_{k-1})$ is a descendant of $(k+\half,s_k)$, the latter must be to the left of 
    $ \geo\from{(k-1,s_{k-1})}\dir{R}{\tht}{+} $, which implies that this geodesic has to remain entirely to the right of $ \geo\from{(k,s_k)}\dir{R}{\tht}{+} $.

    For $k<\ell\le m$ let $\pi_\ell^k$ be the first entry point of $ \geo\from{(k,s_k)}\dir{R}{\tht}{+} $ at level $\ell+1$ and, hence, the last exit point of the geodesic from level $\ell$. This means that the geodesic crosses level $\ell+\half$ at $(\ell+\half,\pi_\ell^k)$. The above observations then show that for each $k\le k'<\ell\le m$ we have $s_\ell\le\pi_\ell^k\le t_\ell$ and $\pi_\ell^{k'}\le\pi_\ell^k$.

    Thus, for each $\ell\le m$ there exists a subsequence $k_j$ and a $\pi_\ell\in[s_\ell,t_\ell]$ such that $\pi_\ell^{k_j}\nearrow\pi_\ell$ as $j\to\infty$. By the diagonal trick, we can extract a single subsequence that works for all $\ell\le m$ simultaneously. 

    Let $\Gamma$ be the bi-infinite up-right path that enters each level $\ell+1$ at $(\ell+1,\pi_\ell)$, for $\ell\le m$, and from $(m+1,\pi_m)$ and on follows  $\geo\from{(m+1,\pi_m)}\dir{L}{\tht}{+} $. Then, by the definition of $\{\pi_\ell:\ell\le m\}$, for each $\ell\le m$, the portion of $\Gamma$ that goes between levels $\ell+1$ and $m+1$ is the limit of the corresponding portion of $ \geo\from{(k_j,s_{k_j})}\dir{R}{\tht}{+} $ as $j\to\infty$. Consequently, the piece of $\Gamma$ from level $\ell+1$ and on matches $\geo\from{(\ell+1,\pi_\ell)}\dir{L}{\tht}{+} $ for all $\ell\le m$. This, together with Lemma \ref{geo-split}, imply that $\Gamma$ is a non-degenerate bi-infinite locally rightmost (and locally leftmost) geodesic, which is prohibited by Theorem \ref{nobiinf}.
\end{proof}

We close this section with two properties relating $\IG{[\gamma,\delta]}$, $\delta>\gamma>0$, to $\IG{\tht}$, $\tht\in[\gamma,\delta]$. 

\begin{lem}\label{IGint=union}
For all $\w\in\Omref{Omega8}$ and $\delta>\gamma>0$,
$\IG{[\gamma,\delta]}=\bigcup_{\tht\in[\gamma,\delta]\cap\baddir}\IG{\tht}$.
\end{lem}

The last result of this section gives one practical use for $\IG{[\gamma,\delta]}$. It says that we can approximate $\IG{\tht}$ (e.g.\ for simulation purposes) by $\IG{[\gamma,\delta]}$ with $\gamma<\tht<\delta$ and $\delta-\gamma$ small.  

\begin{lem}\label{IG-approx}
     For all $\w\in\Omref{Omega8}$ and $\tht>0$,
$\IG{\tht}=\bigcap_{\gamma<\tht<\delta}\IG{[\gamma,\delta]}$.
\end{lem}

We defer the proofs of Lemmas \ref{IGint=union} and \ref{IG-approx} to the end of Section \ref{sec:cif}, as they rely on theorems concerning competition interfaces.

\begin{rmk}
    The points in the union $\bigcup_{\delta>\gamma>0}\IG{[\gamma,\delta]}=\bigcup_{\tht\in\baddir}\IG{\tht}$ are all of $\Z\times\R$. Curiously, the vertical edges do not fill out all of $\R^2$ but rather, by Theorem \ref{CI-Hausdorff} and Theorem \ref{thm:cif.int}\eqref{thm:cif-int.a} below, form a Hausdorff $\half$ dimensional dense subset of $\R^2$.
\end{rmk}

\section{Shocks and their relation to the instability graph}\label{sec:shocks}

Recall that $(m,s)$ is called a $\tht\sigg$-shock point if $(m,s)\in\NU_1^{\tht\sig}$, where $\NU_1^{\tht\sig}$ is defined in \eqref{def:NU1}. 
As a first observation, the following Lemma notes that for any $\tht>0$ and $\sigg\in\{-,+\}$, the $\tht\sigg$ shocks, on any level, are nowhere dense. This contrasts with the fact that, almost surely, for any $m\in\Z$, the set $\bigcup_{\tht>0}\bigcup_{\sig\in\{-,+\}}\{s:(m,s)\in\NU_1^{\tht\sig}\}=\CI_m$ is dense in $\R$ (Theorem \ref{CI-Hausdorff}).

\begin{lem}
  The following holds for all $\w\in\Omref{Omega8}$, $\gamma>0$, and $m\in\Z$. The set $\bigcup_{\tht\ge\gamma}\bigcup_{\sig\in\{-,+\}}\{s:(m,s)\in\NU_1^{\tht\sig}\}$ is nowhere dense.  
\end{lem}

\begin{proof}
    By Theorem \ref{CI-Hausdorff}, for any rational $r$, there exists a $t>r$ such that $\geo\from{(m,r)}\dir{L}{\gamma}{-}$ goes through $(m,t)$. By the geodesics ordering, $\geo\from{(m,r)}\dir{S}{\tht}{\sig}$ also goes through $(m,t)$, for any $\tht\ge\gamma$, $\sigg\in\{-,+\}$, and $S\in\{L,R\}$. Thus, for any $s\in[r,t)$, $(m,s)\not\in\NU_1^{\tht\sig}$. The claim now follows from the fact that rational numbers are dense.
\end{proof}

We begin the study of the properties of shock points by showing that there is a $\tht+$ shock under each left endpoint of a maximal instability interval and that there are no $\tht+$ shocks under any other points in the interval.
The contents of the next five lemmas are summarized in Figure \ref{fig:shock-instability}.

\begin{lem}\label{+shckiffimprop}
        The following holds for all $\w\in\Omref{Omega8}$, $\tht\ge\gamma>0$, $m\in\Z$, and $s<t$. Suppose $\{(m+\half,r):s\le r\le t\}\subset\IG{[\gamma,\tht]}$.
        Then $(m,s)\in \NU_1^{\tht+}$ if, and only if, $(m+\half,s)$ is an improper $[\gamma,\tht]$-instability point. Consequently, for any $r\in(s,t)$, $(m,r)\not\in \NU_1^{\tht+}$. Furthermore, $(m,t)\not\in\NU_1^{\tht+}$.
\end{lem}

\begin{proof}
    Lemma \ref{onlypropinprop} implies that $\{(m+\half,r):s<r<t\}$ is a proper instability interval. The first two claims then follow directly from Lemmas \ref{underpropinstab} and \ref{instableftendpt}.
    This also implies the last claim if $(m+\half,t)$ is a proper instability point. If, on the other hand, 
    it is an improper instability point, then 
    the last claim follows from Lemma \ref{instabrightendpt} and the fact that $\geo\from{(m,t)}\dir{L}{\gamma}{-}$ is always above both $\geo\from{(m,t)}\dir{L}{\tht}{+}$ and $\geo\from{(m,t)}\dir{R}{\tht}{+}$.
\end{proof}

Next, we show that there cannot be a $\gamma-$ shock under the left endpoint of an instability interval and that each $\gamma-$ shock under an instability interval must have an instability edge going through it. Together, this and the previous lemma show that $\tht+$ and $\gamma-$ shocks are distinct under instability intervals. They also show that there are no $\tht+$ or $\gamma-$ shocks under right endpoints of maximal instability intervals.

\begin{lem}\label{properif-shck}
	The following holds for all $\w\in\Omref{Omega8}$, $\tht\ge\gamma>0$, $m\in\Z$, and $s<t$. Suppose $\{(m+\half,r):s\le r\le t\}\subset\IG{[\gamma,\tht]}$. If $(m,s)\in \NU_1^{\gamma-} $, then $ (m+\half,s) $ is a proper $[\gamma,\tht]$-instability point. Furthermore, $[(m-\half,s),(m+\half,s)]$ is a $[\gamma,\tht]$-instability edge. 
\end{lem}

\begin{proof}
If $(m,s)\in \NU_1^{\gamma-}$, then $\geo\from{(m,s)} \dir{R}{\gamma}{-}$ goes through $(m,r)$ for each $r>s$ close enough to $s$. This implies that for such $r$, $\geo\from{(m,s)} \dir{R}{\gamma}{-}\cap\geo\from{(m,r)} \dir{L}{\tht}{+}\ne\varnothing$. By Lemma \ref{lm:edge}, this implies that $\bus{\gamma}{-}(m,s,m,r)=\buse{+}{m}{s}{m}{r}$ for all $r>s$ close enough to $s$. But since $(m+\half,r)\in\IG{[\gamma,\tht]}$ for all $r\in(s,t)$, Theorem \ref{thm:instpt}\eqref{instpt.d} implies that $\bus{\gamma}{-}(m,s,m+1,s)<\buse{+}{m}{s}{m+1}{s}$, which says that $(m+\half,s)$ is a proper $[\gamma,\tht]$-instability point. 

Next, note that $(m,s)\in\NU_1^{\gamma-}$ implies that $\geo\from{(m,s)} \dir{L}{\gamma}{-}$ goes up to $(m+1,s)$ while, since $(m+\half,s)$ is proper, Lemma \ref{underpropinstab} implies that $\geo\from{(m,s)} \dir{R}{\tht}{+}$ and $\geo\from{(m,s)} \dir{L}{\tht}{+}$ both go right from $(m,s)$. Lemma \ref{geo-split} implies then that these two geodesics are equal. But then  $\geo\from{(m,s)} \dir{R}{\tht}{+}=\geo\from{(m,s)} \dir{L}{\tht}{+}$ cannot reconvene with $\geo\from{(m,s)} \dir{L}{\gamma}{-}$ because that would violate the fact that $\geo\from{(m,s)} \dir{L}{\tht}{+}$ is a leftmost point-to-point geodesic between any of its points (Theorem \ref{thm:geo}\eqref{thm:geo:LR}). 
Now Lemma \ref{geodefinstabedge1} implies that $[(m-\half,s),(m+\half,s)]$ is an instability edge. 
\end{proof}

The next lemma says that shock points that are not below instability points are simultaneously $\tht+$ and $\gamma-$ shock points. 
By Lemma \ref{+shckiffimprop}, each left endpoint of an instability interval has a $\tht+$ shock point below it and, by Theorem \ref{thm:shocks}\eqref{thm:shocks:c}, this shock point will have an accumulation of $\tht+$ shock points to its left. The following lemma then says that these are also $\gamma-$ shock points. Consequently, we see that there are infinitely many simultaneous $\gamma-$ and $\tht+$ shock points outside instability intervals.

\begin{lem}\label{+-shckiffnotinstab}
	The following holds for all $\w\in\Omref{Omega8}$ and $\tht\ge\gamma>0$. Take $(m,s)\in \NU_1^{\gamma-}\cup \NU_1^{\tht+}$. Then $ (m+\half,s) \notin \IG{[\gamma,\tht]}$ if and only if $(m,s)\in \NU_1^{\gamma-}\cap \NU_1^{\tht+}$.
\end{lem}

\begin{proof}
    If $(m+\half,s)\in\IG{[\gamma,\tht]}$, then Lemmas \ref{+shckiffimprop} and \ref{properif-shck} imply that $(m,s)\not\in\NU_1^{\gamma-}\cap \NU_1^{\tht+}$.

    For the other direction, suppose that $ (m+\half,s) \notin \IG{[\gamma,\tht]}$ and $(m,s)\in \NU_1^{\tht+}\setminus\NU_1^{\gamma-}$. 
    Then $(m,s)\in \NU_1^{\tht+}$ implies that $ \geo\from{(m,s)}\dir{R}{\tht}{+}$ goes right while $ \geo\from{(m,s)}\dir{L}{\tht}{+}$ and hence also $ \geo\from{(m,s)}\dir{L}{\gamma}{-}$ go up. But then $(m,s)\not\in \NU_1^{\gamma-}$ implies that $ \geo\from{(m,s)}\dir{R}{\gamma}{-}$ goes up as well and Lemma \ref{geo-split} implies that $ \geo\from{(m,s)}\dir{R}{\gamma}{-}=\geo\from{(m,s)}\dir{L}{\gamma}{-}$. 
    Since $ (m+\half,s) \notin \IG{[\gamma,\tht]}$, $\geo\from{(m+1,s)}\dir{L}{\gamma}{-}\cap\geo\from{(m,s)}\dir{R}{\tht}{+}\ne\varnothing$. Since  $\geo\from{(m,s)}\dir{L}{\gamma}{-}$ goes up, this implies that $\geo\from{(m,s)}\dir{L}{\gamma}{-}\cap\geo\from{(m,s)}\dir{R}{\tht}{+}\ne\varnothing$, which in turn says that $\geo\from{(m,s)}\dir{R}{\gamma}{-}\cap\geo\from{(m,s)}\dir{R}{\tht}{+}\ne\varnothing$. Since these two geodesics split at $(m,s)$, we get a contradiction with $\geo\from{(m,s)}\dir{R}{\gamma}{-}$ being a right-most point-to-point geodesic between any of its two points (Theorem \ref{thm:geo}\eqref{thm:geo:LR}). 

    A similar contradiction is obtained from assuming that $ (m+\half,s) \notin \IG{[\gamma,\tht]}$ and $(m,s)\in \NU_1^{\gamma-}\setminus\NU_1^{\tht+}$.
    This proves the lemma.
\end{proof}

The next lemma shows that there are infinitely many $\gamma-$ shocks under each instability interval. Since we have shown that a maximal instability interval has only one $\tht+$ shock, located below its left endpoint, and that shocks outside instability intervals are simultaneously $\tht+$ and $\gamma-$ shocks, while shocks under proper instability points are only $\gamma-$ shocks, we now see that there are many more $\gamma-$ shocks than there are $\tht+$ shocks. 

\begin{lem}\label{-shockafter+shock}
    The following holds for all $\w\in\Omref{Omega8}$, $\tht\ge\gamma>0$, $m\in\Z$, and $s<t$. Suppose $\{(m+\half,r):s\le r\le t\}\subset\IG{[\gamma,\tht]}$ and that $(m,s)$ is an improper $[\gamma,\tht]$-instability point.
    Then there is a sequence $s_k\in(s,t]$, strictly decreasing to $s$, and such that $(m,s_k)\in\NU_1^{\gamma-}\setminus\NU_1^{\tht+}$ for all $k$.
\end{lem}

\begin{proof}
    By Lemma \ref{+shckiffimprop}  $(m,s)\in \NU_1^{\tht+}$. Consequently, $\geo\from{(m,s)} \dir{R}{\tht}{+}$ goes right from $(m,s)$ while  $\geo\from{(m,s)} \dir{L}{\tht}{+}$ goes up, causing $\geo\from{(m,s)} \dir{L}{\gamma}{-}$ to also go up. By Lemma \ref{properif-shck}, $(m,s)\not\in\NU_1^{\gamma-}$, which implies that $\geo\from{(m,s)} \dir{R}{\gamma}{-}$ goes up. Thus, $\geo\from{(m,s)} \dir{R}{\tht}{+}$ and $\geo\from{(m,s)} \dir{R}{\gamma}{-}$ split immediately, but $\geo\from{(m,s)} \dir{L}{\tht}{+}$ and $\geo\from{(m,s)} \dir{L}{\gamma}{-}$ do not. This implies that $(m,s)\in\NU_1^{\tht+}\setminus\NU_1^{\gamma-}$.
    The claim now follows from Theorem \ref{thm:shocks}\eqref{thm:shocks:d} and Lemma \ref{+shckiffimprop}. 
\end{proof}

In contrast, the next result shows that there are no shocks to the immediate right of a right endpoint of an instability interval.

\begin{lem}
    The following holds for all $\w\in\Omref{Omega8}$, $\tht\ge\gamma>0$, $m\in\Z$, and $s<t$. Suppose $\{(m+\half,r):s\le r\le t\}\subset\IG{[\gamma,\tht]}$ and that $(m,t)$ is an improper $[\gamma,\tht]$-instability point.
    Then there exists an $\e>0$ such that for any $r\in[t,t+\e]$, $(m,r)\not\in\NU_1^{\gamma-}\cup\NU_1^{\tht+}$.
\end{lem}

\begin{proof}
    By Lemma \ref{instabrightendpt}, $\geo\from{(m,t)}\dir{L}{\gamma}{-}$ proceeds rightwards out of $(m,t)$. By the ordering of geodesics, this implies that $\geo\from{(m,t)}\dir{R}{\gamma}{-}$, $\geo\from{(m,t)}\dir{L}{\tht}{+}$, and $\geo\from{(m,t)}\dir{R}{\tht}{+}$ also proceed rightwards out of $(m,t)$, which prevents having any shock points $(m,r)$ for $r\ge t$ close enough to $t$.
\end{proof}

Next, we note that above each right endpoint of a maximal instability interval there is a $\gamma-$ shock that is not a $\tht+$ shock.

\begin{lem}\label{-shockabove}
		The following holds for all $\w\in\Omref{Omega8}$ and $\tht\ge\gamma>0$. Suppose that for some $m\in\Z$ and $s<t$, $\{(m+\half,r):s<r<t\}$ is a proper $[\gamma,\tht]$-instability interval. If $(m+\half,t)$ is an improper $[\gamma,\tht]$-instability point, then $(m+1,t)\in \NU_1^{\gamma-}\setminus\NU_1^{\tht+}$. 
\end{lem}

\begin{proof}
	This follows immediately from Lemmas \ref{instabrightendpt}, \ref{+shckiffimprop}, and \ref{properif-shck}.
\end{proof}

To complete the picture, the next two lemmas show that the $\gamma-$ shocks beneath an instability interval occur precisely at the points where a rightmost vertical instability edge passes through. This aligns with the facts that, on each level, there are countably many shock points and a set of instability edges with Hausdorff dimension $\half$.

\begin{lem}\label{lem:right isolated=>-}
The following holds for all $\w\in\Omref{Omega8}$,  $\tht\ge\gamma>0$, $m\in\Z$, and $s\in\R$. Suppose $[(m+\half,s),(m+\frac32,s)]$ is a $[\gamma,\tht]$-instability edge and for some $t>s$, $[(m+\half,r),(m+\frac32,r)]$ is not a $[\gamma,\tht]$-instability edge for all $r\in(s,t)$. Then $(m+1,s)\in\NU_1^{\gamma-}$.
\end{lem}

\begin{proof}
By Lemma \ref{geodefinstabedge1}, the geodesic $\geo\from{(m+1,s)}\dir{R}{\tht}{+}$ goes right, the geodesic $\geo\from{(m+1,s)}\dir{L}{\gamma}{-}$ goes up to $(m+2,s)$, and the two never touch again.  Since no geodesics go up twice in a row (Proposition \ref{nodouble}), the latter geodesic must turn right from $(m+2,s)$. Thus, there exists some $t'\in(s,t)$ such that the  geodesic $\geo\from{(m+1,s)}\dir{R}{\tht}{+}$ goes through $(m+1,t')$ and the geodesic $\geo\from{(m+1,s)}\dir{L}{\gamma}{-}$ goes through $(m+2,t')$. If for some $r\in(s,t')$, the geodesic $\geo\from{(m+1,r)}\dir{L}{\gamma}{-}$  were to go up, then it would continue with the geodesic $\geo\from{(m+1,s)}\dir{L}{\gamma}{-}$  and thus never again touch the geodesic $\geo\from{(m+1,r)}\dir{R}{\tht}{+}$  (which follows the geodesic $\geo\from{(m+1,s)}\dir{R}{\tht}{+}$). Then, $[(m+\half,r),(m+\frac23,r)]$ would be a $[\gamma,\tht]$-instability edge. The assumptions of the lemma prevent that and hence for any $r\in(s,t')$, the geodesic $\geo\from{(m+1,r)}\dir{L}{\gamma}{-}$ must go right. By the convergence in \eqref{geocont}, the geodesic $\geo\from{(m+1,s)}\dir{R}{\gamma}{-}$ must go right and since the geodesic $\geo\from{(m+1,s)}\dir{L}{\gamma}{-}$ goes up we have that $(m+1,s)\in\NU_1^{\gamma-}$.
\end{proof}

\begin{lem}
The following holds for all $\w\in\Omref{Omega8}$, $\tht\ge\gamma>0$, $m\in\Z$, and $s\in\R$.
Suppose that for some sequence $s_n$ that strictly decreases to $s$, $[(m+\half,s_n),(m+\frac32,s_n)]$ is a $[\gamma,\tht]$-instability edge for all $n$.  Then $(m+1,s)\not\in\NU_1^{\gamma-}$.
\end{lem}

\begin{proof}
By Lemma \ref{geodefinstabedge1}, for all $n$, the geodesic $\geo\from{(m+1,s_n)}\dir{L}{\gamma}{-}$  goes up. By the convergence in \eqref{geocont}, the geodesic $\geo\from{(m+1,s)}\dir{R}{\gamma}{-}$  goes up and hence $(m+1,s)\not\in\NU_1^{\gamma-}$. \end{proof}

The following corollary is an interesting contrast to Lemma \ref{lm:edge nowhere dense}. 

\begin{cor}\label{cor:no isolated}
    The following holds for all $\w\in\Omref{Omega8}$, $\tht\ge\gamma>0$, $m\in\Z$, and $s\in\R$. Suppose $[(m-\half,s),(m+\half,s)]$ is a $[\gamma,\tht]$-instability edge. Then there is either a decreasing sequence $s_n\searrow s$ or an increasing sequence $s_n\nearrow s$ {\rm(}or both{\rm)} such that $[(m-\half,s_n),(m+\half,s_n)]$ is a $[\gamma,\tht]$-instability edge for all $n$. 
\end{cor}

\begin{proof}
    Suppose that there does not exist a decreasing sequence $s_n\searrow s$ as in the claim. Then Lemma \ref{lem:right isolated=>-} implies that $(m,s)\in\NU_1^{\gamma-}$.
    By Lemmas \ref{biinf:geo} and \ref{nodoubleinstab}, $(m-\half,s)$ is part of a $[\gamma,\tht]$-instability interval. Then Lemma \ref{properif-shck} implies that $(m-\half,s)$ is a proper $[\gamma,\tht]$-instability point and Lemma \ref{improperendpts} implies that there exists an $r<s$ such that $[(m-\half,r),(m+\half,s)]$ is a $[\gamma,\tht]$-instability interval. The left density of shocks (Theorem \ref{thm:shocks}\eqref{thm:shocks:c}) and Lemma \ref{properif-shck} imply now the existence of an increasing sequence $s_n\nearrow s$ as in the claim.
\end{proof}

We summarize the above results in the following theorem. See Figure \ref{fig:shock-instability}.
Theorem \ref{thm:shocks-IG} comes as the special case when $\gamma=\tht\in\baddir$.

\begin{thm}\label{thm:shocks-IG.int}
The following holds for all $\w\in\Omref{Omega8}$ and $\tht\ge\gamma>0$.
Suppose that for some $m\in\Z$ and $s<t$, $\{(m+\half,r):s\le r\le t\}$ is a $[\gamma,\tht]$-instability interval and that $(m+\half,s)$ and $(m+\half,t)$ are improper $[\gamma,\tht]$-instability points.
\begin{enumerate}  [label={\rm(\roman*)}, ref={\rm\roman*}]   \itemsep=3pt 
\item\label{thm:shocks-IG.int.a} $(m,s)\in\NU_1^{\tht+}\setminus\NU_1^{\gamma-}$ and $(m+1,t)\in\NU_1^{\gamma-}\setminus\NU_1^{\tht+}$. 
\item\label{thm:shocks-IG.int.b} For any $r\in(s,t)$, $(m,r)\not\in\NU_1^{\tht+}$. Furthermore, $(m,r)\in\NU_1^{\gamma-}$ if, and only if, $[(m-\half,r),(m+\half,r)]$ is a $[\gamma,\tht]$-instability edge that is right-isolated among $[\gamma,\tht]$-instability edges.
\item\label{thm:shocks-IG.int.c} For any $r\in(s,t)$ with $(m,r)\in\NU_1^{\gamma-}$, $[m-\half,r),(m+\half,r)]$ is a $[\gamma,\tht]$-instability edge.
\item\label{thm:shocks-IG.int.d} There exists an $\e>0$ such that for any $r\in[t,t+\e]$, $(m,r)\not\in\NU_1^{\gamma-}\cup\NU_1^{\tht+}$.
\item\label{thm:shocks-IG.int.e} For any $r\in\R$ such that $(m+\half,r)\not\in\IG{[\gamma,\tht]}$, $(m,r)$ is either in $\NU_1^{\gamma-}\cap\NU_1^{\tht+}$ or not in $\NU_1^{\gamma-}\cup\NU_1^{\tht+}$.
\item\label{thm:shocks-IG.int.f} There exist sequences $r_n'\nearrow s$ and $r_n''\searrow s$ such that $(m,r_n')\in\NU_1^{\gamma-}\cap\NU_1^{\tht+}$ and $(m,r_n'')\in\NU_1^{\gamma-}\setminus\NU_1^{\tht+}$ for all $n$.
\end{enumerate}
\end{thm}

\begin{figure}[ht!]
\begin{center}

\begin{tikzpicture}[>=latex, scale=0.6]

\draw(-1,0)--(10,0);
\draw[dashed](-1,1)--(10,1);
\draw(-1,2)--(10,2);
\draw[dashed](-1,3)--(10,3);
\draw(-1,4)--(10,4);

\draw(10,0)node[right]{$m$};
\draw(10,1)node[right]{$m+\half$};
\draw(10,2)node[right]{$m+1$};
\draw(10,3)node[right]{$m+\frac32$};
\draw(10,4)node[right]{$m+2$};

\draw[line width=2pt,color=Inst](1,-1)--(1,1)--(5,1)--(5,3)--(9,3);
\draw[line width=2pt,color=Inst](5,1)--(8,1)--(8,3);
\draw[line width=2pt,color=Inst](1.5,-1)--(1.5,1);
\draw[line width=2pt,color=Inst](2.1,-1)--(2.1,1);
\draw[line width=2pt,color=Inst](2.7,-1)--(2.7,1);
\draw[line width=2pt,color=Inst](5.5,1)--(5.5,3);
\draw[line width=2pt,color=Inst](6.1,1)--(6.1,3);
\draw[line width=2pt,color=Inst](7.5,1)--(7.5,3);

\filldraw[fill=white](-0.7,0)circle(2.5mm);
\draw(-0.7,0)node{\tiny$\pm$};
\filldraw[fill=white](-0.15,0)circle(2.5mm);
\draw(-0.15,0)node{\tiny$\pm$};
\filldraw[fill=white](.5,0)circle(2.5mm);
\draw(.5,0)node{\tiny$\pm$};
\filldraw[fill=white](1,0)circle(2.5mm);
\draw(1,0)node{\tiny$+$};
\filldraw[fill=white](1.5,0)circle(2.5mm);
\draw(1.5,0)node{\tiny$-$};
\filldraw[fill=white](2.1,0)circle(2.5mm);
\draw(2.1,0)node{\tiny$-$};
\filldraw[fill=white](2.7,0)circle(2.5mm);
\draw(2.7,0)node{\tiny$-$};
\filldraw[fill=white](5,2)circle(2.5mm);
\draw(5,2)node{\tiny$+$};
\filldraw[fill=white](5.5,2)circle(2.5mm);
\draw(5.5,2)node{\tiny$-$};
\filldraw[fill=white](6.1,2)circle(2.5mm);
\draw(6.1,2)node{\tiny$-$};
\filldraw[fill=white](4.5,2)circle(2.5mm);
\draw(4.5,2)node{\tiny$\pm$};
\filldraw[fill=white](3.85,2)circle(2.5mm);
\draw(3.85,2)node{\tiny$\pm$};
\filldraw[fill=white](8,2)circle(2.5mm);
\draw(8,2)node{\tiny$-$};
\filldraw[fill=white](7.5,2)circle(2.5mm);
\draw(7.5,2)node{\tiny$-$};

\end{tikzpicture}
\end{center}

\caption{\small An illustration of 
Theorem \ref{thm:shocks-IG.int}. Circles with $-$ refer to $\gamma-$ shock points, circles with $+$ refer to $\tht+$ shock points, and circles with $\pm$ refer to points that are simultaneously $\gamma-$ and $\tht+$ shock points.}
\label{fig:shock-instability}
\end{figure}

\begin{rmk}\label{more shocks 2}
Parts \eqref{thm:shocks-IG.int.a}, \eqref{thm:shocks-IG.int.b},  \eqref{thm:shocks-IG.int.e}, and \eqref{thm:shocks-IG.int.f} in the above theorem give a natural injection that shows there are infinitely many more $\gamma-$ shocks than there are $\tht+$ shocks. As mentioned in Remark \ref{rk:(gamma,tht)}, repeating our results for the instability graph $\IG{(\gamma,\tht)}$ gives a similar statement about $\gamma+$ shocks and $\tht-$ shocks, for all $\tht>\gamma>0$. Putting these results together shows that $\P$-almost surely, for any $\sig,\sig'\in\{-,+\}$ and any $\tht>\gamma>0$, there are many more $\gamma\sig$ shocks than there are $\tht\sig'$ shocks. This quantifies the following statement: The number of shock points in $\NU_1^{\tht\sig}$ increases as $\tht\sig$ decreases.
\end{rmk}

\begin{rmk}[Reconstructing $\IG{[\gamma,\tht]}$ from shocks]\label{shock-reconstruct}
	Knowledge of $\NU_1^{\tht+}$ and $\NU_1^{\gamma-}$ can allow one to sketch a skeleton of the instability graph $\IG{[\gamma,\tht]}$ through the following algorithm:	
	\begin{enumerate}
		\item\label{alg1} Mark all solely $\tht+$ shocks a half level up. I.e.\ if $(m,s)\in \NU_1^{\tht+}\setminus\NU_1^{\gamma-}$, mark the point $(m+\half,s)$. These are the left endpoints of maximal instability intervals.
		\item For each point that you marked in \eqref{alg1}, say for example $(m+\half,s)$, find the rightmost solely $\gamma-$ shock up a half level from the marked point (e.g.\ $(m+1,t)\in \NU_1^{\gamma-}\setminus\NU_1^{\tht+}$ for $s<t$) and left of the next marked point on level $(m+\half)$. Half a level down from this point, e.g.\ $(m+\half,t)$ is the right endpoint of the instability interval that began at your marked point $( m+\half,s )$.
		\item Find the remaining solely $\gamma-$ shock points and draw vertical instability edges through these points. E.g.\ if $(m,r)\in\NU_1^{\gamma-}\setminus\NU_1^{\tht+}$ is such a solely $\gamma-$ shock point, then draw an instability edge between $(m+\half,r)$ and $(m-\half,r)$. 
	\end{enumerate}
 What is missing, however, is a Hausdorff dimension 1/2 set of vertical instability edges (Theorem \ref{inst:summary.int}\eqref{inst:summary.int.Hausdorff}) that do not go through any shock points. It is natural that these intervals cannot be reconstructed from shock points because there are only countably many shock points (Theorem \ref{thm:shocks}\eqref{thm:shocks:a}). We will see in the next section that these edges can be recovered from the knowledge of competition interface starting points.
\end{rmk}

Recall the ancestry relationship between shock points, given in Definition \ref{def:shock-tree}. We show that this induces a tree structure.

In our next lemma we show that even though shock points are left-dense (Theorem \ref{thm:shocks}\eqref{thm:shocks:c}),   
they do not have descendants on the same level.

\begin{lem}\label{nodescendants}
    The following holds for all $\w\in\Omref{Omega8}$, $\tht>0$, and $\sigg\in\{-,+\}$. Take $(m,s)\in\NU_1^{\tht\sig}$. Then for any $r<s$ with $(m,r)\in\NU_1^{\tht\sig}$, $(m,r)$ is not a descendant of $(m,s)$.
\end{lem}

\begin{proof}
Since $\geo\from{(m,s)}\dir{L}{\tht}{\sig}$ goes vertically from $(m,s)$ and $\geo\from{(m,r)}\dir{R}{\tht}{\sig}$ converges to $\geo\from{(m,s)}\dir{L}{\tht}{\sig}$ as $r\nearrow s$, there exists a sequence $r_k$ that increases to $s$ as $k\to\infty$ with $\geo\from{(m,r_k)}\dir{R}{\tht}{\sig}$ going through $(m+1,r_k)$, for each $k$. Now take  any $r<s$ and choose $k$ large enough so that $r<r_k<s$. Then $\geo\from{(m,r)}\dir{R}{\tht}{\sig}$ goes up at or before $(m,r_k)$ and hence strictly before reaching $(m,s)$. Consequently, $(m,r)$ is not a descendant of $(m,s)$.
\end{proof}

Next, we show that the SW descendants of a shock point are totally ordered by the ancestry relation.

\begin{lem}\label{totalorder}
    The following holds for all $\w\in\Omref{Omega8}$, $\tht>0$, and $\sigg\in\{-,+\}$.  Take $\bfx, \bfy,\bfz \in\NU_1^{\tht\sig}$. Assume that $\bfx$ and $\bfy$ are both SW descendants of $\bfz$. Then either $\bfx$ is a SW descendant of $\bfy$ or $\bfy$ is a SW descendant of $\bfx$. 
\end{lem}

\begin{proof}
    For $(m,s)\in\Z\times\R$ let $C_{(m,s)}$ denote the set of sites $(n,t)\in\Z\times\R$ that are strictly between $\geo\from{(m,s)}\dir{\gamma}{L}{\sig}$ and $\geo\from{(m,s)}\dir{\gamma}{R}{\sig}$. Let $\overline C_{(m,s)}$ denote the closure of $C_{(m,s)}$, i.e.\ the set of sites that are weakly between the two geodesics.
    
    Suppose $\bfy$ is not a NE ancestor of $\bfx$ and $\bfx$ is not a NE ancestor of $\bfy$. Then $\bfy$ is strictly outside $\overline C_{\bfx}$ and $\bfx$ is strictly outside $\overline C_{\bfy}$. Since for any $S,S'\in\{L,R\}$, $\geo\from{\bfy}\dir{\gamma}{S}{\sig}$ coalesces with $\geo\from{\bfx}\dir{\gamma}{S'}{\sig}$ as soon they meet, the geodesic $\geo\from{\bfy}\dir{\gamma}{S}{\sig}$ cannot penetrate to region $C_{\bfx}$ and the geodesic $\geo\from{\bfx}\dir{\gamma}{S'}{\sig}$ cannot penetrate to region $C_{\bfy}$. This implies that $C_{\bfx}$ and $C_{\bfy}$ are disjoint.

    Since $\bfz$ is a NE ancestor of both $\bfx$ and $\bfy$, it must be in both $\overline C_{\bfx}$ and $\overline C_{\bfy}$. Therefore, it belongs to a common part of the boundary of the two regions.
    That is, either $\bfz\in\geo\from{\bfx}\dir{\gamma}{L}{\sig}\cap\geo\from{\bfy}\dir{\gamma}{R}{\sig}$ or $\bfz\in\geo\from{\bfx}\dir{\gamma}{R}{\sig}\cap\geo\from{\bfy}\dir{\gamma}{L}{\sig}$.

    Since $\geo\from{\bfz}\dir{\gamma}{L}{\sig}$ proceeds vertically from $\bfz$ and since Proposition \ref{nodouble} prevents geodesics from having two consecutive vertical steps, the only place where $\bfz$ can be on $\geo\from{\bfx}\dir{\gamma}{L}{\sig}$ or $\geo\from{\bfy}\dir{\gamma}{L}{\sig}$ is at a corner where the geodesic comes in horizontally and then moves immediately up. 
    But since $\geo\from{\bfz}\dir{\gamma}{R}{\sig}$ proceeds horizontally from $\bfz$, $\bfz$ cannot be at a corner where  $\geo\from{\bfx}\dir{\gamma}{R}{\sig}$ or $\geo\from{\bfy}\dir{\gamma}{R}{\sig}$ arrives horizontally and then moves immediately up. Thus, $\bfz$ cannot belong to $\geo\from{\bfx}\dir{\gamma}{L}{\sig}\cap\geo\from{\bfy}\dir{\gamma}{R}{\sig}$ nor to $\geo\from{\bfx}\dir{\gamma}{R}{\sig}\cap\geo\from{\bfy}\dir{\gamma}{L}{\sig}$, arriving at a contradiction.
\end{proof}

Working towards establishing the tree structure of shocks, we demonstrate that each shock point has a unique SW descendant on the next level down. This is the shock point's ``immediate'' descendant, which we call the shock point's \emph{child}.

\begin{lem} \label{shock:descendant}
	The following holds for all $\w\in\Omref{Omega8}$, $\tht>0$, and $\sigg\in\{-,+\}$.  For each $(m,s) \in \NU_1^{\tht\sig}$ there exists a unique $t<s$ such that $(m-1,t)\in \NU_1^{\tht\sig}$ and $(m-1,t)$ is a descendant of $(m,s)$. Furthermore, if $s_1<s_2$ are such that $(m,s_1),(m,s_2)\in\NU_1^{\tht\sig}$ and $(m-1,t_i)$ is the unique SW descendant of $(m,s_i)$ for each $i\in\{1,2\}$, then $t_1\le t_2$. 
\end{lem}

\begin{proof}
	Take  $(m,s) \in \NU_1^{\tht\sig}$. Then $\geo\from{(m,s)}\dir{L}{\tht}{\sig}$ must immediately proceed upwards to $(m+1,s)$. By Proposition \ref{nodouble}, $\geo\from{(m-1,s)}\dir{L}{\tht}{\sig}$ must initially proceed laterally on level $m-1$, going through $(m-1,s')$ for some $s'>s$. Define 
	\begin{align}\label{def:des} 
	t = \inf\Bigl\{r\le s:\geo\from{(m-1,r)}\dir{L}{\tht}{\sig}\text{ goes through }(m-1,s')  \Bigr\}.
	\end{align}
	The above set is not empty as it contains $s$. Also, $t>-\infty$ since otherwise (with Lemma \ref{geo-split}) we would have a non-degenerate bi-infinite locally rightmost (and locally leftmost) geodesic, which is prohibited by Theorem \ref{nobiinf}.
    
    Note that for any $r\in(t,s')$,  $\geo\from{(m-1,r)}\dir{L}{\tht}{\sig}$ goes through $(m-1,s')$. In particular, $(m-1,r)\not\in\NU_1^{\tht\sig}$ for all such $r$. Since $\geo\from{(m-1,r)}\dir{L}{\tht}{\sig}$ converge to $\geo\from{(m-1,t)}\dir{R}{\tht}{\sig}$ as $r\searrow t$, we get that $\geo\from{(m-1,t)}\dir{R}{\tht}{\sig}$ must go to the right out of $(m-1,t)$.  
	
    If $\geo\from{(m-1,t)}\dir{L}{\tht}{\sig}$ were to proceed horizontally, then the fact that $\geo\from{(m-1,t)}\dir{L}{\tht}{\sig}$ is the limit of $\geo\from{(m-1,r)}\dir{L}{\tht}{\sig}$ as $r\nearrow t$ would imply that for $r<t$ close enough to $t$, $\geo\from{(m-1,r)}\dir{L}{\tht}{\sig}$ must go rightwards and merge with $\geo\from{(m-1,t)}\dir{L}{\tht}{\sig}$, which by proceeding to the right will have to continue going to the right until it reaches $(m,s')$. But then the definition of $t$ as an infimum would imply that $t\le r$, which would be a contradiction.  Therefore, $\geo\from{(m-1,t)}\dir{L}{\tht}{\sig}$ must proceed upwards immediately to level $m$. In particular, we have that $t<s$ and that $(m-1,t)\in\NU_1^{\tht\sig}$.

    That $(m-1,t)$ is a descendant of $(m,s)$ follows from the fact that $\geo\from{(m-1,t)}\dir{L}{\tht}{\sig}$ has to remain to the left of $\geo\from{(m,s)}\dir{L}{\tht}{\sig}$ while $\geo\from{(m-1,t)}\dir{R}{\tht}{\sig}$ proceeds to $(m-1,s')$ and hence has to remain below $\geo\from{(m,s)}\dir{L}{\tht}{\sig}$. 

 We have shown so far that $(m-1,t)$ is the rightmost descendant of $(m,s)$ on level $m-1$. By Lemma \ref{totalorder} any other descendant of $(m,s)$ on level $m-1$ must be a descendant of $(m-1,t)$. However, Lemma \ref{nodescendants} implies that $(m-1,t)$ does not have any descendants on level $m-1$. Hence, $(m-1,t)$ is the only descendant of $(m,s)$ on level $m-1$.

 For the monotonicity claim, observe that if $s'_1$ and $s'_2$ denote the times defining $t_1$ and $t_2$, respectively, via \eqref{def:des}, then we can assume $s'_1<s'_2$ (since $s_1<s_2$) and then \eqref{def:des} implies that $t_1\le t_2$.
\end{proof}

\begin{lem}\label{lm:inf-ancestors}
    The following holds for all $\w\in\Omref{Omega8}$, $\tht>0$, and $\sigg\in\{-,+\}$. 
    Suppose $(m,s)\in\NU_1^{\tht\sig}$ has a NE ancestor. Then it has infinitely many NE ancestors on level $m+1$.
\end{lem}

\begin{proof}
Note that Lemma \ref{shock:descendant} implies that $(m,s)$ cannot be the descendant of $(m+1,s)$.  (This also follows from Proposition \ref{nodouble}.) Therefore, if $(m,s)$ has a NE ancestor, it must have a NE ancestor $(m+1,t)$ with $t>s$. This means that $(m+1,t)$ is to the left of $\geo\from{(m,s)}\dir{R}{\tht}{\sig}$ and strictly to the right of $\geo\from{(m,s)}\dir{L}{\tht}{\sig}$.
By Theorem \ref{thm:shocks}\eqref{thm:shocks:c}, there exists a sequence $(m+1,t_k)\in\NU_1^{\tht\sig}$ such that $t_k$ strictly increases to $t$ as $k\to\infty$. Then for each $k$ large enough $(m+1,t_k)$ is between $\geo\from{(m,s)}\dir{R}{\tht}{\sig}$ and  $\geo\from{(m,s)}\dir{L}{\tht}{\sig}$ and, therefore, is a NE ancestor of $(m,s)$. 
\end{proof}

The above lemmas establish that shocks have a forest structure. In the next lemma, we show that the ancestry relation gives a single tree.

\begin{lem}\label{shockcommonancstr}
	The following holds for all $\w\in\Omref{Omega8}$, $\tht>0$, and $\sigg\in\{-,+\}$.  Take $\bfx, \bfy \in\NU_1^{\tht\sig}$. Then $\bfx$ and $\bfy$ have a common SW descendant $\bfz \in \NU_1^{\tht\sig}$. 
\end{lem}

\begin{proof}
     The proof is similar to that of Theorem \ref{instabcommonancstr}.
     Write $\bfx=(m,s)$ and $\bfy=(n,t)$. By Lemma \ref{shock:descendant} we can assume, without loss of generality, that $m=n$ and $s<t$. Using Lemma \ref{shock:descendant} again, we can find sequences $s_k$ and $t_k$, $k\le m$, such that $s_m=s$, $t_m=t$, and for all $k<m$, $(k,s_k)\in\NU_1^{\tht\sig}$ is the rightmost descendant of $(k+1,s_{k+1})$ on level $k$ and $(k,t_k)\in\NU_1^{\tht\sig}$ is the rightmost descendant of $(k+1,t_{k+1})$ on level $k$. Since $s<t$, we get $s_k\le t_k$ for all $k\le m$. The assumption that $\bfx$ and $\bfy$ have no common descendant implies then that in fact $s_k<t_k$ for all $k\le m$.

The geodesics $\geo\from{(k,s_k)}\dir{R}{\tht}{\sig}$ must now go between $(\ell,s_\ell)$ and $(\ell,t_\ell)$, for all $k<\ell\le m$ and, as in the proof of Theorem \ref{instabcommonancstr}, we can extract from this a bi-infinite locally rightmost (and locally leftmost) geodesic which, by Theorem \ref{nobiinf} has zero $\P$-probability of occuring.
\end{proof}

The next three lemmas demonstrate that both $\NU_1^{\tht+}\setminus\NU_1^{\gamma-}$ and $\NU_1^{\gamma-}\setminus\NU_1^{\tht+}$ also form trees. 

\begin{lem}\label{+tree}
    The following holds for all $\w\in\Omref{Omega8}$ and $\tht\ge\gamma>0$. Take and $(m,s) \in\NU_1^{\tht+}\setminus\NU_1^{\gamma-} $. Then all SW descendants of $ (m,s) $ in the $\NU_1^{\tht+}$ tree  are in $\NU_1^{\tht+}\setminus \NU_1^{\gamma-} $.
\end{lem}

\begin{proof}
        By Lemmas \ref{nodescendants}, \ref{totalorder}, and \ref{shock:descendant}, it is enough to prove the claim for the immediate child of $(m,s)$, in the $\NU_1^{\tht+}$ tree. Denote this by $(m-1,r)$, where $r<s$.

        By Lemmas \ref{+shckiffimprop} and \ref{+-shckiffnotinstab}, $(m+\half,s)$ must be the left endpoint of a $[\gamma,\tht]$-instability interval. By Lemmas \ref{biinf:geo} and \eqref{nodoubleinstab}, 
        there must exist an $r'<s$ such that $[(m-\half,r'),(m-\half,s)] \subset \IG{[\gamma,\tht]}$ and $(m-\half,r')$ is an improper instability point. By Lemma \ref{+shckiffimprop}, $(m-1,r') \in \NU_1^{\tht+}\setminus\NU_1^{\gamma-} $ and there are no points in $\NU_1^{\tht+}$ in the interval $(r',s)$ on level $m-1$. Thus, $r'=r$, i.e.\ $(m-1,r')$ is the SW child of $(m,s)$ and so $(m-1,r) \in \NU_1^{\tht+}\setminus\NU_1^{\gamma-} $. 
\end{proof}

\begin{lem}\label{-tree}
   The following holds for all $\w\in\Omref{Omega8}$ and $\tht\ge\gamma>0$.  Take $ (m,s) \in\NU_1^{\gamma-}\setminus\NU_1^{\tht+} $. Then all SW descendants of $ (m,s) $ in the $\NU_1^{\gamma-}$ tree are in $\NU_1^{\gamma-}\setminus \NU_1^{\tht+} $. 
\end{lem}

\begin{proof}
    By Lemmas \ref{nodescendants}, \ref{totalorder}, and \ref{shock:descendant}, it is enough to prove the claim for the immediate child of $(m,s)$, in the $\NU_1^{\gamma-}$ tree. Denote this by $(m-1,r)$, where $r<s$.
    
	By Lemma \ref{properif-shck}, $[(m+\half,s), (m-\half,s)]$ must be a $[\gamma,\tht]$-instability edge that, by Lemmas \ref{biinf:geo} and \ref{nodoubleinstab}, goes down into a $[\gamma,\tht]$-instability interval on level $m-\half$. 
     By Lemmas \ref{biinf:geo} and \ref{nodoubleinstab}, $(m-\half,s)$ cannot be the left endpoint of that instability interval, and thus there exists an $r'<s$ such that $[(m-\half,r'),(m-\half,s)] \subset \IG{[\gamma,\tht]}$ and $(m-\half,r')$ is an improper $[\gamma,\tht]$-instability point. Then by Lemma \ref{-shockafter+shock}, there exists an $r''\in(r',s)$ such that $(m-1,r'') \in NU_1^{\gamma-}$. This implies that $r\in[r'',s)$. In particular, $(m-\half,r)$ is proper and by Lemma \ref{+shckiffimprop}, it is not in $\NU_1^{\tht+}$.
\end{proof}

A consequence of Lemmas \ref{+tree} and \ref{-tree} is that $\NU_1^{\tht+}\setminus\NU_1^{\gamma-}$ and $\NU_1^{\gamma-}\setminus\NU_1^{\tht+}$ are both forests. Together with Lemma \ref{shockcommonancstr}, we get that they, in fact, form subtrees of, respectively, $\NU_1^{\tht+}$ and $\NU_1^{\gamma-}$.

\begin{lem}\label{+&-trees}
    For all $\w\in\Omref{Omega8}$ and $\tht\ge\gamma>0$,   $\NU_1^{\tht+}\setminus\NU_1^{\gamma-} $ is a subtree of the tree $\NU_1^{\tht+}$ and $\NU_1^{\gamma-}\setminus\NU_1^{\tht+} $ is a subtree of the tree $\NU_1^{\gamma-}$.
\end{lem}

    By Lemmas \ref{+shckiffimprop} and \ref{+-shckiffnotinstab}, shock points in $\NU_1^{\tht+}\setminus\NU_1^{\gamma-}$ are in one-to-one correspondence with maximal $[\gamma,\tht]$-instability intervals. Hence, Lemma \ref{+&-trees} also says that maximal $[\gamma,\tht]$-instability intervals have a tree structure.  In contrast, Lemma \ref{-shockafter+shock} says that each maximal instability interval has infinitely many $\NU_1^{\gamma-}\setminus\NU_1^{\tht+}$ shock points under it.  This motivates the following interesting question:
Can one balance the $\NU_1^{\tht+}\setminus\NU_1^{\gamma-}$ tree by distinguishing a special shock point in $\NU_1^{\gamma-}\setminus\NU_1^{\tht+}$, for each maximal instability interval, in such a way that these special shock points form a subtree of the $\NU_1^{\gamma-}$ tree?

Since $\NU_1^{\tht+}\setminus\NU_1^{\gamma-}$ is a subtree of $\NU_1^{\tht+}$, one may wonder about the structure of the remaining points $\NU_1^{\tht+}\cap\NU_1^{\gamma-}$ inside the tree $\NU_1^{\tht+}$. The next lemma says that these points form ``bushes'', each of which has a finite number of generations. The same is true of $\NU_1^{\tht+}\cap\NU_1^{\gamma-}$ inside the tree $\NU_1^{\gamma-}$.

\begin{lem}\label{pm tree}
	    For all $\w\in\Omref{Omega8}$ and $\tht\ge\gamma>0$. Take $ (m,s) \in \NU_1^{\tht+}\cap \NU_1^{\gamma-} $. Then $ (m,s) $ has a SW descendant in $ \NU_1^{\tht+}\setminus \NU_1^{\gamma-} $ and a SW descendant in $ \NU_1^{\gamma-}\setminus \NU_1^{\tht+} $.
\end{lem}

\begin{proof}
    We prove the first claim, the second being similar.
    
    By Lemma \ref{+-shckiffnotinstab}, $(m+\half,s) \not\in \IG{[\gamma,\tht]}$.
    By Lemma \ref{inf many edges}, there exist $s'<s$ and $s''>s$ such that $[(m-\half,s'),(m+\half,s')]$ and $[(m-\half,s''),(m+\half,s'')]$ are $[\gamma,\tht]$-instability edges. Thus, there exists an $r''>s$ such that $(m+\half,r'')$ is the right endpoint of a $[\gamma,\tht]$-instability interval.  On the other hand, if it were the case that $(m+\half,r)$ is a $[\gamma,\tht]$-instability point for all $r\le s'$, then by Lemma \ref{improperendpts} $(m+\half,r)$ is improper for all $r<s'$. Lemma \ref{underpropinstab} implies then that $\geo\from{(m,r)}\dir{R}{\tht}{+}$ goes through $(m,s')$, which (together with Lemma \ref{geo-split}) says that the path $\{\Gamma_t:t\in\R\}$ with $\Gamma_t=(m,t)$ for $t\le s'$ and $\Gamma_t=\geo\from{(m,s')}\dir{R}{\tht}{+}(t)$ for $t\ge s'$ is a non-trivial bi-infinite locally rightmost (and locally leftmost) geodesic. Since this is prohibited by Theorem \ref{nobiinf}, we get that there must exist an $r'\le s'$ such that $(m+\half,r')$ is a left endpoint of a $[\gamma,\tht]$-instability interval.
    
    By Lemma \ref{+shckiffimprop}, $(m,r')$ and $(m,r'')$ are in $\NU_1^{\tht+}\setminus\NU_1^{\gamma-}$. By Lemma \ref{+&-trees}, the two have a common SW descendant $(n,t)$ in the tree $\NU_1^{\tht+}\setminus\NU_1^{\gamma-}$, which is a subtree of  $\NU_1^{\tht+}$. By the monotonicity in Lemma \ref{shock:descendant}, $(n,t)$ must be a SW descendant of $(m,s)$ in $\NU_1^{\tht+}\setminus\NU_1^{\gamma-}$. 
\end{proof}

\begin{proof}[Proof of Theorem \ref{shocks:summary}]
The claims of the theorem follow directly from Lemmas \ref{nodescendants}-\ref{pm tree}.
\end{proof}

\section{Competition interfaces and their relation to shocks and the instability graph}\label{sec:cif}

We now turn to the relationship between the instability graph and the starting points of competition interfaces.

We show that points where left and/or right competition interfaces emanate and have an asymptotic direction in $[\gamma,\tht]$ are exactly those points where a $[\gamma,\tht]$-instability edge descends. Furthermore, we have the following trichotomy:
the point is the left endpoint of a maximal $[\gamma,\tht]$-instability interval, positioned above a $\tht+$ shock, and generates exclusively a right competition interface; or the point is interior to a $[\gamma,\tht]$-instability interval, situated above a $\gamma-$ shock, and produces solely a left competition interface; or the point is interior to a $[\gamma,\tht]$-instability interval, and no shock exists below it. In this case, both left and right competition interfaces emerge from the point, and they match in this scenario.

Theorem \ref{thm:cif} follows from the following by setting $\gamma=\tht\in\baddir$.

\begin{thm}\label{thm:cif.int}
    The following hold for all $\w\in\Omref{Omega8}$, $\gamma\ge\tht>0$, and $(m,s)\in\Z\times\R$.
    \begin{enumerate} [label={\rm(\alph*)}, ref={\rm\alph*}]   \itemsep=3pt 
    \item\label{thm:cif-int.a} $[(m-\half,s),(m+\half,s)]$ is $[\gamma,\tht]$-instability edge if, and only if, for some $S\in\{L,R\}$, $\geo\from{(m,s)}\dir{S}{\tht}{+}$ and $\geo\from{(m,s)}\dir{S}{\gamma}{-}$ split at $(m,s)$ {\rm(}i.e.\ $[\gamma,\tht]\cap\{\tht^L_{(m,s)},\tht^R_{(m,s)}\}\ne\varnothing${\rm)}.
    \item\label{thm:cif-int.b} If $(m+\half,s)$ is the left endpoint of a maximal $[\gamma,\delta]$-instability interval,
    then $\geo\from{(m,s)}\dir{R}{\tht}{+}$ and $\geo\from{(m,s)}\dir{R}{\gamma}{-}$ split at $(m,s)$ but $\geo\from{(m,s)}\dir{L}{\tht}{+}$ and $\geo\from{(m,s)}\dir{L}{\gamma}{-}$ do not. That is, $\gamma\le\tht^R_{(m,s)}\le\tht<\tht^L_{(m,s)}$.
    \item\label{thm:cif-int.c} If $(m+\half,s)$ is a proper $[\gamma,\delta]$-instability point and $(m,s)\in\NU_1^{\gamma-}$, then $\geo\from{(m,s)}\dir{L}{\tht}{+}$ and $\geo\from{(m,s)}\dir{L}{\gamma}{-}$ split at $(m,s)$ but $\geo\from{(m,s)}\dir{R}{\tht}{+}$ and $\geo\from{(m,s)}\dir{R}{\gamma}{-}$ do not. That is, $\tht\ge\tht^L_{(m,s)}\ge\gamma>\tht^R_{(m,s)}$.
    \item\label{thm:cif-int.d} If $[(m-\half,s),(m+\half,s)]$ is an instability edge but $(m,s)\not\in\NU_1^{\gamma-}\cup\NU_1^{\tht+}$, then $\geo\from{(m,s)}\dir{L}{\gamma}{-}=\geo\from{(m,s)}\dir{L}{\tht}{+}$ and $\geo\from{(m,s)}\dir{R}{\gamma}{-}=\geo\from{(m,s)}\dir{R}{\tht}{+}$ split at $(m,s)$. That is, $\gamma\le\tht^R_{(m,s)}\le\tht^L_{(m,s)}\le\tht$.
    \end{enumerate}
\end{thm}

\begin{proof}
    Part \eqref{thm:cif-int.a}. Suppose that $\geo\from{(m,s)}\dir{R}{\tht}{+}$ and $\geo\from{(m,s)}\dir{R}{\gamma}{-}$ split at $(m,s)$. Then $\geo\from{(m,s)}\dir{R}{\tht}{+}$ goes right while $\geo\from{(m,s)}\dir{R}{\gamma}{-}$ goes up. This forces $\geo\from{(m,s)}\dir{L}{\gamma}{-}$ to go up and Lemma \ref{geodefinstabedge1} implies that $[(m-\half,s),(m+\half,s)]$ is a $[\gamma,\tht]$-instability edge. A similar argument works if $\geo\from{(m,s)}\dir{L}{\tht}{+}$ and $\geo\from{(m,s)}\dir{L}{\gamma}{-}$ split at $(m,s)$. One direction is proved. 
    Observe that Lemmas \ref{biinf:geo} and \ref{properif-shck} imply that $[(m-\half,s),(m+\half,s)]$ is an instability edge in each of the situations in parts (\ref{thm:cif-int.b}-\ref{thm:cif-int.d}). Therefore,  proving these parts implies the other direction in part \eqref{thm:cif-int.a}.

    Part \eqref{thm:cif-int.b}. Lemma \ref{improperendpts} implies that $(m+\half,s)$ is an improper $[\gamma,\tht]$-instability point and Lemma \ref{+shckiffimprop} implies that $(m,s)\in\NU_1^{\tht+}$. This implies that $\geo\from{(m,s)}\dir{R}{\tht}{+}$ goes right from $(m,s)$ while $\geo\from{(m,s)}\dir{L}{\tht}{+}$ goes up, forcing $\geo\from{(m,s)}\dir{L}{\gamma}{-}$ to also go up from $(m,s)$, which means the last two geodesics do not split at $(m,s)$. Lemma \ref{properif-shck} implies that $(m+\half,s)\not\in\NU_1^{\gamma-}$ and hence $\geo\from{(m,s)}\dir{R}{\gamma}{-}$ also goes up, which means $\geo\from{(m,s)}\dir{R}{\gamma}{-}$ and $\geo\from{(m,s)}\dir{R}{\tht}{+}$ split at $(m,s)$.

    Part \eqref{thm:cif-int.c}. Since $(m,s)\in\NU_1^{\gamma-}$, $\geo\from{(m,s)}\dir{L}{\gamma}{-}$ goes up while $\geo\from{(m,s)}\dir{R}{\gamma}{-}$ goes right, forcing $\geo\from{(m,s)}\dir{R}{\tht}{+}$ to also go right from $(m,s)$.
    By Lemma \ref{+shckiffimprop}, $(m,s)\not\in\NU_1^{\tht+}$ and hence  $\geo\from{(m,s)}\dir{L}{\tht}{+}$ must also go right. Part \eqref{thm:cif-int.c} is proved.

    Part \eqref{thm:cif-int.d}. By Lemma \ref{geodefinstabedge1}, $\geo\from{(m,s)}\dir{R}{\tht}{+}$ goes right while $\geo\from{(m,s)}\dir{L}{\gamma}{-}$ goes up from $(m,s)$. Since $(m,s)$ is not a $\tht+$ nor a $\gamma-$ shock point, $\geo\from{(m,s)}\dir{L}{\tht}{+}$ must go right and $\geo\from{(m,s)}\dir{R}{\gamma}{-}$ must go up. The claim follows now from Lemma \ref{geo-split}.

    The claims about the relationship between $\gamma$, $\tht$, $\tht_{(m,s)}^L$, and $\tht_{(m,s)}^R$ all come from Theorem \ref{thm:cif-SS}.
\end{proof}

We can now prove Lemmas \ref{IGint=union} and \ref{IG-approx}.

\begin{proof}[Proof of Lemma \ref{IGint=union}]
    Lemma \ref{mono-inst-graph} implies that for any $\tht>0$, $\IG{\tht}\subset\IG{[\gamma,\delta]}$. 
    Thus, we have the inclusion $\bigcup_{\tht\in[\gamma,\delta]\cap\baddir}\IG{\tht}\subset\IG{[\gamma,\delta]}$. We prove the other inclusion.

    Suppose $[(m-\half,s),(m+\half,s)]$ is a $[\gamma,\delta]$-instability edge. By Theorem \ref{thm:cif.int}\eqref{thm:cif-int.a}, there exists a $\tht\in[\gamma,\delta]\cap\{\tht^L_{(m,s)},\tht^R_{(m,s)}\}$. By the same theorem, applied now to $\gamma=\tht$, $\tht\in\{\tht^L_{(m,s)},\tht^R_{(m,s)}\}$ implies that $[(m-\half,s),(m+\half,s)]$ is a $\tht$-instability edge and, in particular, $\tht\in\baddir$. Thus, we have shown that all vertical instability edges in $\IG{[\gamma,\delta]}$ are also in 
    $\bigcup_{\tht\in[\gamma,\delta]\cap\baddir}\IG{\tht}$.

    Next, suppose $(m+\half,s)$ is an improper $[\gamma,\delta]$-instability point. By Theorem \ref{inst:summary.int}\eqref{inst:summary.int.I}, $(m+\half,s)$ is either a left or a right endpoint of a maximal $[\gamma,\delta]$-instability interval and by part \eqref{inst:summary.int.biinf} of the same theorem, either $[(m-\half,s),(m+\half,s)]$ or $[(m+\half,s),(m+\frac32,s)]$ is a $[\gamma,\delta]$-instability edge. By what we proved above, this edge is in $\IG{\tht}$ for some $\tht\in[\gamma,\delta]\cap\baddir$ and hence so is $(m+\half,s)$. Thus, improper $[\gamma,\delta]$-instability points are all in $\bigcup_{\tht\in[\gamma,\delta]\cap\baddir}\IG{\tht}$.
    
    Now suppose $(m+\half,s)$ is a proper $[\gamma,\delta]$-instability point, i.e.\ $\bus{\gamma}{-}(m,s,m+1,s)<\bus{\delta}{+}(m,s,m+1,s)$. We will prove that there exists a $\tht'\in[\gamma,\delta]\cap\baddir$ such that $(m+\half,s)$ is also a proper $\tht'$-instability point.
    
    The above claim holds if $(m+\half,s)$ is a proper $\delta$-instability point or if it is a proper $\gamma$-instability point. Assume, therefore, that $(m+\half,s)$ is neither. Thus, $\bus{\delta}{-}(m,s,m+1,s)=\bus{\delta}{+}(m,s,m+1,s)$ and $\bus{\gamma}{-}(m,s,m+1,s)=\bus{\gamma}{+}(m,s,m+1,s)$. Let 
        \[\tht'=\inf\bigl\{\kappa\in[\gamma,\delta]:\bus{\kappa}{+}(m,s,m+1,s)=\bus{\delta}{-}(m,s,m+1,s)\bigr\}.\]
    By Theorem \ref{thm:B}\eqref{B:jump}, $\bus{\kappa}{+}(m,s,m+1,s)=\bus{\delta}{-}(m,s,m+1,s)$ for $\kappa<\delta$ close enough to $\delta$. Hence, the set in the above infimum is not empty and $\tht'<\delta$. Similarly, 
        \[\bus{\kappa}{+}(m,s,m+1,s)=\bus{\gamma}{+}(m,s,m+1,s)=\bus{\gamma}{-}(m,s,m+1,s)<\bus{\delta}{+}(m,s,m+1,s)\]
    for $\kappa>\gamma$ close enough to $\gamma$. Hence, $\tht'>\gamma$.
    
    By Theorem \ref{thm:B}\eqref{B:cont2}, $\bus{\kappa}{+}(m,s,m+1,s)$ is right-continuous in $\kappa$.
    Consequently, $\bus{\tht'}{+}(m,s,m+1,s)=\bus{\delta}{-}(m,s,m+1,s)$.
    If $\bus{\tht'}{-}(m,s,m+1,s)$ were equal to $\bus{\delta}{-}(m,s,m+1,s)$, Theorem \ref{thm:B}\eqref{B:jump} would imply that $\bus{\kappa}{+}(m,s,m+1,s)=\bus{\tht'}{-}(m,s,m+1,s)=\bus{\delta}{-}(m,s,m+1,s)$ for $\kappa<\tht'$ close enough to $\tht'$. This would contradict the definition of $\tht'$. Hence, it must be the case that 
    $\bus{\tht'}{-}(m,s,m+1,s)<\bus{\tht'}{+}(m,s,m+1,s)$, which means $(m+\half,s)$ is a proper $\tht'$-instability point and is hence in $\bigcup_{\tht\in[\gamma,\delta]\cap\baddir}\IG{\tht}$.
\end{proof}

\begin{proof}[Proof of Lemma \ref{IG-approx}]
     Lemma \ref{mono-inst-graph} implies that $\IG{\tht}\subset\bigcap_{\gamma<\tht<\delta}\IG{[\gamma,\delta]}$. We prove the other inclusion.

    Suppose $[(m-\half,s),(m+\half,s)]$ is an edge in $\IG{[\gamma,\delta]}$ for all $\delta>\gamma>0$ with $\tht\in(\gamma,\delta)$. Then by Theorem \ref{thm:cif.int}\eqref{thm:cif-int.a}, $[\gamma,\delta]\cap\{\tht^L_{(m,s)},\tht^R_{(m,s)}\}\ne\varnothing$ for all such $\gamma$ and $\delta$, which implies that $\tht\in\{\tht^L_{(m,s)},\tht^R_{(m,s)}\}$. The same theorem implies then that $[(m-\half,s),(m+\half,s)]$ is an edge in $\IG{\tht}$. 

    Suppose next $(m+\half,s)\in\bigcap_{\gamma<\tht<\delta}\IG{[\gamma,\delta]}$. By Theorem \ref{inst:summary.int}\eqref{inst:summary.int.I}, if $(m+\half,s)$ is an improper $[\gamma,\delta]$-instability point, then it is either a left or a right endpoint of a maximal $[\gamma,\delta]$-instability interval and by part \eqref{inst:summary.int.biinf} of the same theorem, either $[(m-\half,s),(m+\half,s)]$ or $[(m+\half,s),(m+\frac32,s)]$ is a $[\gamma,\delta]$-instability edge. By Lemma \ref{mono-inst-graph}, this edge also belongs to $\IG{[\gamma',\delta']}$ for all $\gamma'\in(0,\gamma)$ and all  $\delta'>\delta$. Thus, if for  $\delta_k=\tht+1/k$ and $\gamma_k=\tht-1/k$, $(m+\half,s)$ is an improper $[\gamma_k,\delta_k]$-instability point for infinitely many $k\in\N$, then it must be the case that either $[(m-\half,s),(m+\half,s)]$ is a $[\gamma,\delta]$-instability edge for all $\delta>\gamma>0$ or $[(m+\half,s),(m+\frac32,s)]$ is a $[\gamma,\delta]$-instability edge for all $\delta>\gamma>0$. By what we proved above, this edge is in $\IG{\theta}$ and therefore so is $(m+\half,s)$. 

    It remains to consider the case where there exists an $\e>0$ such that $(m+\half,s)$ is a proper $[\gamma,\delta]$-instability point for all $\delta>\gamma>0$ with $0<\tht-\e<\gamma<\tht<\delta<\tht+\e$. This implies that for all such $\gamma$ and $\delta$, $\bus{\delta}{+}(m,s,m+1,s)>\bus{\gamma}{-}(m,s,m,s+1)$.  By Theorem \ref{thm:B}\eqref{B:jump}, $\bus{\delta}{+}(m,s,m+1,s)=\bus{\tht}{+}(m,s,m+1,s)$ for $\delta>\tht$ close enough to $\tht$ and $\bus{\gamma}{-}(m,s,m+1,s)=\bus{\tht}{-}(m,s,m+1,s)$ for $\gamma<\tht$ close enough to $\tht$. Hence, $\bus{\tht}{+}(m,s,m+1,s)>\bus{\tht}{-}(m,s,m,s+1)$ and $(m+\half,s)$ is a proper $\tht$-instability point. The lemma is now proved.
\end{proof}

\section{Proofs of general BLPP results}\label{sec:aux}

\begin{proof}[Proof of Proposition \ref{nodouble}]
    Note that $L_{(m,s),(n,t)}=L_{(m,s),(n-2,t)}$ implies  $L_{(m,s),(n,t)}=L_{(m,s),(n-1,t)}$. 
    On the event in part \eqref{nodouble-a}, take
        \[m'=\min\{\ell\in[m,n]\cap\Z:L_{(m,s),(n,t)}=L_{(m,s),(\ell,t)}\}.\]
    Then $m\le m'\le n-2$ and 
    $L_{(m,s),(\ell,t)}=L_{(m,s),(n,t)}$ for all integers $\ell\in[m',n]$. In particular, $L_{(m,s),(m',t)}
        =L_{(m,s),(m'+2,t)}$. This means that there exists a geodesic path from $(m,s)$ to $(m'+2,t)$ that goes through $(m',t)$.

    If $m'>m$, then the definition of $m'$ implies that $L_{(m,s),(m'-1,t)}<L_{(m,s),(m'+2,t)}$. This implies that there are no  geodesic paths from $(m,s)$ to $(m'+2,t)$ that go through $(m'-1,t)$. Thus, the geodesic from the previous paragraph must arrive at level $m'$ at some $s'\in[s,t)$. It then must proceed from $(m',s')$ to $(m',t)$ and hence must cross $(m',s'')$ for some integer $s''\in(s',t)$.  Then we have $L_{(m',s''),(m',t)}=L_{(m',s''),(m'+2,t)}$.
    We have thus shown that on the event in part \eqref{nodouble-a}, we must have $L_{(m',s''),(m',t)}=L_{(m',s''),(m'+2,t)}$ for some integer $m'$, some rational $s''$, and some $t>s''$.
    Thus, to deduce the claim, it is enough to prove that for each integer $m'$ and rational $s''$, there is zero $\P$-probability that there exists a $t>s''$ with  $L_{(m',s''),(m',t)}=L_{(m',s''),(m'+2,t)}$. By shift-invariance, we can assume $m'=0$ and $s''=0$. 
    Therefore, to prove that the event in part \eqref{nodouble-a} has zero probability, it suffices to show that
    \begin{align}\label{aux0001}
    \P\{\exists t>0:L_{(0,0),(0,t)}=L_{(0,0),(2,t)}\}=0.
    \end{align}
    Similarly, to show that the event in part \eqref{nodouble-b} has zero probability, it suffices to prove
    \[\P\{\exists t<0:L_{(-2,t),(0,0)}=L_{(0,t),(0,0)}\}=0.\]
    This follows from \eqref{aux0001} 
    by reflection symmetry. Therefore, we only show the proof of \eqref{aux0001}.
    
    Define $ Y_k(t)=L_{(0,0),(k-1,t)} $ for $k\in\{1,2,3\}$. 
    Observe that, $\P$-almost surely,  
    $Y_1(t)\le Y_2(t)\le Y_3(t)$ for all $t\ge0$. Therefore, we 
    want to show that 
    \begin{align}\label{no-triple}
    \P\{\exists t>0:Y_1(t)=Y_2(t)=Y_3(t)\}=0.
    \end{align}
    When the above event occurs, we say that there is \emph{triple collision}.

    Denoting the i.i.d.\ Brownian environment on the first three levels by $B_j$, $j\in\{0,1,2\}$, we have
    \begin{align*}
        &Y_1(t)=B_0(t),\\
        &Y_2(t)=\sup_{0\le s\le t}\{B_0(s)+B_1(t)-B_1(s)\}=B_1(t)+A_{1,2}(t),\quad\text{and}\\
        &Y_3(t)=\sup_{0\le s'\le s\le t}\{B_0(s')+B_1(s)-B_1(s')+B_2(t)-B_2(s)\}
        =B_2(t)+A_{2,3}(t),
    \end{align*}
    where 
    \[A_{1,2}(t)=\sup_{0\le s\le t}\{Y_1(s)-B_1(s)\}\quad\text{and}\quad
    A_{2,3}(t)=\sup_{0\le s\le t}\{Y_2(s)-B_2(s)\}.\]
    Note that $A_{1,2}(0)=A_{2,3}(0)=0$ and both $A_{1,2}$ and $A_{2,3}$ are nondecreasing and only increase at $t\ge0$ such that  $Y_2(t)=Y_1(t)$ and $Y_3(t)=Y_2(t)$, respectively.  
    Thus, the processes $Y_k$, $k\in\{1,2,3\}$, satisfy Definition 1.6 in \cite{Sar-15} with drift coefficients $g_k=0$, diffusion coefficients $\sigma_k^2=1$, and collision coefficients $q_k^+=1$ and $q_k^-=0$ for $k\in\{1,2,3\}$. 
    Then \eqref{no-triple} follows from Theorem 1.9(i) in \cite{Sar-15}. Strictly speaking, the display above Definition 1.6 in \cite{Sar-15} requires $q_k^+,q_k^-\in(0,1)$. The proof of the theorem, however, does not use this condition and works as is in the extreme case where $q_k^+=1$ and $q_k^-=0$. For the convenience of the interested reader, we give a quick sketch of the proof that fleshes out why this works.
    	
     Let $Z_1(t)=Y_2(t)-Y_1(t)$ and $Z_2(t) = Y_3(t)-Y_2(t)$. Then, following Definition 2.8 in \cite{Sar-15}, $Z(t)=(Z_1(t),Z_2(t))\in\R_+^2$ is called the \emph{gap process}.
    A triple collision is equivalent to having $Z(t)=(0,0)$ and, therefore, we are aiming to show that 
    \begin{align}\label{no-triple2}
    \P\{\exists t>0:Z(t)=(0,0)\}=0.
    \end{align}
     
     By Lemma 2.9 in \cite{Sar-15}, the gap process $Z$ is a two-dimensional \emph{semimartingale reflected Brownian motion} in $\R_+^2$ (Definition 2.3 in \cite{Sar-15}) with zero drift, \emph{reflection matrix}
	\[R = 
	\begin{bmatrix}
	1 & 0\\
	-1 & 1
	\end{bmatrix},\]
 and covariance matrix  
    \[A = 
	\begin{bmatrix}
	2 & -1\\
	-1 & 2
	\end{bmatrix}.\]
 The proof of the lemma does not need the assumption that the collision coefficients are in $(0,1)$ and works for the extreme case where $q_k^+=1$ and $q_k^-=0$.

    Let $D$ be the diagonal matrix with the same diagonal as $A$, i.e.\ $D$ is twice the identity matrix. 
    Then we have $RD+DR' = 2A$. 
     By Definition 2.7 in \cite{Sar-15}, the non-smooth part of $\R_+^2$ is the set consisting of just the origin $\{(0,0)\}$.  Now \eqref{no-triple2} follows from Theorem 2.12(i) (and Remark 2.13). The proof of the theorem only depends on the fact that 
     \[Q=I-R=\begin{bmatrix}
	0 & 0\\
	1 & 0
	\end{bmatrix}\]
    has nonnegative entries and a spectral radius less than one. (See Lemma 2.5(iii) in \cite{Sar-15}.)
\end{proof}

\begin{proof}[Proof of Lemma \ref{geo-split}]
First, note that Proposition \ref{nodouble} prohibits geodesics from taking multiple consecutive vertical steps. Therefore, if $ \geo\from{(m,t)}\dir{R}{\tht}{\sig} $ and $ \geo\from{(m,t)}\dir{L}{\tht}{\sig} $ go vertically together they have to next go to the right together.  But Theorem 4.8(iv) in \cite{Sep-Sor-23-pmp} says that the two geodesics can only separate  along the upward vertical ray from $(m,t)$. 
Therefore, they either separate immediately at $(m,t)$ or else they have to stay together forever.
\end{proof}

\begin{lem}\label{nofixedjump}
Fix integers $m<k\le n$ and real numbers $s<r<t$.
Then $\P$-almost surely, $(k,r)$ is not a jump time for $\Gamma_{(m,s)}^{(n,t)}$, meaning $r$ it is not the first time $\Gamma_{(m,s)}^{(n,t)}$ enters level $k$.
\end{lem}

\begin{proof}
    Due to Proposition \ref{nodouble}, $\P$-almost surely, if $(k,r)$ is the first entry point of $\Gamma_{(m,s)}^{(n,t)}$ at level $k$, then there must exist rational $r'<r''$ with $r'<r<r''$ and such that $\Gamma_{(m,s)}^{(n,t)}$ goes through $(k-1,r')$, $(k-1,r)$, $(k,r)$, and $(k,r'')$. 
    Thus, it is enough to show that for each $k\in\Z$ and $r'<r<r''$,  
    \[\P\bigl\{\{(k-1,r),(k,r)\}\subset\Gamma_{(k-1,r')}^{(k,r'')}\bigr\}=0.\]
    Using shift-invariance, this is reduced to showing that for any $s<t$, 
    \begin{align}\label{max0}
    \P\bigl\{\{(0,s),(1,s)\}\subset\Gamma_{(0,0)}^{(1,t)}\bigr\}=0.
    \end{align}

    In the above event, $s$ is the first point where $B_0(s)+B_1(t)-B_1(s)$ reaches is maximum over the interval $[0,t]$. This is the same as the location of the maximum of $B_0(s)-B_1(s)$ over $[0,t]$. Since $B_0-B_1$ is itself a Brownian motion (with diffusion coefficient two), the location of its maximum has a continuous distribution. This proves \eqref{max0}.
\end{proof}

The next proof follows the strategy of Proposition 34 in \cite{Bha-24}, where the non-existence of bi-infinite geodesics was established for the directed landscape. In that setting, the key input is Lemma 3.12 in \cite{Gan-Zha-22-}, which provides a bound on the number of disjoint geodesics.
To obtain the analogous bound for BLPP, we appeal to a theorem of \cite{Ham-20}, the proof of which uses the line ensemble technique. Alternatively, one could follow the proof of Lemma 3.12 in \cite{Gan-Zha-22-}, which uses Theorem 3.10 in \cite{Bas-Hof-Sly-22} together with the scaling limit from exponential last-passage percolation to the directed landscape; in that approach, one would replace the latter scaling step by a direct scaling to BLPP. This route would involve fewer explicitly solvable features of the model. 
A third possibility is to adapt the method of \cite{Bal-Bus-Sep-20}, which depends even less on integrable structure. In any case, the non-existence of bi-infinite geodesics is expected to hold quite generally, well beyond the class of solvable models.

\begin{rmk}\label{rk:lrmllm}
If one proceeds via estimates on the number of disjoint geodesics, it becomes necessary to restrict attention to locally rightmost or locally leftmost geodesics, in order to account for the possibility of non-uniqueness between certain pairs of points along a given geodesic. If such pairs of non-uniqueness do not occur along an arbitrary semi-infinite geodesic---as is known to be the case, for example, in the directed landscape model---then any bi-infinite geodesic must be unique between every pair of its points. In particular, it would be both locally rightmost and locally leftmost, and the theorem would therefore rule out the existence of any bi-infinite geodesic.
For instance, if the conjecture stated in Remark \ref{rk:all geo} holds, then all geodesics are Busemann geodesics, and Lemma \ref{geo-split} shows that such geodesics do not admit pairs of points with non-uniqueness.
\end{rmk}

To prove Theorem \ref{nobiinf}, we need the following result.
For integers $k<\ell$ and intervals $I,J\subset\R$, let $D(k,I;\ell,J)$ denote the maximal number of disjoint geodesics that start at $(k,x)$ and end at $(\ell,y)$ with $x\in I$ and $y\in J$. 

\begin{thm}\label{thm:Ham-adapt}
    Fix $Q_0>0$. 
   There exist positive constants $c_1,c_2,m_0$ such that for any $q\in(0,Q_0)$ and positive integers $m\ge m_0$ and $n$ such that $(2/q)^3\le m\le n^{c_1}$, 
    we have the bound 
   \begin{align}\label{Ham-new}
   \P\bigl\{D\bigl(0,[0,2n^{2/3}];[qn,qn+2n^{2/3}]\bigr)\ge m\bigr\}
    \le m^{-c_2(\log\log m)^2}.
    \end{align}
\end{thm}

\begin{proof}
Set $h=1/q$. Then by
Theorem 1.4 in \cite{Ham-20}, with $t_1=0$, $t_2=1$, and $x=y=0$, there exist positive constants $c_1$, $c_2$, and $m_0$, such that, under the conditions of our theorem, we get 
    \begin{align}\label{Ham}
    \P\bigl\{D\bigl(0,[0,2hn^{2/3}];[n,n+2hn^{2/3}]\bigr)\ge m\bigr\}
    \le m^{-c_2(\log\log m)^2}.
    \end{align}

\begin{rmk}\label{rmk:Ham-Thm}
Theorem 1.4 in \cite{Ham-20} is stated for positive integer values of $h$, and its proof is given within the proof of Theorem 6.2 of that paper. A careful inspection of the argument shows that neither the integrality of $h$ nor the restriction $h \ge 1$ is essential. In fact, the theorem remains valid for any fixed $h_0 > 0$ and all $h \ge h_0$, provided the constants $G$ and $g$ are allowed to depend on $h_0$.
For the convenience of the reader, Appendix \ref{Ham-new-app} contains a variant that admits a comparatively simple proof. Using this version in place of \cite[Theorem 1.4]{Ham-20} yields Theorem \ref{thm:Ham-adapt} with $(2/q)^3$ replaced by $(2/q)^{16}$, while allowing $q$ to be any positive real number. This formulation is sufficient for the proof of Theorem \ref{nobiinf}.
\end{rmk}
    
Now, we perform a Brownian scaling to get to \eqref{Ham-new}.
Specifically, given the i.i.d.\ two-sided standard Brownian motions $\{B_k:k\in\Z\}$, define $\wt B_k(t)=q^{1/2} B_k(q^{-1}t)$, $t\in\R$. 
	  If for some integers $m<n$ and real times $s<t$, the passage time $L_{(m,s),(n,t)}$ in \eqref{LPP} is maximized by the times $s=s_{m-1}\le s_m\le\cdots\le s_n=t$, then the passage time $\wt L_{(m,qs),(n,qt)}$, using Brownian motions $\{\wt B_k:k\in\Z\}$, is maximized by the times $qs=qs_{m-1}\le qs_m\le\cdots\le qs_n=qt$ (and equals $q^{1/2} L_{(m,s),(n,t)}$). Consequently, scaling the geodesics in the environment given by $\{B_k:k\in\Z\}$ by a factor $q$ turns them into geodesics in the environment given by $\{\wt B_k:k\in\Z\}$. This  and the fact that the Brownian motions $\{\wt B_k:k\in\Z\}$ have the same distribution as the original Brownian motions $\{B_k:k\in\Z\}$, give us that the probabilities in \eqref{Ham-new} and \eqref{Ham} are equal.
\end{proof}

\begin{proof}[Proof of Theorem \ref{nobiinf}]
By Theorem 3.1(iv–v) in \cite{Sep-Sor-23-aihp}, it follows that, with $\P$-probability one, every nondegenerate semi-infinite geodesic---and hence every bi-infinite geodesic---is $\tht$-directed for some $\tht>0$.
Since the model is invariant under the reflections $n\mapsto-n$, $n\in\Z$, and $t\mapsto-t$, $t\in\R$, we see that any non-degenerate bi-infinite geodesic must also be $\tht'$-directed in the south-west direct, for some $\tht'>0$.

One consequence of the above directedness properties is that any non-degenerate bi-infinite geodesic must cross all levels $n\in\Z$. We can then encode the bi-infinite geodesic by its entry points at the various levels $n\in\Z$, which we will denote by $\tau_n$.

We now sketch the proof strategy at a high level. Using directedness and the fact that geodesics have transversal fluctuation exponent $2/3$, we show that for any positive $\tht$ and $\tht'$, the probability that a geodesic starting between $-\tht' n \pm n^{2/3}$ at level $-n$ and ending between $\tht n \pm n^{2/3}$ at level $n$ passes through $[-1,1]$ at level $0$ tends to zero as $n\to\infty$. Since the entry point at level $0$ fluctuates on scale $n^{2/3}$, one expects only $O(1)$ such geodesics to pass through $[-n^{2/3}, n^{2/3}]$ at level $0$. Partitioning this interval into $n^{2/3}$ subintervals of fixed size and using shift-invariance, the expected number passing through $[-1,1]$ is therefore of order $n^{-2/3}$, which vanishes as $n\to\infty$. Although the directions $\tht$ and $\tht'$ are not known a priori, one can instead work with intervals on scale $n$, partition them into subintervals of length $n^{2/3}$, and apply union bounds, which remain sufficient to reach the same conclusion. We now give the rigorous argument.

We will argue by contradiction.  We consider the case of locally rightmost geodesics, the case of locally leftmost ones being  identical. Suppose the probability there exists a non-degenerate bi-infinite locally rightmost geodesic $(\tau_n)_{n\in\Z}$ is positive. Denote it by $\delta$. Then, since $\tau_0$ is finite and, for some positive $\tht$ and $\tht'$, $\tau_n/n\to\tht$ as $n\to\infty$ and $\tau_n/n\to\tht'$ as $n\to-\infty$, we must have
\begin{align*}
\P\bigl\{&\exists\text{ a locally rightmost bi-infinite geodesic }(\tau_n)_{n\in\Z}:\\
&\quad|\tau_0|\le a,\,\e n\le\tau_n\le Mn,\ -Mn\le\tau_{-n}\le-\e n\bigr\}\ge\delta/2,
\end{align*}
for all sufficiently small $\e\in(0,1)$ and sufficiently large  positive $a$, $M>1$, and $n>1$. 
This implies  
\begin{align}\label{2343}
\P\{A_n(0,a,\e,M)\}\ge\delta/2,
\end{align}
where $A_n(r,a,\e,M)$ is the event that there exist $s\in[-Mn,-\e n]$, $t\in[\e n,Mn]$, and a rightmost geodesic from $(-n,s)$ to $(n,t)$ that enters level $0$ between $r-a$ and $r+a$. 
To obtain the desired contradiction, we follow the high-level strategy described at the outset and prove that, for any $a>0$, $\e\in(0,1)$, and $M>1$, the probability in \eqref{2343} tends to $0$ as $n\to\infty$. Fix such $a$, $\e$, and $M$.

Fix $\e'\in(0,2\e/3)$ and $M'>M+\e/3$. Let $\mathcal G_n(\e',M')$ denote the set of rightmost geodesics that make up the event $A_n(0,\e' n/2,\e',M')$. Note that this set can be empty. For each geodesic $\gamma\in\mathcal G_n(\e',M')$, let $\tau_0(\gamma)$ denote its entry point at level $0$. Thus,  $|\tau_0(\gamma)|\le \e' n/2$. 
Let $H_n(\e',M')$ denote the cardinality of the set 
\[\mathcal H_n(\e',M')=\{\tau_0(\gamma):\gamma\in \mathcal G_n(\e',M')\}\subset[-\e' n/2,\e' n/2].\]
In words, this is the set of distinct entry points to level $0$ of all the geodesics in the set $\mathcal G_n(\e',M')$. (Naturally, if $\mathcal G_n(\e',M')=\varnothing$, then $H_n(\e',M')=0$.) We aim to upper bound $\E[H_n(\e',M')]$.

Consider two distinct points $s \neq t$ in the $\mathcal H_n(\e',M')$. There exist two rightmost geodesics $\gamma$ and $\gamma'$ in $\mathcal G_n(\e',M')$ such that $\tau_0(\gamma)=s$ and $\tau_0(\gamma')=t$. Suppose that the segments of $\gamma$ and $\gamma'$ from $(0,s)$ and $(0,t)$, respectively, up to the points where they reach level $n$, intersect. Then the segments from the points where the geodesics leave level $-n$ up to $(0,s)$ and $(0,t)$ must be disjoint. Indeed, if these earlier portions were also to intersect, the rightmost property of $\gamma$ and $\gamma'$ would be violated. Likewise, if the segments from level $-n$ to  $(0,s)$ and $(0,t)$ intersect, then the remaining segments to level $n$ are disjoint.
\[H_n(\e',M')\le D(-n,[-M'n,-\e' n];0,[-\e' n/2,\e' n/2])
+D(0,[-\e' n/2,\e' n/2];n,[\e' n,M'n]).\]
The expected values of the two quantities on the right-hand side are bounded similarly. We upper bound the second one. For this, we will apply Theorem \ref{thm:Ham-adapt} and, therefore, we go from scale $n$ down to scale $n^{2/3}$.

Cover $[-\e' n/2,\e' n/2]$ and $[\e' n,M'n]$ by closed intervals of width $2n^{2/3}$. Since $\e'$ and $M'-\e'$
are both $<2M'$, once $n$ is large enough (depending only on $\e'$ and $M'$), we need no more than $M' n^{1/3}$ intervals. Let $[\alpha,\alpha+2n^{2/3}]$ be an interval in the first covering and let $[\beta,\beta+2n^{2/3}]$ be an interval in the second covering. Then $-\e' n/2\le \alpha\le\e' n/2$ and $\e' n\le \beta\le M'n$. 

Using shift-invariance in the first equality, then taking $Q_0=M+\e'/2$ and $q = (\beta-\alpha)/n$ in Theorem \ref{thm:Ham-adapt} gives 
\begin{align}\label{2464}
\begin{split}
&\P\bigl\{D(0,[\alpha,\alpha+2n^{2/3}];n,[\beta,\beta+2n^{2/3}])\ge m\bigr\}\\
&\qquad\qquad=\P\bigl\{D(0,[0,2n^{2/3}];n,[\beta-\alpha,\beta-\alpha+2n^{2/3}])\ge m\bigr\}\\
&\qquad\qquad\le m^{-c_2(\log\log m)^2},
\end{split}
\end{align}
for all $q\in[\e'/2,M+\e'/2]$ and all positive integers $m$ and $n$ with 
\[m_1=\max\bigl((4/\e')^3,m_0\bigr)\le m\le n^{c_1}.\]

Since $\beta\ge\e' n>\e' n/2+2n^{2/3}\ge \alpha+2n^{2/3}$ for $n$ sufficiently large (depending on $\e'$), any geodesic from $[\alpha,\alpha+2n^{2/3}]$ to $[\beta,\beta+2n^{2/3}]$ must cross the vertical line $\e' n e_1+\Z_+ e_2$. Since there are only $n+1$ possible crossing points between levels $0$ and $n+1$, we have that $D(0,[\alpha,\alpha+2n^{2/3}];n,[\beta,\beta+2n^{2/3}])\le n+1$. This and the upper bound \eqref{2464} give 
\begin{align*}
\E\bigl[D(0,[\alpha,\alpha+2n^{2/3}];n,[\beta,\beta+2n^{2/3}])\bigr]
&=\sum_{m=1}^{n+1}\P\bigl\{D(0,[\alpha,\alpha+2n^{2/3}];n,[\beta,\beta+2n^{2/3}])\ge m\bigr\}\\
&\le m_1+\sum_{m=m_1}^{n^{c_1}}m^{-c_2(\log\log m)^2}+(n+1)\cdot {\bigl(n^{c_1}\bigr)}^{-c_2(\log\log n^{c_1})^2}\\
&\le m_1+\sum_{m=m_1}^\infty m^{-c_2(\log\log m)^2}+1,
\end{align*}
for all large enough $n$. Denote the constant on the right-hand side by $C$.

We can bound $D(0,[-\e' n/2,\e' n/2];n,[\e' n,M'n])$ by the sum of $D(0,[\alpha,\alpha+2n^{2/3}];n,[\beta,\beta+2n^{2/3}])$, where $[\alpha,\alpha+2n^{2/3}]$ and $[\beta,\beta+2n^{2/3}]$ vary over the coverings of the intervals $[-\e' n/2,\e' n/2]$ and $[\e' n,M'n]$. Therefore, 
\[\E\bigl[D(0,[-\e' n/2,\e' n/2];n,[\e' n,M'n])\bigr]
\le C(M'n^{1/3})^2.\]
From this, we get that
\[\E[H_n(\e',M')]\le 2C(M')^2 n^{2/3},\]
for all sufficiently large $n$. 

Take $\delta>0$ and $n>(2a+\delta)/\e'$. Let $\ell_n=\lfloor\tfrac{\e'n+\delta}{2a+\delta}\rfloor$, $r_1=-\e'n/2+a$, and for integers $i\in[2,\ell_n]$, let $r_i=r_{i-1}+2a+\delta$. Then the intervals $[r_i-a,r_i+a]$ are all disjoint and contained inside $[-\e'n/2,\e'n/2]$. By shift-invariance, for each such interval, $A_n(r_i,a,\e',M')$ has the same probability as the event on which there exists a rightmost geodesic that starts inside $[-M'n-r_i,-\e'n-r_i]$ on level $-n$, goes through $[-a,a]$ on level $0$, and ends inside $[\e'n-r_i,M'n-r_i]$ on level $n$. By our choice of $\e'$ and $M'$ and the fact that $|r_i|\le\e'n/2$, we have that 
\[[-Mn,-\e n]\subset[-M'n-r_i,-\e'n-r_i]\quad\text{and}\quad[\e n,Mn]\subset[\e'n-r_i,M'n-r_i].\]
Therefore,  
\[\P\bigl\{A_n(0,a,\e,M)\}\le\P\bigl\{A_n(r_i,a,\e',M')\}.\] 
From this, we have
\begin{align*}
\P\bigl\{A_n(0,a,\e,M)\}
&\le\frac1{\ell_n}\sum_{i=1}^{\ell_n}\P\bigl\{A_n(r_i,a,\e',M')\bigr\}
=\frac1{\ell_n}\E\Bigl[\sum_{i=1}^{\ell_n}\one_{A_n(r_i,a,\e',M')}\Bigr]\\
&\le\frac1{\ell_n}\E[H_n(\e',M')]\le\frac{4(2a+\delta)C(M')^2}{\e' n^{1/3}}\,,
\end{align*}
where the last bound holds for large enough $n$. This contradicts \eqref{2343}, proving the theorem.
%
%
\end{proof}

{\it Acknowledgments.} We thank Timo Sepp\"al\"ainen and Evan Sorensen for valuable discussions and comments. We thank Evan Sorensen for pointing us to \cite{Sar-15}, which helped prove Proposition \ref{nodouble}. 

\appendix
\section{Two additional results}

We give here two more results that are not needed for our development, but that may be of independent interest.

\begin{lem}\label{no-geo}
    There exists an event $\Omega_2'\in\kS$ with $\Omega_2'\subset\Omref{Omega1}$, $\P(\Omega_2')=1$, and such that the following holds for all $\w\in\Omega_2'$. For any $\tht>0$, $\sigg\in\{-,+\}$, $(m,s)\in\Z\times\R$, and any semi-infinite geodesic $\geo$ emanating from $(m,s)$, if $\geo\from{(m,s)}\dir{L}{\tht}{\sig}\preceq\geo\preceq\geo\from{(m,s)}\dir{R}{\tht}{\sig}$, then $\geo\in\bigl\{\geo\from{(m,s)}\dir{L}{\tht}{\sig},\geo\from{(m,s)}\dir{R}{\tht}{\sig}\bigr\}$.
\end{lem}

\begin{proof}
 Suppose first that $(m+1,s)\in\geo$.
 Then the coalescence in Theorem \ref{thm:geo}\eqref{thm:geo:coal} implies that  
 $\geo\from{(m+1,s)}\dir{L}{\tht}{\sig}=\geo\from{(m+1,s)}\dir{R}{\tht}{\sig}$. 
 This forces $\geo$ to equal these two geodesics from $(m+1,s)$ onwards. But then $\geo=\geo\from{(m,s)}\dir{L}{\tht}{\sig}$. 
 
Similarly, if $\geo$ moves to the right from $(m,s)$, then there exists a $t>s$ such that $(m,s)\in\geo$. Again, coalescence implies that  $\geo\from{(m,t)}\dir{L}{\tht}{\sig}=\geo\from{(m,t)}\dir{R}{\tht}{\sig}$, which 
forces $\geo=\geo\from{(m,s)}\dir{R}{\tht}{\sig}$. 
\end{proof}

For the next result, we use the natural order relation on $\{-,+\}$ which says that $-\preceq+$.

\begin{lem}
    There exists an event $\Omega_2''\in\kS$ with $\Omega_2''\subset\Omref{Omega1}$, $\P(\Omega_2'')=1$, and such that the following holds for all $\w\in\Omega_2''$. Take $\delta>\gamma>0$, $\tht\in[\gamma,\delta]$, and $\sig'\preceq\sig\preceq\sig''$ in $\{-,+\}$. Then $\NU_1^{\gamma\sig'}\cap\NU_1^{\delta\sig''}\subset\NU_1^{\tht\sig}$.
\end{lem}

\begin{proof}
    Take $(m,s)\in\NU_1^{\gamma\sig'}\cap\NU_1^{\delta\sig''}$. Then $\geo\from{(m,s)}\dir{L}{\gamma}{\sig'}$ and $\geo\from{(m,s)}\dir{L}{\delta}{\sig''}$ both go through $(m+1,s)$ while $\geo\from{(m,s)}\dir{R}{\gamma}{\sig'}$ and $\geo\from{(m,s)}\dir{R}{\delta}{\sig''}$ both proceed rightwards from $(m,s)$. By the monotonicity of geodesics \eqref{geomono}, this forces $\geo\from{(m,s)}\dir{R}{\tht}{\sig}$ to go through $(m+1,s)$ and $\geo\from{(m,s)}\dir{R}{\tht}{\sig}$ to go right from $(m,s)$, which means $(m,s)\in\NU_1^{\tht\sig}$.
\end{proof}

\section{A variant of \cite[Theorem 1.4]{Ham-20}}\label{app:Ham}

In this appendix we prove a version of Theorem 1.4 in \cite{Ham-20} adapted to our needs. In particular, we obtain a statement sufficient to prove Theorem \ref{thm:Ham-adapt}, in which the condition $m\ge (2/q)^3$ is replaced by $m\ge (2/q)^{16}$. The same proof in fact works if $16$ is replaced by any power strictly more than $4$.
Our formulation thus differs from that of \cite{Ham-20}, but it is adequate for our purposes and admits a simpler proof than the argument given in Theorem 6.2 of \cite{Ham-20}. One point to note is that in our theorem, the parameter $h$ can be any positive number.

\begin{thm}\label{thm:Ham-application}
   There exist positive constants $c_1,c_2,m_0$ such that for any $h>0$ and positive integers $m\ge m_0$ and $n$ such that $(2h)^{16}\le m\le n^{c_1}$, we have the bound 
   \begin{align}\label{Ham-new-app}
   \P\bigl\{D\bigl(0,[0,2hn^{2/3}];[n,n+2hn^{2/3}]\bigr)\ge m\bigr\}
    \le \bigl(\tfrac{h}2+1\bigr)^2 m^{-c_2(\log\log m)^2}.
    \end{align}
\end{thm}

\begin{proof}
Theorem 1.1 in \cite{Ham-20} (with $t_1=0$ and $t_2=1$) says that there exists a constant $C>1$ such that for any integer $k\ge2$, any positive $\e\le C^{-4k^2}$, any integer 
\begin{align}\label{n-cond}
n\ge C^{k^2}\e^{-C}(1+\e^{-18}|\log\e|^{-24}C^{-36k}),
\end{align}
and any $x,y\in\R$ with 
\begin{align}\label{xy-cond}
|x-y|\le\e^{-1/2}|\log\e|^{-2/3}C^{-k},
\end{align}
we have the upper bound
\begin{align}\label{2562}
\begin{split}
&\P\bigl\{D\bigl(0,[2(x-\e)n^{2/3},2(x+\e)n^{2/3}];
[n+2(y-\e)n^{2/3},n+2(y+\e)n^{2/3}]\bigr)\ge k\bigr\}\\
&\qquad\qquad\qquad\qquad\qquad\qquad\qquad\qquad
\le \e^{\frac{k^2-1}2}C^{k^3}e^{C^k|\log\e|^{5/6}}.
\end{split}
\end{align}

\begin{rmk}
The constant $C$ above corresponds to the constant $G$ in \cite[Theorem~1.1]{Ham-20}. Although it is not stated explicitly there that this constant is $>1$, enlarging it only weakens the bounds required on $\e$ and $n$ and increases the upper bound in \eqref{2562}. Hence, without loss of generality, we may assume that $C>1$.
\end{rmk}

Let $m_0\ge55$ be large enough so that 
\begin{align}
&\log m\ge C^{12},\quad
\log m\ge\frac{4(\log\log m)^2}{9\log C},\label{m-cond1}\\
&\bigl(5\delta+(2\delta)^{5/6}\bigr)\log m+
\frac{(\log\log m)^3}{(6\log C)^3}\cdot\log C\le \delta\cdot\frac{(\log\log m)^2}{288(\log C)^2}\log m,\label{m-cond2}\\
&e^{\frac18\log m}e^{-\tfrac56\log\log m}
-2e^{-\frac14\log m}\ge e^{\frac1{16}\log m},\quad\text{and}\label{m-cond3}\\
&\frac{m}{(\tfrac 12 m^{5/16}+1)^2}\ge \frac{\log\log m}{6\log C},\label{m-cond4}
\end{align}
for all $m\ge m_0$. 
For $m\ge m_0$ let
\begin{align}\label{kep-def}
k=\Bigl\lfloor\frac{\log\log m}{6\log C}\Bigr\rfloor\quad\text{and}\quad\e=m^{-1/4}.
\end{align}
Then the first two inequalities in \eqref{m-cond1} imply $k\ge2$ and $\e\le C^{-4k^2}$.

Next, we check that the condition \eqref{n-cond} is satisfied. Note that since $m\ge55$, $e^{\frac92\log m}\ge(1+\sqrt 5)/2$ and hence $1+e^{\frac92\log m}\le e^{9\log m}$. Use this, $\lfloor a\rfloor\le a$, $k\ge1$, $\log m\ge4$, 
 $m\ge1$, and the second inequality in \eqref{m-cond1}, to write
\begin{align*}
  C^{k^2}\e^{-C}(1+\e^{-18}|\log\e|^{-24}C^{-36k})
  &\le e^{\frac{(\log\log m)^2}{36\log C}}e^{\tfrac14 C\log m}\bigl(1+e^{\frac92\log m}\bigr)\\
  &\le e^{(\frac1{16}+\frac{C}4+9)\log m}=m^{1/c_1},
\end{align*}
where we set $c_1=(9+\tfrac1{16}+\tfrac{C}4)^{-1}$. Thus, requiring $m\le n^{c_1}$ guarantees \eqref{n-cond}.

We now cover each of the intervals in \eqref{Ham-new-app} by smaller intervals of the form in \eqref{2562}. We thus check that \eqref{xy-cond} is satisfied for the $x$ and $y$ we use for the coverings.
Namely, we have
\[0\le x+\e,\quad x-\e\le h,\quad0\le y+\e,\quad\text{and}\quad y-\e\le h.\]
Thus,
\[|x-y|\le h+2\e.\]
Now bound
\begin{align*}
\e^{-1/2}|\log\e|^{-2/3}C^{-k}-2\e
&\ge e^{\frac18\log m}e^{-\tfrac23\log(\frac14\log m)}
e^{-\frac16\log\log m}-2e^{-\frac14\log m}\\
&\ge e^{\frac18\log m}e^{-\tfrac56\log\log m}
-2e^{-\frac14\log m}\\
&\ge e^{\frac1{16}\log m}=m^{1/16},   
\end{align*}
where the last inequality used \eqref{m-cond3}. Thus, if $m\ge h^{1/6}$, then the $x$ and $y$ used in the covering intervals satisfy \eqref{xy-cond}.

Now, apply a union bound. Specifically, the number of disjoint polymers in \eqref{Ham-new-app} is bounded above by the sum, over all pairs of covering subintervals of $[0,\,2hn^{2/3}]$ and $[n,\,n+2hn^{2/3}]$, of the numbers of disjoint polymers between each such pair. We need at most $h/(2\e)+1$ subintervals for each of the two coverings and by the condition $m\ge h^{16}$ and \eqref{m-cond4}, we have
\[\frac{m}{(\frac{h}{2\e}+1)^2}
=\frac{m}{(\tfrac h2 m^{1/4}+1)^2}
\ge\frac{m}{(\tfrac 12 m^{5/16}+1)^2}\ge \frac{\log\log m}{6\log C}=k.\]
Thus, by a union bound, and applying \eqref{2562} to each of the covering pairs, we get the upper bound
\begin{align*}
&\P\bigl\{D\bigl(0,[0,2hn^{2/3}];[n,n+2hn^{2/3}]\bigr)\ge m\bigr\}
\le\bigl(\tfrac{h}{2\e}+1\bigr)^2
\e^{\frac{k^2-1}2}C^{k^3}e^{C^k|\log\e|^{5/6}}.
\end{align*}
Substituting \eqref{kep-def}, then using  $m\ge1$ and $a-1\le\lfloor a\rfloor\le a$, then \eqref{m-cond1}, then \eqref{m-cond2}, we get 
\begin{align*}
\bigl(\tfrac{h}{2\e}+1\bigr)^2
\e^{\frac{k^2-1}2}C^{k^3}e^{C^k|\log\e|^{5/6}}
&\le\bigl(\tfrac {h}2+1\bigr)^2 e^{4\delta\log m}
e^{-\delta([\frac{\log\log m}{6\log C}-1]^2-1)\log m}
e^{\frac{(\log\log m)^3}{(6\log C)^3}\log C}
e^{(2\delta)^{5/6}\log m}\\
&\le\bigl(\tfrac {h}2+1\bigr)^2 e^{(5\delta+(2\delta)^{5/6})\log m}
e^{\frac{(\log\log m)^3}{(6\log C)^3}\log C}e^{-\delta[\frac{\log\log m}{12\log C}]^2\log m}\\
&\le\bigl(\tfrac {h}2+1\bigr)^2 e^{-\delta\cdot\frac{(\log\log m)^2}{288(\log C)^2}\log m}\\
&=\bigl(\tfrac {h}2+1\bigr)^2 e^{-c_2(\log\log m)^2\log m}=\bigl(\tfrac {h}2+1\bigr)^2 m^{-c_2(\log\log m)^2},
\end{align*}
where we set $c_2=\delta/(288(\log C)^2)$. The bound \eqref{Ham-new-app} is proved.
\end{proof}

\bibliographystyle{aop-no-url}
\bibliography{firasbib2010}

\end{document}